\documentclass[a4paper,oneside,11pt]{article}

\usepackage[a4paper]{geometry}
\usepackage{aeguill}                  
\usepackage{graphicx}
\usepackage{amsmath,amsfonts,amssymb,amsthm}
\usepackage{pstricks}
\usepackage{epstopdf}
\usepackage[utf8]{inputenc} 
\usepackage[T1]{fontenc} 
 
\usepackage[english]{babel} 

\title{Hyperbolic graphs for free products, and the Gromov boundary of the graph of cyclic splittings}
\author{Camille Horbez}

\usepackage[all]{xy}

\usepackage{bbm}
\begin{document}
\maketitle
\newtheorem{de}{Definition} [section]
\newtheorem{theo}[de]{Theorem} 
\newtheorem{prop}[de]{Proposition}
\newtheorem{lemma}[de]{Lemma}
\newtheorem{cor}[de]{Corollary}
\newtheorem{propd}[de]{Proposition-Definition}

\theoremstyle{remark}
\newtheorem{rk}[de]{Remark}
\newtheorem{ex}[de]{Example}
\newtheorem{question}[de]{Question}

\normalsize

\addtolength\topmargin{-.5in}
\addtolength\textheight{1.in}
\addtolength\oddsidemargin{-.045\textwidth}
\addtolength\textwidth{.09\textwidth}

\newcommand{\coucou}[1]{\footnote{#1}\marginpar{$\leftarrow$}}

\begin{abstract}
We define hyperbolic analogues of the graphs of free splittings, of cyclic splittings, and of maximally-cyclic splittings of $F_N$ for free products of groups. Given a countable group $G$ which splits as $G=G_1\ast\dots\ast G_k\ast F$, where $F$ denotes a finitely generated free group, we identify the Gromov boundary of the graph of relative cyclic splittings with the space of equivalence classes of $\mathcal{Z}$-averse trees in the boundary of the corresponding outer space. A tree is \emph{$\mathcal{Z}$-averse} if it is not compatible with any tree $T'$, that is itself compatible with a relative cyclic splitting. Two $\mathcal{Z}$-averse trees are \emph{equivalent} if they are both compatible with a common tree in the boundary of the corresponding outer space. We give a similar description of the Gromov boundary of the graph of maximally-cyclic splittings. 
\end{abstract}

\section*{Introduction}

Masur and Minsky's celebrated proof \cite{MM99} of the hyperbolicity of the curve graph $\mathcal{C}(S)$ of a closed, orientable surface $S$, paved the way to many developments in the study of mapping class groups of surfaces. Striking applications of curve complexes include rigidity results for $\text{Mod}(S)$, bounded cohomology of subgroups of $\text{Mod}(S)$ (see \cite{BF02}), finite asymptotic dimension for $\text{Mod}(S)$ (which was established by Bestvina--Bromberg--Fujiwara \cite{BBF14}, using previous work of Bell--Fujiwara who proved in \cite{BF08} that $\mathcal{C}(S)$ has finite asymptotic dimension), or the study of convex cocompact subgroups of $\text{Mod}(S)$, initiated by Farb--Mosher in \cite{FM02}. 

An important result regarding the geometry of $\mathcal{C}(S)$ was a concrete description by Klarreich \cite{Kla99} of its Gromov boundary (see also \cite{Ham06} for an alternative proof of Klarreich's theorem): Klarreich identified $\partial_{\infty}\mathcal{C}(S)$ with the space of equivalence classes of arational measured foliations in $\mathcal{PMF}$, two foliations $F$ and $F'$ being equivalent (denoted by $F\sim F'$) if they have the same topological support, and only differ by their transverse measure. This space is also homeomorphic to the space of ending laminations on $S$. 

The space $\mathcal{PMF}$ is also homeomorphic to Thurston's boundary of the Teichmüller space $\mathcal{T}(S)$. Klarreich also proved that the coarse projection map $\psi$ from the Teichmüller space $\mathcal{T}(S)$ to the curve graph extends continuously to a map $\partial\psi$ from the subspace $\mathcal{AF}\subseteq\mathcal{PMF}$ made of arational foliations to $\partial_{\infty}\mathcal{C}(S)$, and $\partial\psi$ induces a homeomorphism from the quotient space $\mathcal{AF}/{\sim}$ to $\partial_{\infty}\mathcal{C}(S)$. This relates the geometry at infinity of the Teichmüller space and the curve graph of $S$. Connectedness of $\partial_{\infty}\mathcal{C}(S)$ was first established by Leininger--Schleimer \cite{LS09} when $S$ either has genus at least $4$, or is any punctured surface of genus at least $2$, their result was then extended with a different proof by Gabai \cite{Gab09} to all surfaces of complexity at least $2$.

The description of the Gromov boundary of the curve graph turned out to have several important applications. For example, Rafi--Schleimer used the description of $\partial_{\infty}\mathcal{C}(S)$ and its connectedness to prove in \cite{RS11} rigidity results for curve complexes, and deduce in particular that if $S$ and $S'$ are two surfaces such that $\mathcal{C}(S)$ and $\mathcal{C}(S')$ are quasi-isometric, then $S$ and $S'$ are homeomorphic (apart from a few sporadic cases). More remarkably, Klarreich's result was a key ingredient in the proof by Brock--Canary--Minsky of the Ending Lamination Conjecture \cite{Min10,BCM12}, stating that hyperbolic $3$-manifolds with finitely generated fundamental groups are uniquely determined by their topological type and their end invariants. Other applications of the description of $\partial_{\infty}\mathcal{C}(S)$ include proofs of boundary amenability for $\text{Mod}(S)$, from which it follows that $\text{Mod}(S)$ satisfies the Novikov conjecture, see \cite{Ham09,Kid08}.
\\
\\
\indent These successful applications of curve graphs in the study of mapping class groups of surfaces naturally led geometric group theorists to look for $\text{Out}(F_N)$-analogues of the curve graph of a compact surface. Several analogues have been proposed, and proven to be Gromov hyperbolic: among them stand the \emph{free factor graph} $FF_N$ (whose hyperbolicity was proved by Bestvina--Feighn \cite{BF12}), the \emph{free splitting graph} $FS_N$ (whose hyperbolicity was proved by Handel--Mosher \cite{HM12}, and the \emph{cyclic splitting graph} $FZ_N$ (whose hyperbolicity was proved by Mann \cite{Man12}). 

Bestvina and Reynolds \cite{BR13}, and independently Hamenstädt \cite{Ham12}, have identified the Gromov boundary of the free factor graph of $F_N$ with a space of equivalence classes of arational $F_N$-trees in the boundary $\partial CV_N$ of Culler and Vogtmann's outer space. A tree $T\in\partial CV_N$ is \emph{arational} if no proper free factor of $F_N$ acts trivially or with dense orbits on its minimal subtree in $T$, and two arational trees are \emph{equivalent} if their underlying topological trees are homeomorphic. As in the context of mapping class groups of surfaces, the coarse map from $CV_N$ to $FF_N$ extends to a map from the subspace of $\partial CV_N$ made of arational trees to $\partial_{\infty}FF_N$, and this boundary map induces a homeomorphism in the quotient. Hamenstädt also gives a description of the Gromov boundaries of the graphs $FZ_N$ and $FS_N$.   
\\
\\
\indent In the present paper, we will work in a more general context than free groups, and consider generalizations of the graphs $FS_N$ and $FZ_N$, adapted to the study of automorphism groups of free products. These graphs actually appear to be of interest even in the context of free groups, for studying subgroups of $\text{Out}(F_N)$ made of automorphisms that preserve the conjugacy class of a proper free factor of $F_N$.  

Our setting is the following: let $\{G_1,\dots,G_k\}$ be a finite collection of nontrivial countable groups, let $F_N$ be a free group of rank $N$, and let $$G:=G_1\ast\dots\ast G_k\ast F_N.$$ We denote by $\mathcal{F}:=\{[G_1],\dots,[G_k]\}$ the finite collection of the $G$-conjugacy classes of the subgroups $G_1,\dots,G_k$, which we call a \emph{free factor system} of $G$. Subgroups of $G$ that are conjugate into one of the $G_i$'s will be called \emph{peripheral}. We denote by $\mathcal{Z}$ the collection of subgroups of $G$ that are either trivial, or cyclic and nonperipheral. We denote by $\mathcal{Z}^{max}$ the collection of elements in $\mathcal{Z}$ which are either trivial, or \emph{isolated}, i.e. closed under taking roots. A \emph{$(G,\mathcal{F})$-splitting} is a minimal, simplicial $G$-tree, in which all subgroups in $\mathcal{F}$ have a fixed point. It is a \emph{free splitting} (resp. a \emph{$\mathcal{Z}$-splitting}, resp. a \emph{$\mathcal{Z}^{max}$-splitting}) if all edge stabilizers are trivial (resp. belong to $\mathcal{Z}$, resp. belong to $\mathcal{Z}^{max}$). It is a \emph{one-edge splitting} if it contains a single orbit of edges.

We now define generalizations of the aforementioned $\text{Out}(F_N)$-graphs. Vertices of the graph $FS(G,\mathcal{F})$ (resp. $FZ(G,\mathcal{F})$, resp. $FZ^{max}(G,\mathcal{F})$) are the equivalence classes of one-edge $(G,\mathcal{F})$-free splittings (resp. $\mathcal{Z}$-splittings, resp. $\mathcal{Z}^{max}$-splittings). In all cases, two distinct splittings are joined by an edge if they admit a common refinement. All these graphs are hyperbolic, as follows from a direct adaptation of the proofs of the analogous results in context of free groups (see the Appendix of the present paper for a presentation of the argument). There is also a version of Culler--Vogtmann's outer space for free products \cite{GL07}: the \emph{outer space} $P\mathcal{O}(G,\mathcal{F})$ is the space of homothety classes of minimal, simplicial, metric $(G,\mathcal{F})$-splittings, in which nontrivial vertex stabilizers coincide with the subgroups in $\mathcal{F}$, and edge stabilizers are trivial (we also consider the \emph{unprojectivized outer space} $\mathcal{O}(G,\mathcal{F})$, where trees are considered up to equivariant isometry, instead of homothety). We identified in \cite{Hor14-5} the closure $\overline{\mathcal{O}(G,\mathcal{F})}$ with the space of minimal, \emph{very small} $(G,\mathcal{F})$-trees, i.e. trees whose arc stabilizers belong to the class $\mathcal{Z}^{max}$, and whose tripod stabilizers are trivial.
\\
\\
\indent The goal of the present paper is to give concrete models for the Gromov boundaries of $FZ(G,\mathcal{F})$ and $FZ^{max}(G,\mathcal{F})$. Our original motivation for this comes from our work on the Tits alternative for automorphism groups of free products \cite{Hor14-8}. In \cite{Hor14-8}, we show that any subgroup of $\text{Out}(G,\mathcal{F})$ either contains a nonabelian free subgroup generated by two elements acting as loxodromic isometries of $FZ(G,\mathcal{F})$, or else it virtually fixes either a proper $(G,\mathcal{F})$-free factor, or a boundary point in $\partial_{\infty}FZ(G,\mathcal{F})$. In order to deal with this last situation, it is crucial to have a description of the Gromov boundary of $FZ(G,\mathcal{F})$. It is also important in our proof of the Tits alternative to have a boundary map from a subset of the boundary of the relative outer space to $\partial_{\infty}FZ(G,\mathcal{F})$, relating the geometries at infinity of these two spaces: roughly speaking, the lack of local compactness of $FZ(G,\mathcal{F})$ is counterbalanced by the existence of this boundary map relating the geometry at infinity of $FZ(G,\mathcal{F})$ to the geometry of the compact space $\overline{P\mathcal{O}(G,\mathcal{F})}$.

We now present our descriptions of $\partial_{\infty}FZ(G,\mathcal{F})$ and $\partial_{\infty}FZ^{max}(G,\mathcal{F})$. In what follows, we will use the notation $FZ^{(max)}(G,\mathcal{F})$ as a shortcut to denote either $FZ(G,\mathcal{F})$ or $FZ^{max}(G,\mathcal{F})$. We will show that the natural map $\psi^{(max)}: P\mathcal{O}(G,\mathcal{F})\to FZ^{(max)}(G,\mathcal{F})$ admits a continuous extension to the subspace of $\partial P\mathcal{O}(G,\mathcal{F})$ made of what we call \emph{$\mathcal{Z}^{(max)}$-averse trees}, which induces a homeomorphism from a quotient of this subspace to $\partial_{\infty}FZ^{(max)}(G,\mathcal{F})$. Here, the subspace $\mathcal{X}^{(max)}(G,\mathcal{F})\subseteq\overline{\mathcal{O}(G,\mathcal{F})}$ of \emph{$\mathcal{Z}^{(max)}$-averse} trees consists of those trees that are not compatible with any tree in $\overline{\mathcal{O}(G,\mathcal{F})}$ that is itself compatible with some $\mathcal{Z}^{(max)}$-splitting (two $(G,\mathcal{F})$-trees $T$ and $T'$ are \emph{compatible} if there exists a $(G,\mathcal{F})$-tree $\widehat{T}$ which admits alignment-preserving maps onto both $T$ and $T'$). In Theorems \ref{averse} and \ref{maxi-averse} below, we give several equivalent definitions of $\mathcal{Z}^{(max)}$-averse trees. In particular, it is proved that if two trees are compatible, and one is $\mathcal{Z}^{(max)}$-averse, then so is the other. We show that being compatible with a common tree in $\overline{\mathcal{O}(G,\mathcal{F})}$ defines an equivalence relation on $\mathcal{X}^{(max)}(G,\mathcal{F})$, which we denote by $\sim$. The main theorem of the present paper is the following.

\begin{theo}\label{intro-main}
Let $G$ be a countable group, and let $\mathcal{F}$ be a free factor system of $G$. Then there exists a unique $\text{Out}(G,\mathcal{F})$-equivariant homeomorphism $$\partial\psi:\mathcal{X}(G,\mathcal{F})/{\sim}\to \partial_{\infty} FZ(G,\mathcal{F}),$$ so that for all $T\in\mathcal{X}(G,\mathcal{F})$, and all sequences $(T_n)_{n\in\mathbb{N}}\in \mathcal{O}(G,\mathcal{F})^{\mathbb{N}}$ converging to $T$, the sequence $(\psi(T_n))_{n\in\mathbb{N}}$ converges to $\partial\psi (T)$.
\end{theo}

When taking the quotient by equivariant homotheties, the relation $\sim$ induces an equivalence relation on $P\mathcal{X}(G,\mathcal{F})$, and $\partial_{\infty}FZ(G,\mathcal{F})$ is also homeomorphic to the quotient $P\mathcal{X}(G,\mathcal{F})/\sim$. The analogous statements also hold true for the Gromov boundary $\partial_{\infty}FZ^{max}(G,\mathcal{F})$, with $\mathcal{X}^{max}(G,\mathcal{F})$ instead of $\mathcal{X}(G,\mathcal{F})$. 
\\
\\
\indent We also provide information on the fibers of the equivalence relation $\sim$: every $\sim$-class of $\mathcal{Z}^{(max)}$-averse trees contains \emph{mixing} representatives (in the sense of Morgan \cite{Mor88}), and any two such representatives belong to the same simplex of length measures in $\partial P\mathcal{O}(G,\mathcal{F})$. We refer to Propositions \ref{mixing-representative} and \ref{mixing-rep-max} for precise statements. As was pointed to us by the referee, this implies that the fibers in $\overline{P\mathcal{O}(G,\mathcal{F})}$ of the boundary map $\partial\psi$ are compact and contractible, and hence that the boundary map $\partial\psi$ is cell-like. Since there is a bound on the topological dimension of $\partial P\mathcal{O}(G,\mathcal{F})$ \cite{Hor14-5}, this yields a bound on the cohomological dimension of $\partial_{\infty}FZ^{(max)}(G,\mathcal{F})$ (but not on its covering dimension, however).
\\
\\
\indent We now mention other applications of Theorem \ref{intro-main} to the geometry of $FZ^{(max)}(G,\mathcal{F})$ and its relations with other complexes. A first consequence of Theorem \ref{intro-main} is unboundedness of $FZ^{(max)}(G,\mathcal{F})$, except in the two sporadic cases where either $G=G_1\ast G_2$ and $\mathcal{F}=\{[G_1],[G_2]\}$, or $G=G_1\ast\mathbb{Z}$ and $\mathcal{F}=\{[G_1]\}$. For these two sporadic cases, an explicit description of the graphs $FS(G,\mathcal{F})$, $FZ(G,\mathcal{F})$ and $FZ^{max}(G,\mathcal{F})$ is given in Remarks \ref{sporadic-fs} and \ref{sporadic}.

Another application is the fact that the inclusion map from $FZ^{max}(G,\mathcal{F})$ to $FZ(G,\mathcal{F})$ is not a quasi-isometry as soon as the free rank $N$ of the decomposition of $G$ is at least $1$, and its Kurosh rank $k+N$ is at least $3$. However, the two graphs are quasi-isometric when $N=0$. We refer to Section \ref{sec-ex-max} for details.

In the context of free groups, there is also a natural map from $FZ_N^{(max)}$ to the free factor graph $FF_N$, defined by mapping any free splitting to one of its vertex groups, and mapping any one-edge $\mathcal{Z}^{(max)}$-splitting with nontrivial edge group to the smallest free factor of $F_N$ containing the edge group. We deduce from Theorem \ref{intro-main} that this map is not a quasi-isometry as soon as $N\ge 3$.
\\
\\
\indent We finally say a word of our proof of Theorem \ref{intro-main}. For simplicity, we will only describe the boundary $\partial_{\infty} FZ^{max}(G,\mathcal{F})$. Our general strategy of proof follows Klarreich's. 

The first step of the proof consists in showing that for all $\mathcal{Z}^{max}$-averse trees $T$, and all sequences of trees $(T_n)_{n\in\mathbb{N}}\in\mathcal{O}(G,\mathcal{F})^{\mathbb{N}}$ that converge to $T$, the sequence $(\psi^{max}(T_n))_{n\in\mathbb{N}}$ is unbounded in $FZ^{max}(G,\mathcal{F})$, and actually converges to a point in the Gromov boundary (Proposition \ref{X-psi}). This will define the boundary map $\partial\psi^{max}$ from $\mathcal{X}^{max}(G,\mathcal{F})$ to $\partial_{\infty}FZ^{max}(G,\mathcal{F})$. Our proof of this fact follows an argument of Kobayashi \cite{Kob88} for showing unboundedness of the curve complex of a nonsporadic surface, and relies on the following characterization of $\mathcal{Z}^{max}$-averse trees: a tree $T\in\overline{\mathcal{O}(G,\mathcal{F})}$ is $\mathcal{Z}^{max}$-averse if and only if there is no finite sequence $(T=T_0,\dots,T_k)$ of trees in $\overline{\mathcal{O}(G,\mathcal{F})}$ such that $T_i$ is compatible with $T_{i+1}$ for all $i\in\{0,\dots,k-1\}$, and $T_k$ is simplicial. We then show that the map $\partial\psi^{max}$ induces a map from the quotient space $\mathcal{X}^{max}(G,\mathcal{F})/{\sim}$ to $\partial_{\infty}FZ^{max}(G,\mathcal{F})$, which is shown to be both continuous and injective. 

The second step consists in showing surjectivity of $\partial\psi^{max}$. To this end, we associate to every tree $T\in\overline{\mathcal{O}(G,\mathcal{F})}$ which is not $\mathcal{Z}^{max}$-averse a set of \emph{reducing splittings}. These are defined as those $\mathcal{Z}^{max}$-splittings $S$ such that there exists a tree $T'\in\overline{\mathcal{O}(G,\mathcal{F})}$ that is compatible with both $T$ and $S$. We prove in Section \ref{sec-not-X} that the set of reducing splittings of $T$ has bounded diameter in $FZ^{max}(G,\mathcal{F})$: the proof relies on a study of folding paths in $\overline{\mathcal{O}(G,\mathcal{F})}$, which requires us to extend in Section \ref{sec-coll-pull} some aspects of folding path theory to the context of nonsimplicial trees. Boundedness of the set of reducing splittings is used to prove that the projection of any sequence of trees in $\mathcal{O}(G,\mathcal{F})$ that converges to a tree which is not $\mathcal{Z}^{max}$-averse, does not converge to any point of the Gromov boundary $\partial_{\infty}FZ^{max}(G,\mathcal{F})$ (Proposition \ref{psi-not-infty}). The proof of surjectivity of $\partial\psi^{max}$ then goes as follows: given $\xi\in\partial_{\infty}FZ^{max}(G,\mathcal{F})$, there exists a sequence $(T_n)_{n\in\mathbb{N}}\in\mathcal{O}(G,\mathcal{F})^{\mathbb{N}}$ such that $(\psi^{max}(T_n))_{n\in\mathbb{N}}$ converges to $\xi$. It then follows from the above that all projective accumulation points of the sequence $(T_n)_{n\in\mathbb{N}}$ belong to a common equivalence class of $\mathcal{Z}^{max}$-averse trees, and given any tree $T$ in this class, we have $\xi=\partial\psi^{max}(T)$. Similar arguments also show that $\partial\psi^{max}$ is closed, and hence it is a homeomorphism from $\mathcal{X}^{max}/{\sim}$ to $\partial_{\infty}FZ^{max}(G,\mathcal{F})$.

\paragraph*{Structure of the paper.}~
\\
\\
The paper is organized as follows. In Section \ref{sec-background}, we review some basic facts about free products of groups, outer spaces, folding paths, and Gromov hyperbolic spaces. We then introduce in Section \ref{sec-FS} the graphs of free splittings and of $\mathcal{Z}^{(max)}$-splittings (a sketch of the proof of their hyperbolicity, following closely Bestvina--Feighn's and Mann's arguments, is given in the Appendix). The rest of the paper is devoted to the description of the Gromov boundary $\partial_{\infty} FZ^{(max)}(G,\mathcal{F})$. In Section \ref{sec-material}, we introduce some more material about the geometry of folding paths in $\mathcal{O}(G,\mathcal{F})$, and the description of trees in $\overline{\mathcal{O}(G,\mathcal{F})}$. All this material will be used in the proof of our main theorem. In Section \ref{sec-X}, we study the properties of $\mathcal{Z}^{(max)}$-averse trees, and explain why these trees lie in some sense \emph{at infinity} of the complex $FZ^{(max)}(G,\mathcal{F})$. We then describe constructions of collapses and pullbacks of folding sequences in Section \ref{sec-coll-pull}, which are then used in Section \ref{sec-not-X} to associate to every tree in $\overline{\mathcal{O}(G,\mathcal{F})}\smallsetminus\mathcal{X}^{(max)}(G,\mathcal{F})$ a bounded set of \emph{reducing splittings} in $FZ^{(max)}(G,\mathcal{F})$. We eventually complete the proof of Theorem \ref{intro-main} in Section \ref{sec-boundary}. 

\section*{Acknowledgments}

It is a pleasure to thank my advisor Vincent Guirardel. The present paper would never have been written without the many hours he spent explaining me a lot of notions and ideas that turned out to be crucial for this work. I am also very grateful to the anonymous referee for a careful reading of the paper, and many valuable suggestions for improving its readibility. I am particularly indepted to the referee for pointing out an inaccuracy in the construction of collapses of optimal folding paths in a previous version of the manuscript. I also owe to the referee the suggestion of using the description of the fibers of the boundary map to deduce a bound on the cohomological dimension of $\partial_{\infty}FZ^{(max)}(G,\mathcal{F})$. I acknowledge support from ANR-11-BS01-013 and from the Lebesgue Center of Mathematics.
\setcounter{tocdepth}{1}
\tableofcontents

\section{Background}\label{sec-background}

We start by collecting general facts about free products of groups, outer spaces, $\mathbb{R}$-trees and Gromov hyperbolic spaces.

\subsection{Free products and free factors}\label{sec-free-product}

Let $G$ be a countable group which splits as a free product of the form $$G=G_1\ast\dots \ast G_k\ast F,$$ where $F$ is a finitely generated free group. We let $\mathcal{F}:=\{[G_1],\dots,[G_k]\}$ be the finite collection of the $G$-conjugacy classes of the $G_i$'s, which we call a \emph{free factor system} of $G$. The rank of the free group $F$ arising in such a splitting only depends on $\mathcal{F}$. We call it the \emph{free rank} of $(G,\mathcal{F})$ and denote it by $\text{rk}_f(G,\mathcal{F})$. The \emph{Kurosh rank} of $(G,\mathcal{F})$ is defined as $\text{rk}_K(G,\mathcal{F}):=\text{rk}_f(G,\mathcal{F})+|\mathcal{F}|$. Subgroups of $G$ which are conjugate into some subgroup in $\mathcal{F}$ are called \emph{peripheral} subgroups. A \emph{$(G,\mathcal{F})$-free splitting} is a minimal simplicial $G$-tree in which all subgroups in $\mathcal{F}$ are elliptic, and all of whose edge stabilizers are trivial. A \emph{$(G,\mathcal{F})$-free factor} is a subgroup of $G$ which is a point stabilizer in some $(G,\mathcal{F})$-free splitting. A $(G,\mathcal{F})$-free factor is \emph{proper} if it is nonperipheral (and in particular nontrivial), and not equal to $G$.

\begin{figure}
\begin{center}
\input{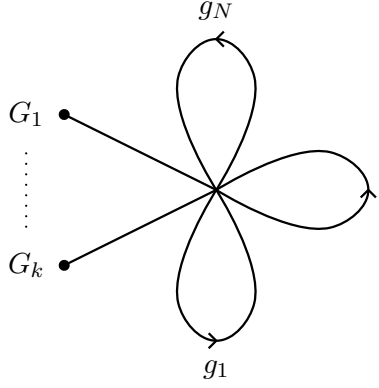}
\caption{A standard $(G,\mathcal{F})$-free splitting.}
\label{fig-Grushko}
\end{center}
\end{figure}

Subgroups of free products were studied by Kurosh in \cite{Kur34}, we will follow the geometric approach by Scott--Wall \cite[Theorem 3.1]{SW79} in the presentation below. Let $H$ be a subgroup of $G$. Let $T$ be the Bass--Serre tree of the decomposition of $G$ as a graph of groups represented in Figure \ref{fig-Grushko} (on which $\{g_1,\dots,g_N\}$ denotes a free basis of $F$). By considering the $H$-minimal subtree in $T$, we get the existence of a (possibly infinite) set $J$, together with an integer $i_j\in\{1,\dots,k\}$, a nontrivial subgroup $H_{j}\subseteq G_{i_j}$ and an element $g_{j}\in G$ for each $j\in J$, and a (not necessarily finitely generated) free subgroup $F'\subseteq G$, so that $$H=\ast_{j\in J}~ g_{j}H_{j}g_{j}^{-1}\ast F'.$$ This is called a \emph{Kurosh decomposition} of $H$. The \emph{Kurosh rank} of $H$ is defined as $\text{rk}_K(H):=|J|+\text{rk}(F')$ (this does not depend on a Kurosh decomposition of $H$). We note that $\text{rk}_K(H)$ may be infinite in general. We let $\mathcal{F}_H$ be the collection of all $H$-conjugacy classes of the subgroups $g_{j}H_{j}g_{j}^{-1}$. If $H$ is a $(G,\mathcal{F})$-free factor, then for all $j\in J$, we have $H_{j}=G_{i_j}$ by definition. In addition, the integers $i_j$ are pairwise distinct, because two distinct $G$-conjugates of the subgroup $G_{i_j}$ cannot have a common fixed point in a splitting of $G$ without fixing an arc in the splitting. In this case, we also get that $F'$ is finitely generated. So any $(G,\mathcal{F})$-free factor has finite Kurosh rank, smaller than $\text{rk}_K(G,\mathcal{F})$. 

We denote by $\text{Out}(G,\mathcal{F})$ the subgroup of the outer automorphism group $\text{Out}(G)$ of $G$ made of those automorphisms which preserve all conjugacy classes in $\mathcal{F}$. 

We denote by $\mathcal{Z}$ the collection of all subgroups of $G$ that are either trivial, or cyclic and nonperipheral. We denote by $\mathcal{Z}^{max}$ the collection of subgroups in $\mathcal{Z}$ that are either trivial, or \emph{isolated}, i.e. closed under taking roots.

\subsection{Outer space and its closure}

An \emph{$\mathbb{R}$-tree} is a metric space $(T,d_T)$ in which any two points $x,y\in T$ are joined by a unique arc, which is isometric to a segment of length $d_T(x,y)$. A \emph{$(G,\mathcal{F})$-tree} is an $\mathbb{R}$-tree equipped with an isometric action of $G$, in which all peripheral subgroups are elliptic. A \emph{Grushko $(G,\mathcal{F})$-tree} is a metric simplicial minimal $(G,\mathcal{F})$-tree with trivial arc stabilizers, whose vertex stabilizers coincide with the subgroups in $G$ whose $G$-conjugacy class belongs to $\mathcal{F}$. Two $(G,\mathcal{F})$-trees are \emph{equivalent} if there exists a $G$-equivariant isometry between them. The \emph{unprojectivized outer space} $\mathcal{O}(G,\mathcal{F})$, introduced by Guirardel and Levitt in \cite{GL07}, is the space of all equivalence classes of Grushko $(G,\mathcal{F})$-trees. \emph{Outer space} $P\mathcal{O}(G,\mathcal{F})$ is defined as the space of homothety classes of trees in $\mathcal{O}(G,\mathcal{F})$. The group $\text{Out}(G,\mathcal{F})$ acts on both $\mathcal{O}(G,\mathcal{F})$ and $P\mathcal{O}(G,\mathcal{F})$ on the right, by precomposing the actions. Let $T\in \mathcal{O}(G,\mathcal{F})$. For all $g\in G$, the \emph{translation length} of $g$ in $T$ is defined to be
\begin{displaymath}
||g||_T:=\inf_{x\in T}d_T(x,gx).
\end{displaymath}

\noindent Culler and Morgan have shown in \cite[Theorem 3.7]{CM87} that the function
\begin{displaymath}
\begin{array}{cccc}
i:&\mathcal{O}(G,\mathcal{F})&\to &\mathbb{R}^{G}\\
&T&\mapsto & (||g||_T)_{g\in G}
\end{array}
\end{displaymath}

\noindent is injective. We equip $\mathcal{O}(G,\mathcal{F})$ with the topology induced by this embedding, which is called the \emph{axes topology}. Taking the quotient by $G$-equivariant homotheties yields an embedding of $P\mathcal{O}(G,\mathcal{F})$ into the projective space $\mathbb{PR}^{G}$, whose image has compact closure $\overline{P\mathcal{O}(G,\mathcal{F})}$, see \cite[Theorem 4.5]{CM87} and \cite[Proposition 1.2]{Hor14-5}. We denote by $\overline{\mathcal{O}(G,\mathcal{F})}$ the preimage of $\overline{P\mathcal{O}(G,\mathcal{F})}$ in $\mathbb{R}^{G}$. A $(G,\mathcal{F})$-tree $T$ is \emph{very small} if arc stabilizers in $T$ belong to the class $\mathcal{Z}^{max}$, and tripod stabilizers are trivial. We identified the space $\overline{\mathcal{O}(G,\mathcal{F})}$ with the space of nontrivial, minimal, very small $(G,\mathcal{F})$-trees \cite[Theorem 0.1]{Hor14-5}. 

Given a tree $T\in\mathcal{O}(G,\mathcal{F})$, the \emph{cone} of $T$ is the set of trees in $\mathcal{O}(G,\mathcal{F})$ obtained by equivariantly varying the lengths of the edges of $T$ (we allow some edges to have length $0$ as long as the tree obtained by collapsing the corresponding edges still belongs to $\mathcal{O}(G,\mathcal{F})$.

\subsection{Liberal folding paths}\label{sec-folding}

From now on, all maps between $G$-trees will be $G$-equivariant. Let $T$ and $T'$ be two $\mathbb{R}$-trees. A \emph{direction} at a point $x\in T$ is a connected component of $T\smallsetminus\{x\}$. A \emph{train track structure} on $T$ is a partition of the set of directions at each point $x\in T$. Elements of the partition are called \emph{gates} at $x$. A pair $(d,d')$ of directions at $x$ is \emph{legal} if $d$ and $d'$ do not belong to the same gate. A path in $T$ is \emph{legal} if it only crosses legal pairs of directions. A \emph{morphism} $f:T\to T'$ is a map such that every segment of $T$ can be subdivided into finitely many subsegments, in restriction to which $f$ is an isometry. Any morphism $f:T\to T'$ defines a train track structure on $T$, two directions $d,d'$ at a point $x\in T$ being in the same gate if there are nondegenerate intervals $(x,a]\subseteq d$ and $(x,b]\subseteq d'$ which have the same $f$-images. A morphism is \emph{optimal} if there are at least two gates at every point in $T$. Let $T$ and $T'$ be two $(G,\mathcal{F})$-trees, and $f:T\to T'$ be a morphism. Following \cite[Appendix A.1]{BF13}, we define a \emph{liberal folding path guided by $f$} to be a continuous family $(T_t)_{t\in\mathbb{R}_+}$, together with a collection of morphisms $f_{t_1,t_2}:T_{t_1}\to T_{t_2}$ for all $0\le t_1<t_2$, such that 

\begin{itemize}
\item there exists $L\in\mathbb{R}$ such that for all $t\ge L$, we have $T_t=T'$, and
\item we have $f_{0,L}=f$, and 
\item for all $0\le t_1<t_2<t_3$, we have $f_{t_1,t_3}=f_{t_2,t_3}\circ f_{t_1,t_2}$.
\end{itemize}

\noindent Given two $(G,\mathcal{F})$-trees $T$ and $T'$, a \emph{liberal folding path} from $T$ to $T'$ is a folding path guided by some morphism $f:T\to T'$. A liberal folding path guided by a morphism $f$ is \emph{optimal} if $f$ is optimal. Notice that in this case, all morphisms $f_{t,t'}$ with $t<t'$ are also optimal. 

Given two $(G,\mathcal{F})$-trees $T$ and $T'$, we say that $T'$ is obtained from $T$ by a \emph{fold} if there exist arcs $e$ and $e'$ in $T$ having a common endpoint, such that $T'$ is obtained from $T$ by $G$-equivariantly identifying $e$ with $e'$. Given a morphism $f:T\to T'$, a way of constructing liberal folding paths from $T$ to $T'$ is by folding pairs of directions that have the same $f$-image. We refer to \cite[Section 2]{BF12} and references therein for various constructions of liberal folding paths between $(G,\mathcal{F})$-trees.

Given two $(G,\mathcal{F})$-trees $T$ and $T'$, and an optimal morphism $f:T\to T'$, Guirardel and Levitt described in \cite[Section 3]{GL07} a construction of a \emph{canonical optimal folding path} $(T_t)_{t\in\mathbb{R}_+}$ guided by $f$. As a set, the tree $T_t$ is defined in the following way. Given $a,b\in T$, we define the \emph{identification time} between $a$ and $b$ as $\tau(a,b):=\sup_{x\in [a,b]}d_{T'}(f(a),f(x))$. We define equivalence relations $\sim_t$ on $T$ for all $t\in\mathbb{R}_+$ by letting $a\sim_t b$ if $f(a)=f(b)$ and $\tau(a,b)\le t$. The tree $T_t$ is the quotient of $T$ by the equivalence relation $\sim_t$. In \cite[Section 3.1]{GL07}, Guirardel and Levitt defined for all $t\in\mathbb{R}_+$ a metric on $T_t$ that turns it into an $\mathbb{R}$-tree.

We now establish one more property of liberal folding paths.

\begin{prop}\label{folding-simplicial}
Let $T$ and $T'$ be simplicial $(G,\mathcal{F})$-trees with trivial edge stabilizers, and let $(T_t)_{t\in [0,L]}$ be a liberal folding path from $T$ to $T'$. Then for all $t\in [0,L]$, the tree $T_t$ is simplicial and has trivial edge stabilizers.
\end{prop}

\begin{proof}
Arc stabilizers in $T_t$ are trivial, otherwise the $f_{t,L}$-image of an arc with nontrivial stabilizer in $T_t$ would be an arc with nontrivial stabilizer in $T'$. Therefore, the tree $T_t$ splits as a graph of actions, whose vertex actions have dense orbits for their stabilizers (see Proposition \ref{Levitt} below). The morphism $f_{s,L}$ is an isometry in restriction to the vertex trees of this decomposition (see Corollary \ref{alignment-preserved} below). As $T'$ is simplicial, this implies that $T_t$ is simplicial.
\end{proof}

We will also work with the following discrete version of optimal liberal folding paths. Let $T$ be a $(G,\mathcal{F})$-tree. A \emph{folding sequence ending at $T$} is a sequence $(T_p)_{p\in\mathbb{N}}$ of $(G,\mathcal{F})$-trees that converges to $T$ in a nonstationary way, such that for all integers $p<q$, there are morphisms $f_p:T_p\to T$ and $f_{p,q}:T_p\to T_{q}$ such that $f_p=f_q\circ f_{p,q}$ for all $p<q$, and $f_{p,r}=f_{q,r}\circ f_{p,q}$ for all $p<q<r$.

\subsection{Coarse geometry notions}

We now recall the notions of quasi-isometries between metric spaces, and of (reparameterized) quasi-geodesics. 

Two metric spaces $(X,d)$ and $(X',d')$ are \emph{quasi-isometric} if there exist $K,L\ge 0$ and a map $f:X\to X'$ such that 
\begin{itemize}
\item for all $x'\in X'$, there exists $x\in X$ such that $d'(x',f(x))\le L$, and
\item for all $x,y\in X$, we have $$\frac{1}{K}d(x,y)-L\le d'(f(x),f(y))\le K d(x,y)+L.$$
\end{itemize}

Let $(X,d)$ be a metric space, let $x,y\in X$, and let $K,L\ge 0$. A \emph{$(K,L)$-quasi-geodesic} from $x$ to $y$ is a map $\gamma:[a,b]\to X$, where $[a,b]\subseteq\mathbb{R}$ is a segment, such that $\gamma(a)=x$ and $\gamma(b)=y$, and for all $s,t\in[a,b]$, we have $$\frac{1}{K}|t-s|-L\le d(\gamma(s),\gamma(t))\le K |t-s|+L.$$ A \emph{reparameterized quasi-geodesic} is a map $\gamma':[a',b']\to X$, where $[a',b']\subseteq\mathbb{R}$ is a segment, so that there exists a segment $[a,b]\subseteq\mathbb{R}$ and a continuous nondecreasing map $\theta:[a,b]\to [a',b']$, such that $\gamma'\circ\theta$ is a $(K,L)$-quasi-geodesic. 

\subsection{Gromov hyperbolic spaces}

We give a very brief account on hyperbolic spaces, which were defined by Gromov \cite{Gro87}. The reader is referred to \cite{BH99,CDP90,GdlH90} for a detailed introduction.  Let $(X,d)$ be a metric space. Let $p\in X$ be some basepoint. For all $x,y\in X$, the \emph{Gromov product} of $x$ and $y$ with respect to $p$ is defined as
\begin{displaymath}
(x|y)_p:=\frac{1}{2}(d(p,x)+d(p,y)-d(x,y)).
\end{displaymath}

A metric space $X$ is \emph{Gromov hyperbolic} if there exists a constant $\delta>0$ such that for all $x,y,z,p\in X$, we have
\begin{displaymath}
(x|y)_p\ge\min\{(x|z)_p,(y|z)_p\}-\delta.
\end{displaymath}

\noindent (When $X$ is geodesic, hyperbolicity of $X$ is equivalent to a \emph{thin triangles} condition, see the aforementioned references). If $(X,d)$ is Gromov hyperbolic, we say that a sequence $(x_n)_{n\in\mathbb{N}}\in X^{\mathbb{N}}$ \emph{converges to infinity} if the Gromov product $(x_n|x_m)_p$ goes to $+\infty$ as $n$ and $m$ both go to $+\infty$. Two sequences $(x_n)_{n\in\mathbb{N}}$ and $(y_n)_{n\in\mathbb{N}}$ that both converge to infinity are \emph{equivalent} if the Gromov product $(x_n|y_m)_p$ goes to $+\infty$ as $n$ and $m$ go to $+\infty$. It follows from the hyperbolicity of $(X,d)$ that this is indeed an equivalence relation. The \emph{Gromov boundary} $\partial_{\infty} X$ of $X$ is defined to be the collection of equivalence classes of sequences that converge to infinity. For all $a,b\in\partial_{\infty} X$, the \emph{Gromov product} of $a$ and $b$ with respect to $p$ is defined as
\begin{displaymath}
(a|b)_p=\sup\liminf_{i,j\to +\infty}(x_i|y_j)_p,
\end{displaymath}

\noindent where the supremum is taken over all sequences $(x_i)_{i\in\mathbb{N}}$ converging to $a$ and all sequences $(y_j)_{j\in\mathbb{N}}$ converging to $b$. The set $\partial_{\infty} X$ is equipped with the topology for which every point $a\in\partial_{\infty} X$ has a basis of open neighborhoods made of the sets of the form $N_r(a):=\{b\in\partial_{\infty} X|(a|b)_p\ge r\}$. One can also define the Gromov product between an element in $X$ and an element in $\partial_{\infty} X$ similarly, and get a topology on $X\cup\partial_{\infty} X$. Given any $\xi\in\partial_{\infty} X$, there exists a quasi-geodesic ray $\tau:\mathbb{R}_+\to X$ such that $\tau(t)$ converges to $\xi$ as $t$ goes to $+\infty$.

\section{The graphs $FS(G,\mathcal{F})$ and $FZ^{(max)}(G,\mathcal{F})$} \label{sec-FS}

\subsection{The free splitting graph}

We recall that a \emph{$(G,\mathcal{F})$-free splitting} is a minimal, simplicial $(G,\mathcal{F})$-tree, all of whose edge stabilizers are trivial. Two $(G,\mathcal{F})$-free splittings are \emph{equivalent} if they are $G$-equivariantly homeomorphic. Given two $(G,\mathcal{F})$-free splittings $T$ and $T'$, we say that $T'$ is a \emph{refinement} of $T$ if $T$ is obtained from $T'$ by collapsing a $G$-invariant set of edges of $T$ to points. Two $(G,\mathcal{F})$-free splittings are \emph{compatible} if they have a common refinement. The \emph{free splitting graph} $FS(G,\mathcal{F})$ is the graph whose vertices are the equivalence classes of one-edge $(G,\mathcal{F})$-free splittings, two distinct splittings being joined by an edge if they are compatible. Alternatively, one can define $FS(G,\mathcal{F})$ to be the graph whose vertices are the equivalence classes of all $(G,\mathcal{F})$-free splittings, two splittings being joined by an edge if one properly refines the other. The two versions of the complex are quasi-isometric to each other. 

There is map $\phi:{\mathcal{O}(G,\mathcal{F})}\to FS(G,\mathcal{F})$, which extends to the set of simplicial trees in $\overline{\mathcal{O}(G,\mathcal{F})}$ with trivial edge stabilizers, defined by choosing a one-edge collapse of every simplicial tree in $\overline{\mathcal{O}(G,\mathcal{F})}$ (as we have to make choices, this map is not equivariant, but making any other choice can only change distances by at most $2$). Proposition \ref{folding-simplicial} implies that the $\phi$-image of any folding path between simplicial trees in $\overline{\mathcal{O}(G,\mathcal{F})}$ with trivial edge stabilizers is well-defined. The graph $FS(G,\mathcal{F})$ comes with a right action of $\text{Out}(G,\mathcal{F})$, by precomposition of the actions.

Hyperbolicity of the free splitting graph of a finitely generated free group was shown by Handel and Mosher \cite{HM12}, whose proof involves studying folding paths between simplicial $F_N$-trees with trivial edge stabilizers. An alternative proof in the sphere model of the free splitting graph, based on the study of surgery paths, was given by Hilion and Horbez \cite{HH13}; Bestvina and Feighn also gave a simplified proof in Handel and Mosher's setting \cite[Appendix]{BF13}. Handel--Mosher's proof, as interpreted by Bestvina--Feighn, adapts with almost no change to yield hyperbolicity of $FS(G,\mathcal{F})$. We will give a sketch of the argument, following the exposition from \cite[Appendix]{BF13} very closely, in the appendix of the present paper for completeness of the exposition. 

\begin{theo}\label{FS-hyperbolic}
Let $G$ be a countable group, and let $\mathcal{F}$ be a free factor system of $G$. Then the graph $FS(G,\mathcal{F})$ is Gromov hyperbolic. There exist $K,L>0$ such that all $\phi$-images in $FS(G,\mathcal{F})$ of optimal liberal folding paths between simplicial trees in $\overline{\mathcal{O}(G,\mathcal{F})}$ with trivial edge stabilizers are $(K,L)$-reparameterized quasi-geodesics.
\end{theo}

\begin{rk}\label{sporadic-fs}
If $G=G_1\ast G_2$ and $\mathcal{F}=\{[G_1],[G_2]\}$, then $FS(G,\mathcal{F})$ is reduced to a point. If  $G=G_1\ast\mathbb{Z}$ and $\mathcal{F}=\{[G_1]\}$, then $FS(G,\mathcal{F})$ is a star of diameter equal to $2$, whose central vertex corresponds to the HNN extension $G=G_1\ast$, and whose extremal vertices correspond to all splittings of the form $G=G_1\ast\langle g_1t\rangle$, where $t$ is a stable letter of the HNN extension, and $g_1$ varies in $G_1$.
\end{rk}
 
\subsection{The $\mathcal{Z}$-splitting graph and the $\mathcal{Z}^{max}$-splitting graph} 
 
We recall from Section \ref{sec-free-product} that $\mathcal{Z}^{(max)}$ is the collection of all subgroups of $G$ that are either trivial, or (maximally-)cyclic and nonperipheral. A \emph{$\mathcal{Z}^{(max)}$-splitting} is a minimal, simplicial (hence cocompact) $(G,\mathcal{F})$-tree, all of whose edge stabilizers belong to the collection $\mathcal{Z}^{(max)}$. The \emph{graph of $\mathcal{Z}^{(max)}$-splittings} $FZ^{(max)}(G,\mathcal{F})$ is the graph whose vertices are the equivalence classes of one-edge $\mathcal{Z}^{(max)}$-splittings, two distinct vertices being joined by an edge if the corresponding splittings are compatible (note that if two $\mathcal{Z}^{(max)}$-splittings have a common refinement, then they have a common refinement which is a $\mathcal{Z}^{(max)}$-splitting). Again, there are natural maps $\psi^{(max)}:\mathcal{O}(G,\mathcal{F})\to FZ^{(max)}(G,\mathcal{F})$, which extend to the set of trees in $\overline{\mathcal{O}(G,\mathcal{F})}$ having a nontrivial simplicial part. The graphs $FZ(G,\mathcal{F})$ and $FZ^{max}(G,\mathcal{F})$ both come with right $\text{Out}(G,\mathcal{F})$-actions, given by precomposition of the actions. 

In the case where $G$ is a finitely generated free group and $\mathcal{F}=\emptyset$, hyperbolicity of the graph of $\mathcal{Z}$-splittings was proved by Mann \cite{Man12}. Mann's proof actually adapts to the relative setting, and also to the case of  $\mathcal{Z}^{max}$-splittings. Again, we will sketch a proof of the following theorem in the appendix, following Mann's proof very closely.

\begin{theo}\label{FZ-hyperbolic}
Let $G$ be a countable group, and let $\mathcal{F}$ be a free factor system of $G$. Then $FZ(G,\mathcal{F})$ and $FZ^{max}(G,\mathcal{F})$ are Gromov hyperbolic. There exists $K>0$ (only depending on the Kurosh rank of $(G,\mathcal{F})$) such that images in $FZ(G,\mathcal{F})$ and in $FZ^{max}(G,\mathcal{F})$ of optimal liberal folding paths between simplicial trees in $\overline{\mathcal{O}(G,\mathcal{F})}$ with trivial edge stabilizers are $K$-Hausdorff close to any geodesic segment joining their endpoints.
\end{theo}

\begin{rk}\label{sporadic}
When $G=G_1\ast G_2$ and $\mathcal{F}=\{[G_1],[G_2]\}$, all graphs $FZ(G,\mathcal{F})$, $FZ^{max}(G,\mathcal{F})$ and $FS(G,\mathcal{F})$ are equal, and reduced to a point. When $G=G_1\ast\mathbb{Z}$ and $\mathcal{F}=\{[G_1]\}$, the graphs $FZ^{max}(G,\mathcal{F})$ and $FS(G,\mathcal{F})$ are equal: all one-edge $\mathcal{Z}^{max}$-splittings are free splittings. The graph $FZ(G,\mathcal{F})$ is also star-shaped: the splitting $G=G_1\ast$ is the central vertex, it is joined by an edge to all free splittings of the form $G=G_1\ast\langle g_1t\rangle$, where $t$ denotes the stable letter of the HNN extension, and $g_1$ varies in $G_1$. These free splittings are also joined by edges to one-edge $\mathcal{Z}$-splittings of the form $G=(G_1\ast\langle (g_1t)^k\rangle)\ast_{\langle(g_1t)^k\rangle}\langle g_1t\rangle$, with $k\ge 2$. Therefore $FZ(G,\mathcal{F})$ is bounded, of diameter $4$.
\end{rk}

\section{More material}\label{sec-material}

The following sections of the paper aim at describing the Gromov boundaries of the graphs $FZ(G,\mathcal{F})$ and $FZ^{max}(G,\mathcal{F})$. We first introduce more background material that will be used in the proof of our main theorem. 
\subsection{Tame $(G,\mathcal{F})$-trees}\label{sec-good}

In this section, we review the definition and the properties of the class of tame $(G,\mathcal{F})$-trees, introduced in \cite[Section 6]{Hor14-5}, in which we will carry out most of our arguments.

\paragraph*{Definition and properties.}

Let $T$ be a $(G,\mathcal{F})$-tree. A point $x\in T$ is a \emph{branch point} if $T\smallsetminus\{x\}$ has at least three connected components. It is an \emph{inversion point} if $T\smallsetminus\{x\}$ has exactly two connected components, and there exists $g\in G$ that exchanges these two components.

\begin{de}
A minimal $(G,\mathcal{F})$-tree is \emph{small} if its arc stabilizers belong to the class $\mathcal{Z}$. It is \emph{tame} if in addition, it has finitely many orbits of directions at branch or inversion points. A \emph{$\mathcal{Z}^{max}$-tame tree} is a tame $(G,\mathcal{F})$-tree whose arc stabilizers belong to the class $\mathcal{Z}^{max}$.
\end{de}

Tame $(G,\mathcal{F})$-trees also have the following alternative description. For all $k\in\mathbb{N}$, we say that a small $(G,\mathcal{F})$-tree $T$ is \emph{$k$-tame} if for all nonperipheral $g\in G$, all arcs $I\subseteq T$, and all $l\ge 1$, if $g^lI=I$, then $g^kI=I$. Equivalently, a small $(G,\mathcal{F})$-tree is $k$-tame if for all nonperipheral $g\in G$, and all $l\in\mathbb{N}$, we have $\text{Fix}(g^k)=\text{Fix}(g^{kl})$. Notice that being $1$-tame is equivalent to being $\mathcal{Z}^{max}$-tame. Notice also that if a $(G,\mathcal{F})$-tree $T$ is $k$-tame, then it is also $kl$-tame for all $l\ge 1$. In particular, in view of the proposition below, if $T$ and $T'$ are two tame $(G,\mathcal{F})$-trees, then there exists $k\in\mathbb{N}$ such that both $T$ and $T'$ are $k$-tame.

\begin{prop}(Horbez \cite[Proposition 6.5]{Hor14-5})\label{good}
Let $T$ be a minimal $(G,\mathcal{F})$-tree. Then $T$ is tame if and only if there exists $k\in\mathbb{N}$ such that $T$ is $k$-tame.
\end{prop}

In \cite[Corollary 4.5]{Hor14-5}, we proved that trees in $\overline{\mathcal{O}(G,\mathcal{F})}$ have finitely many $G$-orbits of directions at branch or inversion points, so trees in $\overline{\mathcal{O}(G,\mathcal{F})}$ are $\mathcal{Z}^{max}$-tame trees. The converse is not true in general, because $\mathcal{Z}^{max}$-tame trees are not required to have trivial tripod stabilizers. However, tame $(G,\mathcal{F})$-trees with dense $G$-orbits have trivial arc stabilizers \cite[Proposition 4.17]{Hor14-5}. So tame $(G,\mathcal{F})$-trees with dense orbits belong to $\overline{\mathcal{O}(G,\mathcal{F})}$. 

Recall that the space of minimal $(G,\mathcal{F})$-trees is equipped with the axes topology. The subspace consisting of small $(G,\mathcal{F})$-trees is closed (\cite[5.3]{CM87}, \cite[Lemme 5.7]{Pau88}, \cite[Proposition 3.1]{Hor14-5}), and for all $k\in\mathbb{N}$, the subspace consisting of $k$-tame $(G,\mathcal{F})$-trees is closed \cite[Proposition 6.4]{Hor14-5}. However, the subspace consisting of all tame $(G,\mathcal{F})$-trees is not: for example, a sequence of splittings of $F_2=\langle a,b\rangle$ of the form $F_2=(\langle a\rangle\ast_{\langle a^2\rangle}\langle a^2\rangle\ast_{\langle a^4\rangle}\dots\ast_{\langle a^{2^n}\rangle}\langle a^{2^n}\rangle)\ast \langle b\rangle$, in which the edge with stabilizer generated by $a^{2^k}$ has length $\frac{1}{2^k}$, does not converge to a tame $F_2$-tree. The following proposition gives a condition under which a limit of tame trees is tame.

\begin{prop} (Horbez \cite[Proposition 6.7]{Hor14-5})\label{limit-one-edge}
Let $(T_n)_{n\in\mathbb{N}}$ be a sequence of one-edge $\mathcal{Z}$-splittings that converges (projectively) in the axes topology to a minimal $(G,\mathcal{F})$-tree $T$. Then $T$ is tame.
\end{prop}

\paragraph*{Tame optimal folding paths and sequences.}

An optimal liberal folding path $(T_t)_{t\in\mathbb{R}_+}$ is \emph{tame} (resp. \emph{$k$-tame}) if for all $t\in\mathbb{R}_+$, the tree $T_t$ is tame (resp. $k$-tame). Similarly, an optimal folding sequence $(T_n)_{n\in\mathbb{N}}$ is \emph{tame} if $T_n$ is tame for all $n\in\mathbb{N}$. We recall from Section \ref{sec-folding} the existence of a canonical optimal liberal folding path associated to any optimal morphism between two $(G,\mathcal{F})$-trees $T$ and $T'$.  

\begin{prop}\label{canonical-good}
Let $T$ and $T'$ be two tame $(G,\mathcal{F})$-trees, and let $f:T\to T'$ be an optimal morphism. Then the canonical optimal folding path $\gamma$ guided by $f$ is tame. More precisely, for all $k\in\mathbb{N}$, if $T$ and $T'$ are $k$-tame, then $\gamma$ is $k$-tame.
\end{prop}

\begin{proof}
Let $g\in G$ be a nonperipheral element, let $l\ge 1$, and let $K_t$ be the fixed point set of $g^{kl}$ in $T_t$. We want to show that $g^k$ also fixes $K_t$. Let $a_t\in K_t$, and let $a$ be a preimage of $a_t$ in $T$. By definition of $\sim_t$ (see Section \ref{sec-folding} for notations), we have $g^{kl}f(a)=f(a)$ and $\tau(a,g^{kl}a)\le t$. As $T'$ is $k$-tame, this implies that $g^kf(a)=f(a)$. We claim that $\tau(a,g^ka)= \tau(a,g^{kl}a)$.  This will imply that $g^ka_t=a_t$, and therefore $g^k$ fixes $K_t$ pointwise.

First assume that $g$ is hyperbolic in $T$. The segment $[a,g^{kl}a]\subseteq T$ decomposes as $[a,a']\cup [a',g^{k}a']\cup g^{k}[a',g^{k}a']\cup\dots\cup g^{k(l-1)}[a',g^{k}a']\cup g^{kl}[a',a]$. As $g^{k}f(a)=f(a)$, by equivariance of $f$, the supremum of the distance between $f(a)$ and a point in the $f$-image of either $[a,g^{kl}a]$ or $[a,g^ka]$ is achieved by a point in the $f$-image of $[a,g^{k}a']$. This implies that $\tau(a,g^ka)=\tau(a,g^{kl}a)$.  

Now assume that $g$ is elliptic in $T$. Then $[a,g^{kl}a]$ decomposes as $[a,g^{kl}a]=[a,a']\cup [a',g^{kl}a]$, where $a'$ is the point in $\text{Fix}_T(g^k)=\text{Fix}_T(g^{kl})$ closest to $a$. Since $g^{kl}f(a)=f(a)$, the supremum of the distance between $f(a)$ and a point in the $f$-image of $[a,g^{kl}a]$ is achieved by a point in the $f$-image of $[a,a']$. This implies that $\tau(a,g^ka)=\tau(a,g^{kl}a)$.   
\end{proof}

\begin{prop}\label{existence-folding}
Let $T$ be a tame $(G,\mathcal{F})$-tree, and let $T_0\in\mathcal{O}(G,\mathcal{F})$. Then there exists a tame optimal liberal folding path from a point in the closure of the cone of $T_0$ to $T$. In particular, there exists a tame optimal folding sequence ending at $T$.
\end{prop}

\begin{proof}
In view of Proposition \ref{canonical-good}, it is enough to show the existence of an optimal morphism from a point in the closure of the cone of $T_0$ to $T$. This is done in the following way: starting from any $G$-equivariant map from $T_0$ to $T$, one first assign to each edge of $T_0/G$ a new length equal to the length of its image in $T$. This defines a new tree $T'_0$ in the closure of the cone of $T_0$ and a morphism $f$ from $T'_0$ to $T$. By folding $T'_0$ at all points at which there is only one gate for the train track structure determined by $f$, one reaches yet another tree $T''_0$ in the closure of the cone of $T_0$, equipped with an optimal morphism from $T_0$ to $T$.
\end{proof}

\subsection{Metric properties of $\mathcal{O}(G,\mathcal{F})$}

We review work by Francaviglia and Martino \cite{FM13}. Let $G$ be a countable group, and $\mathcal{F}$ be a free factor system of $G$. For all $T,T'\in\mathcal{O}(G,\mathcal{F})$, we denote by $\text{Lip}(T,T')$ the infimal Lipschitz constant of an equivariant map from $T$ to $T'$. Let $T\in \mathcal{O}(G,\mathcal{F})$. An element $g\in G$ is a \emph{candidate} in $T$ if it is hyperbolic in $T$ and, denoting by $C_T(g)$ its translation axis in $T$, there exists $v\in C_T(g)$ such that the segment $[v,gv]$ projects to a loop $\gamma$ in the quotient graph $X:=T/G$ which is either

\begin{itemize}
\item an embedded loop, or
\item a bouquet of two circles in $X$, i.e. $\gamma=\gamma_1\gamma_2$, where $\gamma_1$ and $\gamma_2$ are embedded circles in $X$ which meet in a single point, or
\item a barbell graph, i.e. $\gamma=\gamma_1\eta\gamma_2\overline{\eta}$, where $\gamma_1$ and $\gamma_2$ are embedded circles in $X$ that do not meet, and $\eta$ is an embedded path in $X$ that meets $\gamma_1$ and $\gamma_2$ only at their origin (and $\overline{\eta}$ denotes the path $\eta$ crossed in the opposite direction), or
\item a \emph{simply-degenerate barbell}, i.e. $\gamma$ is of the form $u\eta\overline{\eta}$, where $u$ is an embedded loop in $X$ and $\eta$ is an embedded path in $X$, with two distinct endpoints, which meets $u$ only at its origin, and whose terminal endpoint is a vertex in $X$ with nontrivial stabilizer, or
\item a \emph{doubly-degenerate barbell}, i.e. $\gamma$ is of the form $\eta\overline{\eta}$, where $\eta$ is an embedded path in $X$ whose two distinct endpoints have nontrivial stabilizer,
\end{itemize} 

\noindent see Figure \ref{fig-cand} for a representation of the possible shapes of candidate loops. Given $T\in \mathcal{O}(G,\mathcal{F})$, we denote by $\text{Cand}(T)$ the (infinite) set of all elements in $G$ which are candidates in $T$. 

\begin{figure}
\begin{center}
\input{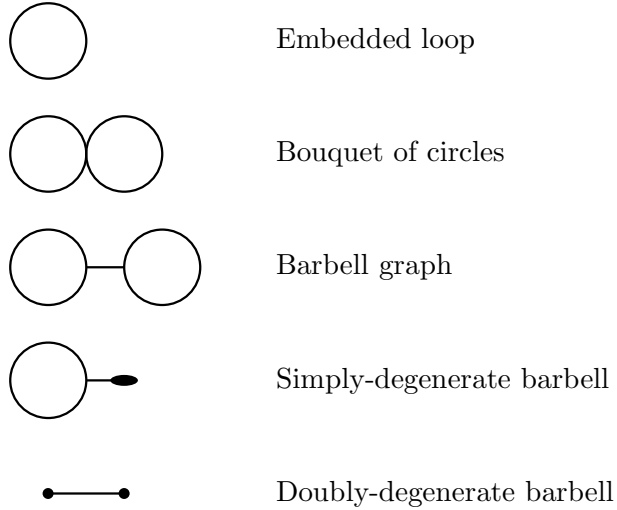}
\caption{Shapes of candidate loops.}
\label{fig-cand}
\end{center}
\end{figure}

\begin{theo} (Francaviglia--Martino \cite[Theorem 9.10]{FM13}) \label{candidate}
For all $T,T'\in \mathcal{O}(G,\mathcal{F})$, we have
\begin{displaymath}
\text{Lip}(T,T')=\sup_{g\in\text{Cand}(T)}\frac{||g||_{T'}}{||g||_T}.
\end{displaymath}

\noindent In addition, there exists a tree $\overline{T}\in\mathcal{O}(G,\mathcal{F})$ onto which $T$ admits a $\text{Lip}(T,T')$-Lipschitz alignment-preserving map, together with an optimal morphism from $\overline{T}$ to $T'$.
\end{theo}

\noindent Building on Francaviglia and Martino's theorem, we show the following result.

\begin{theo}\label{candidate-finite}
For all $T\in \mathcal{O}(G,\mathcal{F})$, there exists a finite set $X(T)\subseteq \text{Cand}(T)$ such that for all $T'\in \mathcal{O}(G,\mathcal{F})$, we have
\begin{displaymath}
\text{Lip}(T,T')=\sup_{g\in X(T)}\frac{||g||_{T'}}{||g||_T}.
\end{displaymath}
\end{theo}

Let $T,T'\in\mathcal{O}(G,\mathcal{F})$, and $f:T\to T'$ be an optimal morphism. An element $g\in G$ is \emph{legal} for $f$ if it is hyperbolic in $T$, and if its axis is legal for $f$. The \emph{tension graph} of $f$ is the set of edges of $T$ that are maximally strecthed by $f$. Francaviglia and Martino's proof of Theorem \ref{candidate} shows that there exists an optimal morphism $f:T\to T'$, and there exists $g\in G$ which is legal for $f$, and whose axis is contained in the tension graph of $f$. In addition, such an element $g\in G$ can be chosen to be a candidate in $T$. Theorem \ref{candidate} follows, because every such element maximizes the stretch factor from $T$ to $T'$. 

\begin{proof}[Proof of Theorem \ref{candidate-finite}]
The set $Y(T)$ of possible projections of axes of candidates in $T$ to the quotient graph $T/G$ is finite. By Theorem \ref{candidate}, it is thus enough to show that from the set of all candidates in $T$ whose projections are equal to some $\gamma\in Y(T)$, we can extract a finite subset $X(\gamma)$ so that for all $T'\in \mathcal{O}(G,\mathcal{F})$, and all optimal morphisms $f:T\to T'$, at least one element in $X(\gamma)$ is legal for $f$. This follows from the observation that for every pair $(e,e')$ of adjacent edges in $T$ whose common vertex $v$ has nontrivial stabilizer $G_v$, and any two distinct elements $g,g'\in G_v$, either $(e,ge')$ or $(e,g'e')$ is legal for $f$, since otherwise $g'g^{-1}$ would fix a nondegenerate arc in $T'$. In addition, any loop in $Y(T)$ crosses boundedly many vertices in $T/G$.  
\end{proof}

\subsection{Lipschitz approximations of trees}\label{sec-Lipschitz}

Let $T\in\overline{\mathcal{O}(G,\mathcal{F})}$. A \emph{Lipschitz approximation} of $T$ is a sequence $(T_n)_{n\in\mathbb{N}}\in\overline{\mathcal{O}(G,\mathcal{F})}^{\mathbb{N}}$ converging (non-projectively) to $T$ such that for all $n\in\mathbb{N}$, there exists a $1$-Lipschitz map from $T_n$ to $T$. The following proposition follows from \cite[Theorem 5.3]{Hor14-5}.

\begin{prop} (Horbez \cite[Theorem 5.3]{Hor14-5}) \label{Lipschitz-approximation}
Every tree $T\in\overline{\mathcal{O}(G,\mathcal{F})}$ with dense orbits admits a Lipschitz approximation by Grushko $(G,\mathcal{F})$-trees. 
\end{prop}

\begin{prop}\label{csq-Lipschitz}
Let $S,T\in\overline{\mathcal{O}(G,\mathcal{F})}$, let $(S_i)_{i\in\mathbb{N}}\in\mathcal{O}(G,\mathcal{F})^{\mathbb{N}}$ be a Lipschitz approximation of $S$, and let $(T_j)_{j\in\mathbb{N}}\in\mathcal{O}(G,\mathcal{F})^{\mathbb{N}}$ be a sequence that converges (non-projectively) to $T$. Assume in addition that there exists a $1$-Lipschitz map from $S$ to $T$. Then for all $i\in\mathbb{N}$, there exists $J_i\in\mathbb{N}$ so that for all $j\ge J_i$, we have $\text{Lip}(S_i,T_j)\le 1+\frac{1}{i}$.
\end{prop}

\begin{proof}
Let $i\in\mathbb{N}$. As $(T_j)_{j\in\mathbb{N}}$ converges non-projectively to $T$, there exists $J_i\in\mathbb{N}$ so that for all $j\ge J_i$, all elements $g$ in the finite set $X(S_i)$ provided by Theorem \ref{candidate-finite} have translation length at most $(1+\frac 1 i)||g||_T\le (1+\frac 1 i)||g||_S\le (1+\frac 1 i)||g||_{S_i}$ in $T_{j}$. The claim then follows from Theorem \ref{candidate-finite}.
\end{proof}

\subsection{Alignment-preserving maps}

A map $f:T\to T'$ is \emph{alignment-preserving} if the $f$-image of every segment in $T$ is a segment in $T'$. We note that alignment-preserving maps are not assumed to be continuous. However, any alignment-preserving map is continuous in restriction to every segment of $T$, and more generally in restriction to every finite subtree of $T$. If there exists an alignment-preserving map from $T$ to $T'$, we say that $T$ \emph{collapses} to $T'$. The following lemma states a few basic topological properties of alignment-preserving maps. It follows from the fact that a subtree of an $\mathbb{R}$-tree is closed if and only if its intersection with every segment in $T$ is closed. Details of the proof are left to the reader.

\begin{lemma} \label{projection-closed}
Let $T$ and $\widehat{T}$ be two $(G,\mathcal{F})$-trees. Let $p:\widehat{T}\to T$ be a surjective alignment-preserving map. Then the $p$-preimage of every closed subtree in $T$ is a closed subtree of $\widehat{T}$. The $p$-image of every closed subtree of $\widehat{T}$ is a closed subtree of $T$.
\qed
\end{lemma}

\subsection{Limits of folding paths}

The goal of this section is to prove Proposition \ref{limit-folding}, which will be used in Section \ref{sec-boundary}, and gives information about possible limit points in $\overline{\mathcal{O}(G,\mathcal{F})}$ of some folding paths in $\mathcal{O}(G,\mathcal{F})$. In the following statement, the last assertion about alignment-preserving maps is an immediate consequence of our description in \cite{Hor14-1} of the map $f$, which is obtained from an ultralimit of the maps $f_n$ by projecting to the minimal subtree.

\begin{prop} (Horbez \cite[Theorem 4.3]{Hor14-1}) \label{Lipschitz-limit}
Let $T$ and $T'$ be tame $(G,\mathcal{F})$-trees, let $(T_n)_{n\in\mathbb{N}}$ (resp. $(T'_n)_{n\in\mathbb{N}}$) be a sequence of tame $(G,\mathcal{F})$-trees that converges to $T$ (resp. $T'$), and let $M\in\mathbb{R}$. Assume that for all $n\in\mathbb{N}$, there exists an $M$-Lipschitz map from $T_n$ to $T'_n$. Then there exists an $M$-Lipschitz map from $T$ to the metric completion of $T'$. If in addition, all maps $f_n$ are alignment-preserving, then $f$ can be chosen to be alignment-preserving.
\end{prop}

Let $T,T'\in\overline{\mathcal{O}(G,\mathcal{F})}$, and $f:T\to T'$ be a map. The \emph{bounded cancellation constant} of $f$, denoted by $BCC(f)$, is defined to be the supremum of all real numbers $B$ with the property that there exist $a,b,c\in T$ with $b\in [a,c]$, such that $d_{T'}(f(b),[f(a),f(c)])=B$. Note that a map $f:T\to T'$ is alignment-preserving if and only if $BCC(f)=0$. We denote by $\text{Lip}(f)$ the Lipschitz constant of $f$, and by $qvol(T)$ the quotient volume of $T$, defined as the infimal volume of a finite subtree of $T$ whose $G$-translates cover $T$ (the existence of such a tree was proved by Guirardel in \cite[Lemma 1.14]{Gui08}). The following proposition is a generalization of \cite[Lemma 3.1]{BFH97}.

\begin{prop} \label{bcc} 
Let $T\in\mathcal{O}(G,\mathcal{F})$, let $T'\in \overline{\mathcal{O}(G,\mathcal{F})}$, and let $f:T\to T'$ be a Lipschitz map. Then $BCC(f)\le Lip(f) qvol(T)$.
\end{prop}

\begin{proof}[Sketch of proof]
In the case where $T'$ is a Grushko $(G,\mathcal{F})$-tree, the statement follows by decomposing $f$ into Stallings' folds (see the proof of \cite[Proposition 9.6]{Gui98}). The claim is then proved by approximating $T'$ by trees in $\mathcal{O}(G,\mathcal{F})$.
\end{proof}

As in \cite[Corollary 3.9]{Hor14-1}, the following fact is a corollary of Propositions \ref{Lipschitz-approximation} and \ref{bcc}.

\begin{cor} (Horbez \cite[Corollary 3.9]{Hor14-1}) \label{alignment-preserved} 
Let $T,T'\in\overline{\mathcal{O}(G,\mathcal{F})}$ have dense orbits. Then any Lipschitz map from $T$ to the metric completion of $T'$ preserves alignment (and hence takes its values in $T'$). In particular, every morphism from $T$ to $T'$ is an isometry.
\qed
\end{cor}

\begin{prop} \label{limit-folding}
Let $S,T\in\overline{\mathcal{O}(G,\mathcal{F})}$ be two trees with dense orbits. Let $(S_i)_{i\in\mathbb{N}}\in\mathcal{O}(G,\mathcal{F})^{\mathbb{N}}$ (resp. $(T_i)_{i\in\mathbb{N}}\in\mathcal{O}(G,\mathcal{F})^{\mathbb{N}}$) be a sequence that converges (non-projectively) to $S$ (resp. to $T$). Assume that $S$ admits a $1$-Lipschitz alignment-preserving map onto $T$, and that for all $i\in\mathbb{N}$, we have $\text{Lip}(S_i,T_i)\le 1+\frac{1}{i}$. Then there exists an optimal liberal folding path $\gamma_i$ from the cone of $S_i$ to $T_{i}$ for all $i\in\mathbb{N}$, so that all sequences $(Z_i)_{i\in\mathbb{N}}\in\prod_{i\in\mathbb{N}}\text{Im}(\gamma_i)$ have nontrivial accumulation points in $\overline{\mathcal{O}(G,\mathcal{F})}$, and all such accumulation points $Z$ come with $1$-Lipschitz alignment-preserving maps from $S$ to $Z$ and from $Z$ to $T$.  
\end{prop}

\begin{proof}
Theorem \ref{candidate} yields the existence for all $i\in\mathbb{N}$ of a tree $S'_i\in\mathcal{O}(G,\mathcal{F})$, obtained from $S_i$ by rescaling the lengths of the edges by a factor bounded above by $1+\frac{1}{i}$, such that there exist optimal liberal folding paths from $S'_i$ to $T_{i}$. For all $i\in\mathbb{N}$, let $Z_{i}\in \mathcal{O}(G,\mathcal{F})$ be a tree that lies on a liberal folding path from $S'_i$ to $T_{i}$. There are $1$-Lipschitz maps from $(1+\frac{1}{i})S_i$ to $Z_{i}$ and from $Z_{i}$ to $T_{i}$. In particular, for any accumulation point $Z$ of the sequence $(Z_{i})_{i\in\mathbb{N}}$, Proposition \ref{Lipschitz-limit} yields the existence of $1$-Lipschitz maps from $S$ to the metric completion of $Z$ and from $Z$ to the metric completion of $T$ (in particular, the set of accumulation points contains nontrivial $G$-trees). Corollary \ref{alignment-preserved} implies that these maps are alignment-preserving, and take values in $Z$ and $T$ (without passing to the completion).
\end{proof}

\subsection{Refinements of metric trees} \label{sec-compatibility}

Let $T_1$ and $T_2$ be two \emph{compatible} $(G,\mathcal{F})$-trees, i.e. there exists a $(G,\mathcal{F})$-tree $\widehat{T}$ and alignment-preserving maps $g_i:\widehat{T}\to T_i$ for all $i\in\{1,2\}$. Then $T_1$ and $T_2$ have a \emph{standard common refinement}, defined as follows \cite[Section 3.2]{GL10-2}. For all $i\in\{1,2\}$, denote by $d_i$ the metric on $T_i$, and by $l_i$ the associated length function, and for all $x,y\in\widehat{T}$, let
\begin{displaymath}
\delta(x,y):=d_1(g_1(x),g_1(y))+d_2(g_2(x),g_2(y)).
\end{displaymath} 

\noindent This defines a pseudometric on $\widehat{T}$, which satisfies $\delta(x,y)=\delta(x,z)+\delta(z,y)$ whenever $z\in [x,y]$. The metric space $T_s$ obtained by making this pseudometric Hausdorff is a $(G,\mathcal{F})$-tree, which admits a $1$-Lipschitz alignment-preserving map $f_i:T_s\to T_i$ for all $i\in\{1,2\}$, such that 
\begin{displaymath}
d_{T_s}(x,y)=d_1(f_1(x),f_1(y))+d_2(f_2(x),f_2(y)).
\end{displaymath} 

\noindent Arc stabilizers of $T_s$ fix an arc in either $T_1$ or $T_2$, and if $T_1$ and $T_2$ are $k$-tame, then so is $T_s$. Since
\begin{displaymath}
||g||_{T_s}=\lim_{n\to +\infty}\frac{1}{n}d_{T_s}(x,g^nx)
\end{displaymath}  

\noindent for all $x\in T_s$, it follows that the length function of $T_s$ is the sum of the length functions of $T_1$ and $T_2$. We will denote $T_s=:T_1+T_2$.

\begin{lemma}\label{compatible-limit}(Guirardel--Levitt \cite[Corollary 3.9]{GL10-2})
Let $S$ and $T$ be two $(G,\mathcal{F})$-trees. Let $(S_i)_{i\in\mathbb{N}}$ (resp. $(T_i)_{i\in\mathbb{N}}$) be a sequence of trees that converges in the axes topology to $S$ (resp. to $T$). If $S_i$ is compatible with $T_i$ for all $i\in\mathbb{N}$, then $S$ is compatible with $T$.
\end{lemma}

\subsection{Transverse families, transverse coverings and graphs of actions}\label{sec-goa}

Let $T$ be a $(G,\mathcal{F})$-tree. A \emph{transverse family} in $T$ is a $G$-invariant collection $\mathcal{Y}$ of nondegenerate subtrees of $T$ such that for all $Y\neq Y'\in\mathcal{Y}$, the intersection $Y\cap Y'$ contains at most one point. A \emph{transverse covering} of $T$ is a transverse family $\mathcal{Y}$ in $T$, all of whose elements are closed subtrees of $T$, such that every finite arc in $T$ can be covered by finitely many elements of $\mathcal{Y}$. A transverse covering $\mathcal{Y}$ of $T$ is \emph{trivial} if $\mathcal{Y}=\{T\}$. The \emph{skeleton} of a transverse covering $\mathcal{Y}$ is the bipartite simplicial tree $S$, whose vertex set is $V(S)=V_0(S)\cup\mathcal{Y}$, where $V_0(S)$ is the set of points of $T$ which belong to at least two distinct trees in $\mathcal{Y}$, with an edge between $x\in V_0(S)$ and $Y\in\mathcal{Y}$ whenever $x\in Y$ \cite[Definition 4.8]{Gui04}. A \emph{$(G,\mathcal{F})$-graph of actions} consists of

\begin{itemize}
\item a marked metric graph of groups $\mathcal{G}$ (in which we might allow some edges to have length $0$), whose fundamental group is isomorphic to $G$, such that all subgroups in $\mathcal{F}$ are conjugate into vertex groups of $\mathcal{G}$, and 

\item an isometric action of every vertex group $G_v$ on a $G_v$-tree $T_v$ (possibly reduced to a point), in which all intersections of $G_v$ with peripheral subgroups of $G$ are elliptic, and

\item a point $p_e\in T_{t(e)}$ fixed by $i_e(G_e)\subseteq G_{t(e)}$ for every oriented edge $e$.
\end{itemize}

It is \emph{nontrivial} if $\mathcal{G}$ is not reduced to a point. Associated to any $(G,\mathcal{F})$-graph of actions $\mathcal{G}$ is a $G$-tree $T(\mathcal{G})$. Informally, the tree $T(\mathcal{G})$ is obtained from the Bass--Serre tree of the underlying graph of groups by equivariantly attaching each vertex tree $T_v$ at the corresponding vertex $v$, an incoming edge being attached to $T_v$ at the prescribed attaching point. The reader is referred to \cite[Proposition 3.1]{Gui98} for a precise description of the tree $T(\mathcal{G})$. We say that a $(G,\mathcal{F})$-tree $T$ \emph{splits as a $(G,\mathcal{F})$-graph of actions} if there exists a $(G,\mathcal{F})$-graph of actions $\mathcal{G}$ such that $T=T({\mathcal{G}})$.

\begin{prop} (Guirardel \cite[Lemmas 1.5 and 1.15]{Gui08})\label{skeleton}
A $(G,\mathcal{F})$-tree splits as a nontrivial $(G,\mathcal{F})$-graph of actions if and only if it admits a nontrivial transverse covering. The skeleton of any transverse covering of $T$ is compatible with $T$, and it is minimal if $T$ is minimal.
\end{prop}

Knowing that a $(G,\mathcal{F})$-tree $T$ is compatible with a simplicial $(G,\mathcal{F})$-tree $S$ gives a way of splitting $T$ as a $(G,\mathcal{F})$-graph of actions, in the following way. Let $\pi_T:T+S\to T$ and $\pi_S:T+S\to S$ be the natural alignment-preserving maps. 

We first claim that the family $\mathcal{Y}$ made of all nondegenerate $\pi_S$-preimages of vertices of $S$, and of the closures of $\pi_S$-preimages of open edges of $S$, is a transverse covering of $T+S$. Indeed, these are closed subtrees of $T+S$ (see Lemma \ref{projection-closed}), whose pairwise intersections are degenerate (i.e either empty, or reduced to a point). In addition, let $I\subseteq T+S$ be a segment. Then $\pi_S(I)$ is a segment in $S$, so it is covered by a finite set of open edges and vertices of $S$. The $\pi_S$-preimages in $T+S$ of these edges and vertices cover $I$, which proves that $I$ is covered by finitely many elements of the family $\mathcal{Y}$. 

We claim that the family consisting of nondegenerate subtrees in $\pi_T(\mathcal{Y})$ is a nontrivial transverse covering of $T$. Indeed, since $\pi_T$ is alignment-preserving, this is a transverse family made of closed subtrees of $T$ (Lemma \ref{projection-closed}). If $I\subseteq T$ is a segment, then there is a segment $J\subseteq T+S$ with $\pi_T(J)=I$. Then $J$ is covered by finitely many elements of $\mathcal{Y}$, and $I$ is covered by the $\pi_T$-images of these elements. The family $\pi_T(\mathcal{Y})$ is nontrivial, because otherwise $\mathcal{Y}$ would also be trivial, and hence $\pi_S(T+S)$ would be contained in a vertex or a closed edge of $S$, a contradiction.
\\
\\
We finish this section by mentioning a result due to Levitt \cite{Lev94}, which gives a canonical way of splitting any tame $(G,\mathcal{F})$-tree as a graph of actions, whose vertex actions have dense orbits. The key point in the proof of Proposition \ref{Levitt} is finiteness of the number of orbits of branch points in any tame $(G,\mathcal{F})$-tree. In other words, tame actions are $J$-actions in Levitt's sense.

\begin{prop} (Levitt \cite{Lev94})\label{Levitt}
Every tame $(G,\mathcal{F})$-tree $T$ splits uniquely as a graph of actions, all of whose vertex trees have dense orbits for the action of their stabilizer (they might be reduced to points), and all of whose edges have positive length, and have a stabilizer that belongs to the class $\mathcal{Z}$.
\end{prop}

We call this decomposition the \emph{Levitt decomposition} of $T$ as a graph of actions. Proposition \ref{Levitt} gives a natural way of extending the map $\psi:\mathcal{O}(G,\mathcal{F})\to FZ(G,\mathcal{F})$ to the set of tame $(G,\mathcal{F})$-trees without dense orbits.

\section{$\mathcal{Z}$-averse trees}\label{sec-X}

We now introduce the notion of \emph{$\mathcal{Z}$-averse} trees. These will be the trees lying \emph{at infinity} of the complex $FZ(G,\mathcal{F})$. Most arguments work exactly the same way when working with $\mathcal{Z}^{max}$-splittings instead of $\mathcal{Z}$-splittings, we will mention the places where some slight adaptations are required. The case of $\mathcal{Z}^{max}$-splittings will be treated in Section \ref{sec-max}.

\subsection{Definition}

Given a tame $(G,\mathcal{F})$-tree $T$, we denote by $\mathcal{R}^1(T)$ the set of $\mathcal{Z}$-splittings that are compatible with $T$, and by $\mathcal{R}^2(T)$ the set of $\mathcal{Z}$-splittings that are compatible with a tame $(G,\mathcal{F})$-tree $T'$, which is itself compatible with $T$. We say that $T$ is \emph{$\mathcal{Z}$-incompatible} if $\mathcal{R}^1(T)=\emptyset$, and \emph{$\mathcal{Z}$-compatible} otherwise.

\begin{theo}\label{averse}
For all tame $(G,\mathcal{F})$-trees $T$, the following assertions are equivalent.

\begin{enumerate}
\item There exists a finite sequence $(T=T_0,T_1,\dots,T_k=S)$ of tame $(G,\mathcal{F})$-trees, such that $S$ is simplicial, and for all $i\in\{0,\dots,k-1\}$, the trees $T_i$ and $T_{i+1}$ are compatible.
\item We have $\mathcal{R}^2(T)\neq\emptyset$.
\item The tree $T$ collapses to a tame $\mathcal{Z}$-compatible $(G,\mathcal{F})$-tree.
\end{enumerate}
\end{theo}

The implications $(3)\Rightarrow (2)\Rightarrow (1)$ are obvious. The proof that $(1)$ implies $(3)$ is postponed to Section \ref{sec-averse}. We note that the conditions in Theorem \ref{averse} are not equivalent to being $\mathcal{Z}$-incompatible, see Example \ref{R-empty}, which is why we really need to introduce the set $\mathcal{R}^{2}(T)$ in our arguments. We will see however in Proposition \ref{mixing-representative} that mixing $\mathcal{Z}$-incompatible trees satisfy $\mathcal{R}^2(T)=\emptyset$.

\begin{de}
~
\begin{enumerate}
\item A tree $T\in\overline{\mathcal{O}(G,\mathcal{F})}$ is \emph{$\mathcal{Z}$-averse} if any of the conditions in Theorem \ref{averse} fails. We denote by $\mathcal{X}(G,\mathcal{F})$ the subspace of $\overline{\mathcal{O}(G,\mathcal{F})}$ consisting of $\mathcal{Z}$-averse trees. 

\item Two $\mathcal{Z}$-averse trees $T,T'\in\mathcal{X}(G,\mathcal{F})$ are \emph{equivalent}, which we denote by $T\sim T'$, if there exists a finite sequence $(T=T_0,T_1,\dots,T_k=T')$ of trees in $\overline{\mathcal{O}(G,\mathcal{F})}$ such that for all $i\in\{0,\dots,k-1\}$, the trees $T_i$ and $T_{i+1}$ are compatible.

\item Given a tame $(G,\mathcal{F})$-tree $T$, the set $\mathcal{R}^2(T)$ is called the set of \emph{reducing splittings} of $T$.
\end{enumerate}
\end{de}

Note that $\mathcal{Z}$-averse trees are $\mathcal{Z}$-incompatible, so in particular they have dense orbits. A tree $T\in\overline{\mathcal{O}(G,\mathcal{F})}$ is \emph{mixing} \cite{Mor88} if for all finite subarcs $I,J\subseteq T$, there exist $g_1,\dots,g_k\in G$ such that $J\subseteq g_1 I\cup\dots\cup g_kI$ and for all $i\in\{1,\dots,k-1\}$, we have $g_iI\cap g_{i+1}I\neq\emptyset$. We will show the existence of a canonical simplex of mixing representatives in any equivalence class of $\mathcal{Z}$-averse trees. Two $\mathbb{R}$-trees $T$ and $T'$ are \emph{weakly homeomorphic} if there exist maps $f:T\to T'$ and $g:T'\to T$ that are continuous in restriction to segments, and inverse of each other. Again, the proof of Proposition \ref{mixing-representative} is postponed to Section \ref{sec-averse}.

\begin{prop}\label{mixing-representative}
For all $T\in\mathcal{X}(G,\mathcal{F})$, there exists a mixing tree $\overline{T}\in\mathcal{X}(G,\mathcal{F})$ onto which all trees $T'\in\mathcal{X}(G,\mathcal{F})$ that are equivalent to $T$ collapse. In addition, any two such trees are $G$-equivariantly weakly homeomorphic. Any tree $T\in\overline{\mathcal{O}(G,\mathcal{F})}$ that is both mixing and $\mathcal{Z}$-incompatible is $\mathcal{Z}$-averse.
\end{prop}

When taking the quotient by homotheties, the equivalence relation on $\mathcal{X}(G,\mathcal{F})$ induces an equivalence relation of $P\mathcal{X}(G,\mathcal{F})$. As was pointed to us by the referee, we obtain as a consequence of Proposition \ref{mixing-representative} the following information about equivalence classes in $\overline{P\mathcal{O}(G,\mathcal{F})}$.

\begin{prop}\label{contractible}
Let $T\in\mathcal{X}(G,\mathcal{F})$. Then the $\sim$-class of $T$ is a compact, contractible subspace of $\overline{P\mathcal{O}(G,\mathcal{F})}$.
\end{prop}

\begin{proof}
Let $\overline{T}$ be a mixing tree in the $\sim$-class of $T$, provided by Proposition \ref{mixing-representative}. Then the straight paths from the trees $T'$ in the $\sim$-class of $T$, to the tree $\overline{T}$, consisting of all trees of the form $t\overline{T}+(1-t)T'$ with $t\in [0,1]$, define a retraction of the $\sim$-class of $T$ to $\overline{T}$. This shows that the $\sim$-class of $T$ is contractible.

To prove compactness, let $(T_n)_{n\in\mathbb{N}}$ be a sequence of trees in the $\sim$-class of $T$, and let $T_{\infty}$ be a limit point of the sequence $(T_n)_{n\in\mathbb{N}}$ in $\overline{P\mathcal{O}(G,\mathcal{F})}$. All trees $T_n$ are compatible with $\overline{T}$, so by Lemma \ref{compatible-limit}, the tree $T_{\infty}$ is also compatible with $\overline{T}$. Hence $T_{\infty}$ belongs to the $\sim$-class of $T$.
\end{proof}

\subsection{Unboundedness of $FZ(G,\mathcal{F})$}\label{sec-unbounded}

\subsubsection{Kobayashi's argument}

We now explain that $\mathcal{Z}$-averse trees lie \emph{at infinity} of $FZ(G,\mathcal{F})$ in some sense. In particular, we show that $\mathcal{X}(G,\mathcal{F})$ is unbounded (except in the two sporadic cases mentioned in Remark \ref{sporadic}, for which we have $\mathcal{X}(G,\mathcal{F})=\emptyset$). The following theorem is a variation over an argument due to Kobayashi \cite{Kob88} to prove unboundedness of the curve complex of a compact surface. We recall our notation $\psi$ for the map from $\mathcal{O}(G,\mathcal{F})$ to $FZ(G,\mathcal{F})$.

\begin{theo} \label{Luo}
Let $T\in\mathcal{X}(G,\mathcal{F})$, and let $(T_i)_{i\in\mathbb{N}}\in \mathcal{O}(G,\mathcal{F})^{\mathbb{N}}$ be a sequence that converges to $T$. Then $\psi(T_i)$ is unbounded in $FZ(G,\mathcal{F})$.
\end{theo}

\begin{proof}
Assume towards a contradiction that the sequence $(\psi(T_i))_{i\in\mathbb{N}}$ lies in a bounded region of $FZ(G,\mathcal{F})$. Up to passing to a subsequence, there exist $M\in\mathbb{N}$ and $\ast\in FZ(G,\mathcal{F})$ such that for all $i\in\mathbb{N}$, we have $d_{FZ(G,\mathcal{F})}(\ast,\psi(T_i))=M$. For all $i\in\mathbb{N}$, let $(T_i^k)_{0\le k\le M}$ be a geodesic segment joining $\ast$ to $\psi(T_i)$ in $FZ(G,\mathcal{F})$. Up to passing to a subsequence again and rescaling, we may assume that for all $k\in\{0,\dots,M\}$, the sequence $(T_i^k)_{i\in\mathbb{N}}$ of one-edge splittings converges (non-projectively) to a tame $(G,\mathcal{F})$-tree $T_{\infty}^k$ (Proposition \ref{limit-one-edge}). In addition, for all $k\in\{1,\dots,M\}$ and all $i\in\mathbb{N}$, the trees $T_i^k$ and $T_i^{k-1}$ are compatible. Lemma \ref{compatible-limit} implies that for all $k\in\{1,\dots,M\}$, the trees $T^k_{\infty}$ and $T^{k-1}_{\infty}$ are compatible, and $T$ is compatible with $T^M_{\infty}$. As $T_{\infty}^{0}=\ast$, the tree $T$ does not satisfy the first definition of $\mathcal{Z}$-averse trees, a contradiction.   
\end{proof}

\begin{rk}
In the case of $\mathcal{Z}^{max}$-splittings, the argument is even a bit simpler, because in this case, we know that all trees $T_i^k$, and hence all limits $T_{\infty}^k$, belong to the closure of outer space (i.e. they are very small). Therefore, we can avoid to appeal to our analysis of tame $(G,\mathcal{F})$-trees (in particular Proposition \ref{limit-one-edge}) in Section \ref{sec-good}.
\end{rk}

\subsubsection{Examples of $\mathcal{Z}$-averse trees, and unboundedness of $FZ(G,\mathcal{F})$}

We now give examples of $\mathcal{Z}$-averse trees in $\overline{\mathcal{O}(G,\mathcal{F})}$ when either $G=F_2$ and $\mathcal{F}=\emptyset$, or $\text{rk}_K(G,\mathcal{F})\ge 3$. In view of Theorem \ref{Luo}, this will prove unboundedness of the graph $FZ(G,\mathcal{F})$ in these cases. Recall from Remark \ref{sporadic} that if either $G=G_1\ast G_2$ and $\mathcal{F}=\{[G_1],[G_2]\}$, or $G=G_1\ast$ and $\mathcal{F}=\{[G_1]\}$, then $FZ(G,\mathcal{F})$ is bounded. In other words, we have the following. 

\begin{prop}\label{ex-unbounded}
Let $G$ be a countable group, and $\mathcal{F}$ be a free factor system of $G$. Assume that either $G=F_2$ and $\mathcal{F}=\emptyset$, or that $\text{rk}_K(G,\mathcal{F})\ge 3$. Then $\mathcal{X}(G,\mathcal{F})\neq\emptyset$.
\end{prop}

\begin{cor}\label{fz-unbounded}
Let $G$ be a countable group, and let $\mathcal{F}$ be a free factor system of $G$. Then $FZ(G,\mathcal{F})$ has unbounded diameter if and only if either $G=F_2$ and $\mathcal{F}=\emptyset$, or $\text{rk}_K(G,\mathcal{F})\ge 3$.
\qed
\end{cor}

\begin{rk}
The examples we provide belong to $\mathcal{X}^{max}(G,\mathcal{F})$. We refer to Section \ref{sec-ex-max} for examples of $\mathcal{Z}^{max}$-averse trees that are not $\mathcal{Z}$-averse.
\end{rk}

\begin{rk}
It will actually follow from our main result (Theorem \ref{main}) that, except in the sporadic cases, the Gromov boundary $\partial_{\infty}FZ(G,\mathcal{F})$ is nonempty.
\end{rk}

A tree $T\in\overline{\mathcal{O}(G,\mathcal{F})}$ is \emph{indecomposable} \cite[Definition 1.17]{Gui08} if for all segments $I,J\subseteq T$, there exist $g_1,\dots,g_r\in G$ such that $J=\cup_{i=1}^r g_iI$, and for all $i\in\{1,\dots,r-1\}$, the intersection $g_iI\cap g_{i+1}I$ is a nondegenerate arc (i.e. it is nonempty, and not reduced to a point). The following lemma follows from \cite[Lemma 1.18]{Gui08} and the description of the transverse covering of $T$ provided by a simplicial tree $S$ that is compatible with $T$ (Section \ref{sec-goa}).  

\begin{lemma}(Guirardel \cite[Lemma 1.18]{Gui08})\label{lemma-ex}
Let $T\in\overline{\mathcal{O}(G,\mathcal{F})}$ be a tree that is compatible with a $\mathcal{Z}$-splitting $S$. Let $H\subseteq G$ be a subgroup, such that the $H$-minimal subtree $T_H$ of $T$ is indecomposable. Then $H$ is elliptic in $S$.
\end{lemma}

\begin{proof}[Proof of Proposition \ref{ex-unbounded}]
If $G=F_2$ and $\mathcal{F}=\emptyset$, then any tree dual to an arational measured lamination on a compact surface of genus $1$ having exactly one boundary component is arational in the sense of \cite{Rey12}. Hence it is mixing and $\mathcal{Z}$-incompatible, so it belongs to $\mathcal{X}(G,\mathcal{F})$ by Proposition \ref{mixing-representative}. 

We now assume that $\text{rk}_K(G,\mathcal{F})\ge 3$. Let $N:=\text{rk}_f(G,\mathcal{F})$, and let $\{G_1,\dots,G_k\}$ be a set of representatives of the conjugacy classes in $\mathcal{F}$, such that $$G=G_1\ast\dots\ast G_k\ast F_N.$$ For all $i\in\{1,\dots,k\}$, we choose an element $g_i\in G_i\smallsetminus\{e\}$, whose order we denote by $p_i\in\mathbb{N}\cup\{+\infty\}$. We denote by $l$ be the number of indices $i$ so that $p_i=+\infty$. Up to reordering the $g_i$'s, we can assume that $p_1,\dots,p_l=+\infty$, and $p_{l+1},\dots,p_k<+\infty$. 

Let $\mathcal{O}$ be the orbifold obtained from a sphere with $N+l+1$ boundary components by adding a conical point of order $p_i$ for each $i\in\{l+1,\dots,k\}$. For all $i\in\{1,\dots,l\}$, we denote by $b_i$ a generator of the $i^{th}$ boundary curve in $\pi_1(\mathcal{O})$, and for all $i\in\{l+1,\dots,k\}$, we denote by $b_i$ a generator of the subgroup of $\pi_1(\mathcal{O})$ associated to the corresponding conical point. The group $G$ is isomorphic to the group obtained by amalgamating $\pi_1(\mathcal{O})$ with the groups $G_i$, identifying $b_i$ with $g_i$ for all $i\in\{1,\dots,k\}$, see Figure \ref{fig-orbifold}. We denote by $S$ the corresponding splitting of $G$.

As $\text{rk}_K(G,\mathcal{F})\ge 3$, we can equip $\mathcal{O}$ with a minimal and filling measured foliation. Dual to this foliation is an indecomposable $\pi_1(\mathcal{O})$-tree $Y$ (indecomposability is shown in \cite[Proposition 1.25]{Gui08}). We then form a graph of actions $\mathcal{G}$ over the splitting $S$: vertex trees are the $\pi_1(\mathcal{O})$-tree $Y$, and a trivial $G_i$-tree for all $i\in\{1,\dots,k\}$, and edges have length $0$. We denote by $T$ the $(G,\mathcal{F})$-tree defined in this way. 

\begin{figure}
\begin{center}
\input{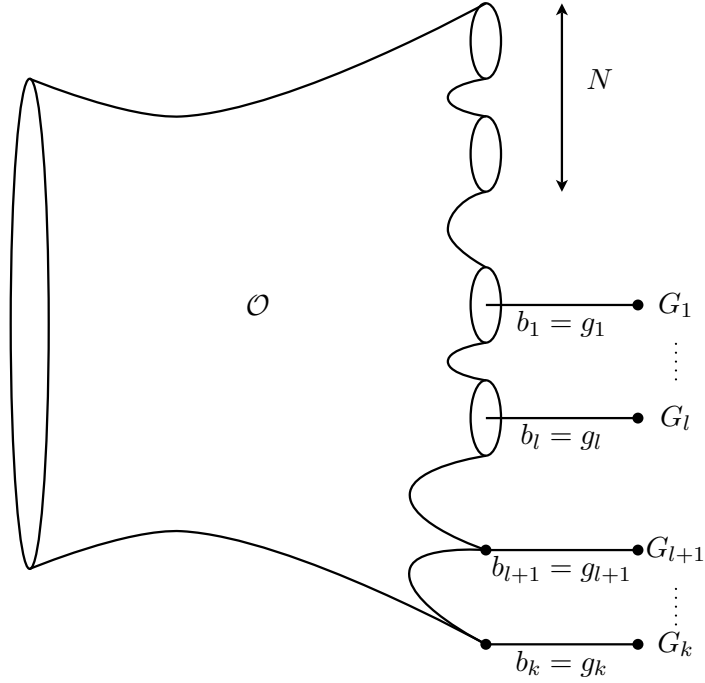}
\caption{The decomposition of $G$ as an amalgam of $\pi_1(\mathcal{O})$ and the $G_i$'s.}
\label{fig-orbifold}
\end{center}
\end{figure}

We claim that $T\in\mathcal{X}(G,\mathcal{F})$. Indeed, the tree $T$ admits a transverse covering by translates of $Y$, so $T$ is mixing. We claim that $T$ is also $\mathcal{Z}$-incompatible, which implies that $T$ is $\mathcal{Z}$-averse by the last assertion of Proposition \ref{mixing-representative}. If $T$ were compatible with a $\mathcal{Z}$-splitting $S'$, then Lemma \ref{lemma-ex} would imply that the stabilizer $\pi_1(\mathcal{O})$ of the indecomposable subtree $Y$ fixes a vertex $v$ in $S'$. Therefore, for all $i\in\{1,\dots,k\}$, the element $b_i=g_i$ fixes $v$. As $S'$ is a $(G,\mathcal{F})$-splitting, the subgroup $G_i$ fixes a vertex $v_i$ in $S'$, and in particular $g_i$ fixes $v_i$. As nontrivial edge stabilizers in $S'$ are nonperipheral, the element $g_i$ does not fix any arc in $S'$, so $v_i=v$. So all subgroups $G_i$ fix the same vertex $v$ of $S'$. Hence $G$ is elliptic in $S'$, a contradiction.   
\end{proof} 

\begin{rk}
When $G=F_N$ with $N\ge 3$, the trees we get are $\mathcal{Z}$-averse trees whose $\sim$-class does not contain any arational tree in the sense of \cite{Rey12}. By comparing our description of the Gromov boundary of $\partial_{\infty} FZ_N$ with Bestvina--Reynolds' and Hamenstädt's description of the Gromov boundary of the free factor graph $FF_N$ as the space of equivalence classes of arational $F_N$-trees \cite{BR13,Ham12}, we get that the natural map from $FZ_N$ to $FF_N$ is not a quasi-isometry (this map is defined by mapping any one-edge free splitting of $F_N$ to one of its vertex groups, and mapping any $\mathcal{Z}$-splitting with nontrivial edge stabilizers to the smallest free factor of $F_N$ that contains the edge group, which is proper by \cite[Lemma 5.11]{Hor14-5}). When $N=2$, it is known that all trees with dense orbits in the boundary $\partial cv_2$ are dual to arational measured foliations on a once-punctured torus, and are therefore arational. So the Gromov boundaries $\partial_{\infty}FF_2$, $\partial_{\infty}FZ_2^{max}$ and $\partial_{\infty}FZ_2$ are all isomorphic.
\end{rk}

\subsection{Proof of the equivalences in the definition of $\mathcal{Z}$-averse trees (Theorem \ref{averse})}\label{sec-averse}

Our proofs of Theorem \ref{averse} and Proposition \ref{mixing-representative} are based on the following two propositions.

\begin{prop}\label{mixing}
Every tame $(G,\mathcal{F})$-tree is either $\mathcal{Z}$-compatible, or collapses to a mixing tree in $\overline{\mathcal{O}(G,\mathcal{F})}$.
\end{prop}

\begin{prop}\label{unique-projection-2}
Let $T_1$ and $T_2$ be tame $(G,\mathcal{F})$-trees. If $T_1$ and $T_2$ are compatible, and if $T_1$ is mixing and $\mathcal{Z}$-incompatible, then there is an alignment-preserving map from $T_2$ to $T_1$.
\end{prop}

\begin{rk}
The analogues of Propositions \ref{mixing} and \ref{unique-projection-2} for $\mathcal{Z}^{max}$-tame trees also hold. The proof of Proposition \ref{mixing} is the same. For Proposition \ref{unique-projection-2}, we will explain how one has to slightly adapt the argument in the proof of Proposition \ref{stab-not-cyclic} to handle the case of $\mathcal{Z}^{max}$-splittings. 
\end{rk}

We first explain how to deduce Theorem \ref{averse} and Proposition \ref{mixing-representative} from Propositions \ref{mixing} and \ref{unique-projection-2}, before proving these two propositions. 

\begin{proof}[Proof of Theorem \ref{averse}]
The implications $(3)\Rightarrow (2)\Rightarrow (1)$ are obvious, so we need only show that $(1)$ implies $(3)$. Let $T$ be a tame $(G,\mathcal{F})$-tree. Assume that there exists a finite sequence $(T=T_0,T_1,\dots,T_k=S)$ of tame $(G,\mathcal{F})$-trees, where $S$ is simplicial, and for all $i\in\{0,\dots,k-1\}$, the trees $T_i$ and $T_{i+1}$ are compatible. If $T$ did not collapse onto a tame $\mathcal{Z}$-compatible tree, then by Proposition \ref{mixing}, the tree $T$ would collapse onto a mixing $\mathcal{Z}$-incompatible tree $\overline{T}\in\overline{\mathcal{O}(G,\mathcal{F})}$. Notice in particular that $T_1$ is compatible with $\overline{T}$. An iterative application of Proposition \ref{unique-projection-2} then implies that all $T_i$'s collapse onto $\overline{T}$ (see Figure \ref{diagram-averse}, where all arrows represent collapse maps). In particular, the $\mathcal{Z}$-splitting $S$ collapses to $\overline{T}$, a contradiction.
\end{proof}

\begin{figure}
\begin{displaymath}
\xymatrix{
&T+T_1\ar@{>>}[dl]\ar@{>>}[dr]&&T_1+T_2\ar@{>>}[dl]\ar@{>>}[dr]&&\dots &&T_{k-1}+S\ar@{>>}[dl]\ar@{>>}[dr]&\\
T\ar@{>>}[dd]&&T_1\ar@{.>>}[ddll]&&T_2\ar@{.>>}[ddllll]&\dots &T_{k-1}\ar@{.>>}[ddllllll]&&S\ar@{.>>}[ddllllllll]\\
&&&&&&&&\\
\overline{T}&&&&&&&&}
\end{displaymath}
\centering
\caption{The situation in the proof of Theorem \ref{averse}.}
\label{diagram-averse}
\end{figure}
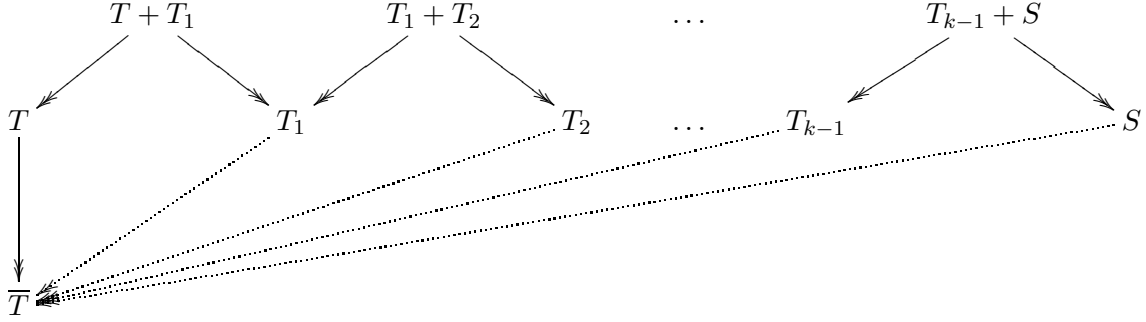

\begin{proof}[Proof of Proposition \ref{mixing-representative}]
The argument is similar to the proof of Theorem \ref{averse}. Let $T,T'\in\mathcal{X}(G,\mathcal{F})$ be two equivalent trees. As $T$ is $\mathcal{Z}$-incompatible, by Proposition \ref{mixing}, it collapses onto a mixing tree $\overline{T}\in\overline{\mathcal{O}(G,\mathcal{F})}$, and $\overline{T}\in\mathcal{X}(G,\mathcal{F})$ because $T\in\mathcal{X}(G,\mathcal{F})$. As $T\sim T'$, there exists a finite sequence $(T=T_0,T_1,\dots,T_{k-1},T_k=T')$ of trees in $\overline{\mathcal{O}(G,\mathcal{F})}$ such that for all $i\in\{0,\dots,k-1\}$, the trees $T_i$ and $T_{i+1}$ are compatible. In particular $T_1$ is compatible with $\overline{T}$. An iterative application of Proposition \ref{unique-projection-2} shows that all $T_i$'s collapse to $\overline{T}$, which proves the first assertion of Proposition \ref{mixing-representative}.

If $\overline{T}_1$ and $\overline{T}_2$ are two mixing trees in $\overline{\mathcal{O}(G,\mathcal{F})}$ that both satisfy the conclusion of Proposition \ref{mixing-representative}, then there is an alignment-preserving map from $T_1$ to $T_2$, and an alignment-preserving map from $T_2$ to $T_1$. As any alignment-preserving map from a tree in $\overline{\mathcal{O}(G,\mathcal{F})}$ with dense orbits to itself is an isometry, this implies that $\overline{T_1}$ and $\overline{T_2}$ are weakly homeomorphic. 

To prove the last assertion of Proposition \ref{mixing-representative}, let $T\in\overline{\mathcal{O}(G,\mathcal{F})}$ be a mixing tree that is not $\mathcal{Z}$-averse. Then there exists a finite sequence $(T=T_0,T_1,\dots,T_k=S)$ of trees in $\overline{\mathcal{O}(G,\mathcal{F})}$, where $S$ is simplicial, and for all $i\in\{0,\dots,k-1\}$, the trees $T_i$ and $T_{i+1}$ are compatible. If $T$ were $\mathcal{Z}$-incompatible, an iterative application of Proposition \ref{unique-projection-2} would imply that $S$ collapses to $T$, a contradiction. So $T$ is $\mathcal{Z}$-compatible.
\end{proof}

\subsubsection{Proof of Proposition \ref{mixing}}\label{sec-mixing}

A \emph{topological $(G,\mathcal{F})$-tree} is a topological space $T$ which is homeomorphic to an $\mathbb{R}$-tree, together with a minimal, bijective, non-nesting (i.e. for all $g\in G$ and all segments $I\subseteq T$, we have $gI\nsubseteq I$), alignment-preserving $G$-action with trivial arc stabilizers and no simplicial arc, with a finite number of orbits of branch points, such that there exists a tree $\widehat{T}\in\overline{\mathcal{O}(G,\mathcal{F})}$ which admits an alignment-preserving map onto $T$ (we recall that a map is \emph{alignment-preserving} if it sends segments onto segments). A topological $(G,\mathcal{F})$-tree $T$ \emph{splits over terminal points} if there exists a subtree $Y\varsubsetneq T$ such that for all $g\in G$, we either have $gY=Y$, or $gY\cap Y=\emptyset$, and $\{\overline{gY}\}_{g\in G}$ is a transverse covering of $T$. Proposition \ref{mixing} is a consequence of the following three propositions. 

\begin{prop} (Guirardel--Levitt \cite{GL14-2})\label{alternative-2}
Every tree $T\in\overline{\mathcal{O}(G,\mathcal{F})}$ with trivial arc stabilizers collapses onto a topological $(G,\mathcal{F})$-tree that is either mixing, or splits over terminal points. 
\end{prop}

\begin{prop} (Guirardel--Levitt \cite{GL14-2}) \label{mixing-metrizable}
Every mixing topological $(G,\mathcal{F})$-tree admits a $G$-invariant metric that turns it into an element of $\overline{\mathcal{O}(G,\mathcal{F})}$.
\end{prop}

\begin{prop}\label{alternative-3}
Let $T\in\overline{\mathcal{O}(G,\mathcal{F})}$ be a $(G,\mathcal{F})$-tree with trivial arc stabilizers. If $T$ collapses onto a topological $(G,\mathcal{F})$-tree which splits over terminal points, then $T$ is compatible with a $(G,\mathcal{F})$-free splitting.
\end{prop}

\begin{proof}[Proof of Proposition \ref{alternative-3}]
Let $T'$ be a topological $(G,\mathcal{F})$-tree which splits over terminal points, and $\pi:T\to T'$ be an alignment-preserving map. Let $Y\varsubsetneq T'$ be a subtree of $T'$, such that for all $g\in G$, we either have $gY=Y$ or $gY\cap Y=\emptyset$, and $\{\overline{gY}\}_{g\in G}$ is a transverse covering of $T'$. The tree $Y$ is not closed, since otherwise, any segment in $T'$ would be covered by finitely many closed disjoint subtrees, which would imply that $Y=T'$. We denote by $H$ the stabilizer of $Y$ in $T'$. Denote by $\{x_1,\dots,x_k\}$ a set of representatives of the orbits of points in $\overline{Y}\smallsetminus Y$. Finiteness of this set comes from the fact that these points are vertices of the skeleton $S$ of the transverse covering $\{\overline{gY}\}_{g\in G}$, and $S$ is a minimal simplicial $(G,\mathcal{F})$-tree by Lemma \ref{skeleton}.

Let $x_{i_1},\dots,x_{i_s}$ be those of the $x_i$'s that do not belong to any $G$-translate of $Y$ (there might not be any such $x_{i_j}$). We claim that the family $\mathcal{Y}$ made of $\{\overline{g\pi^{-1}(Y)}\}_{g\in G}$ and the sets $\{\overline{g\pi^{-1}(x_{i_j})}\}_{g\in G}$ for $j\in\{1,\dots,s\}$ is a transverse covering of $T$. Indeed, this is a transverse family made of closed subtrees of $T$. Let now $I\subseteq T$ be a segment. Then $\pi(I)$ is a segment in $T'$, which is covered by a finite number of translates of $Y$ and of the points $x_{i_j}$. Their $\pi$-preimages provide a covering of $I$ by finitely many subtrees in $\mathcal{Y}$.

We now claim that the skeleton of $\mathcal{Y}$ contains an edge with trivial stabilizer. This will conclude the proof of Proposition \ref{alternative-3}, since the skeleton of any transverse covering of $T$ is compatible with $T$ (Lemma \ref{skeleton}). 

To check the above claim, we first notice that the preimage $\pi^{-1}(Y)$ is not closed (Lemma \ref{projection-closed}). Let $y\in\overline{\pi^{-1}(Y)}\smallsetminus\pi^{-1}(Y)$. There is only one direction at $y$ in $\overline{\pi^{-1}(Y)}$. As $T$ is minimal, there exists a subtree $Y'\neq\overline{\pi^{-1}(Y)}$ in $\mathcal{Y}$ such that $y\in Y'$. The point $y$ is a vertex of the skeleton of $\mathcal{Y}$, and there is an edge $e$ in this skeleton associated to the pair $(\overline{\pi^{-1}(Y)},y)$. We claim that $e$ has trivial stabilizer. Indeed, if $g\in G$ stabilizes $e$, then as $y$ has valence $1$ in $\overline{\pi^{-1}(Y)}$, the element $g$ stabilizes an arc in $\overline{\pi^{-1}(Y)}$. As $T$ has trivial arc stabilizers, this implies that $g$ is the identity of $G$.
\end{proof}

\begin{proof}[Proof of Proposition \ref{mixing}]
Let $T$ be a tame $(G,\mathcal{F})$-tree. If $T$ has trivial arc stabilizers, then the conclusion of Proposition \ref{mixing} is a consequence of Propositions \ref{alternative-2}, \ref{mixing-metrizable} and \ref{alternative-3}. If $T$ contains an arc with nontrivial stabilizer, then $T$ does not have dense orbits. Proposition \ref{Levitt} implies that $T$ projects to a simplicial tree $S$ with cyclic, non-peripheral arc stabilizers, so $T$ is $\mathcal{Z}$-compatible.
\end{proof} 

\subsubsection{Proof of Proposition \ref{unique-projection-2}}\label{sec-proof}

The following proposition gives control over the possible point stabilizers in a tree in $\overline{\mathcal{O}(G,\mathcal{F})}$. It can be deduced from \cite[Proposition 4.4]{Gui08} by noticing that any simple closed curve on a closed $2$-orbifold with boundary provides a $\mathcal{Z}^{max}$-splitting of its fundamental group. 

\begin{prop} (Bestvina--Feighn \cite{BF95}, Guirardel \cite[Proposition 4.4]{Gui08}, Guirardel--Levitt \cite{GL14}) \label{stabilizers}
Let $T$ be a tame $(G,\mathcal{F})$-tree, and let $X\subset T$ be a finite subset of $T$. Then there exists a $\mathcal{Z}^{max}$-splitting in which $\text{Stab}_T(x)$ is elliptic for all $x\in X$.
\end{prop}

\begin{rk}
Knowing the existence of a $\mathcal{Z}$-splitting would be enough if we were only interested in proving the $\mathcal{Z}$-version of Proposition \ref{unique-projection-2}.
\end{rk}

\begin{proof}[Proof of Proposition \ref{unique-projection-2}]
Let $\widehat{T}:=T_1+T_2$. As $T_1$ is $\mathcal{Z}$-incompatible, the tree $\widehat{T}$ has dense orbits. Let $p_1:\widehat{T}\to T_1$ and $p_2:\widehat{T}\to T_2$ be the associated $1$-Lipschitz alignment-preserving maps. Assuming that $p_2$ is not a bijection (otherwise the map $p_1\circ {p_2}^{-1}$ satisfies the conclusion of Proposition \ref{unique-projection-2} and we are done), we can find a point $x\in T_2$ whose $p_2$-preimage in $\widehat{T}$ is a nondegenerate closed subtree $Y$ of $\widehat{T}$. The set $\{gY\}_{g\in G}$ is a transverse family in $\widehat{T}$. 

First assume that $p_1(Y)$ is reduced to a point for all $x\in T_2$, and let $f$ be the map from $T_2$ to $T_1$ that sends any $x\in T_2$ to $p_1(Y)$, with the above notations. We claim that $f$ preserves alignment. Indeed, let $x,z\in T_2$, and $y\in [x,z]$. Then ${p_2}^{-1}(\{x\})$ and ${p_2}^{-1}(\{z\})$ are closed subtrees of $\widehat{T}$, and the bridge in $\widehat{T}$ between them meets ${p_2}^{-1}(\{y\})$. Since $p_1$ preserves alignment, this implies that $f$ preserves alignment, and we are done in this case. 

We now choose $x\in T_2$ so that $p_1(Y)$ is not reduced to a point. The family $\{gp_1(Y)\}_{g\in G}$ is a transverse family made of closed subtrees of $T_1$ (Lemma \ref{projection-closed}). As $T_1$ is mixing, it is a transverse covering of $T_1$. The stabilizer of $p_1(Y)$ in $T_1$ is equal to the stabilizer of $Y$ in $\widehat{T}$, which in turn is also equal to the stabilizer of $x$ in $T'$. Proposition \ref{stabilizers} shows that there exists a $\mathcal{Z}$-splitting in which $\text{Stab}_{T'}(x)$, and hence $\text{Stab}_T(p_1(Y))$, is elliptic. This contradicts the following proposition.
\end{proof}

\begin{prop}\label{stab-not-cyclic}
Let $T\in\overline{\mathcal{O}(G,\mathcal{F})}$ be mixing and $\mathcal{Z}$-incompatible, and let $\mathcal{Y}$ be a transverse covering of $T$. Then for all $Y\in\mathcal{Y}$, the stabilizer $\text{Stab}_T(Y)$ is not elliptic in any $\mathcal{Z}$-splitting.
\end{prop}

\begin{rk}
We warn the reader that the argument in the following proof has to be slightly adapted in the case of $\mathcal{Z}^{max}$-splittings. This will be done in Proposition \ref{stab-nc2} below. The difficulty comes from edges with nonperipheral cyclic stabilizers (not belonging to $\mathcal{Z}^{max}$) in $\mathcal{G}$. 
\end{rk}

\begin{proof}[Proof of Proposition \ref{stab-not-cyclic}]
As $T$ is mixing, any transverse covering of $T$ contains at most one orbit of subtrees (otherwise a segment contained in one of these orbits could not be covered by translates of a segment contained in another subtree). We denote by $S$ the skeleton of $\mathcal{Y}$, and by $\Gamma:=S/G$ the quotient graph of groups. The vertex set of $\Gamma$ consists of a vertex associated to $Y$, with vertex group $\text{Stab}_T(Y)$, together with a finite collection of points $x_1,\dots,x_l$. Each $x_i$ is joined to $Y$ by a finite set of edges, whose stabilizers do not belong to the class $\mathcal{Z}$ because $T$ is $\mathcal{Z}$-incompatible. We denote by $G_Y$ and $G_{x_i}$ the corresponding stabilizers. Assume towards a contradiction that $G_Y$ fixes a vertex $v$ in a $\mathcal{Z}$-splitting $S_0$.

Suppose first that all vertex groups of $S$ are elliptic in $S_0$. As edge stabilizers of $S$ do not belong to $\mathcal{Z}$, and as all vertex stabilizers of $S$ fix a point in $S_0$, two adjacent vertex stabilizers of $S$ must have the same fixed point in $S_0$. This implies that $G$ is elliptic in $S_0$, a contradiction. 

Hence one of the $G_{x_i}$'s acts nontrivially on $S_0$. Edge groups of $S$ are elliptic in $S_0$ because $\text{Stab}_T(Y)$ is. By blowing up $S$ at the vertex $x_i$, using the action of $G_{x_i}$ on its minimal subtree in $S_0$, we get a splitting $S'$, which contains an edge whose stabilizer belongs to the class $\mathcal{Z}$. The tree $T$ splits as a graph of actions over $S'$ (by the discussion following Proposition \ref{skeleton} in Section \ref{sec-goa}). This contradicts $\mathcal{Z}$-incompatibility of $T$.  
\end{proof}

\begin{figure}
\begin{center}
\input{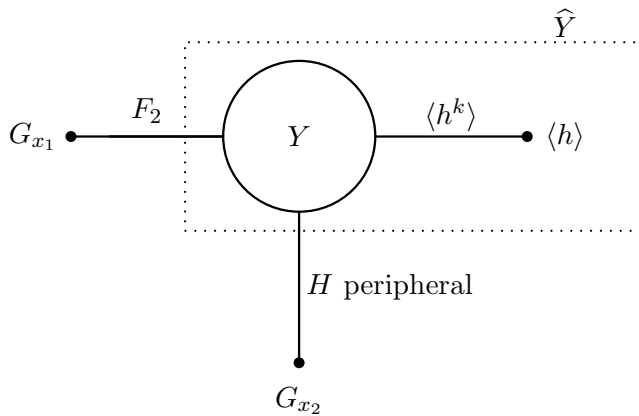}
\caption{The graph of actions in the proof of Proposition \ref{stab-nc2}.}
\label{fig-skeleton}
\end{center}
\end{figure}

\begin{prop}\label{stab-nc2}
Let $T\in\overline{\mathcal{O}(G,\mathcal{F})}$ be mixing and $\mathcal{Z}^{max}$-incompatible, and let $\mathcal{Y}$ be a transverse covering of $T$. Then for all $Y\in\mathcal{Y}$, the stabilizer $\text{Stab}_T(Y)$ is not elliptic in any $\mathcal{Z}^{max}$-splitting.
\end{prop}

\begin{proof}
We keep the notations from the proof of Proposition \ref{stab-not-cyclic}, where this time $S_0$ is a $\mathcal{Z}^{max}$-splitting. We denote by $\mathcal{G}$ the graph of actions corresponding to $\mathcal{Y}$, which is represented on Figure \ref{fig-skeleton}. Note that for all $i\in\{1,\dots,l\}$, and all nonperipheral elements $g\in G$, if $g^p\in G_{x_i}$ for some $p\ge 1$, then $g\in G_{x_i}$. Up to reordering the $x_i$'s, we can assume that for all $i\in\{1,\dots,k\}$, no edge joining $Y$ to $x_i$ has $\mathcal{Z}$-stabilizer, and for all $i\in\{k+1,\dots,l\}$, there is an edge $e_i$ with $\mathcal{Z}\smallsetminus\mathcal{Z}^{max}$-stabilizer $\langle g_i\rangle$ joining $x_i$ to $Y$. Then $g_i$ is a proper power of the form $h_i^k$, with $h_i\in G_{x_i}$. Subdivide the edge $[x_i,Y]$ into $[x_i,m_i]\cup [m_i,Y]$, and fold $[x_i,m_i]$ with its image under $h_i$. We get a refinement $S'$ of $S$ that is still compatible with $T$. By collapsing the orbit of the edge with stabilizer equal to $h_i^k$, we get a new splitting of $T$ as a graph of actions $\mathcal{G}'$. The stabilizer of the edge $e'_i$ joining $Y'$ to $x_i$ in $\mathcal{G}'$ is equal to $\langle h_i\rangle$, and hence it belongs to the class $\mathcal{Z}^{max}$. The splitting of $G$ dual to the edge $e'_i$ is not minimal, otherwise $T$ would be $\mathcal{Z}^{max}$-compatible. Hence $G_{x_i}=\langle h_i\rangle$, and $x_i$ is joined to $Y$ by a single edge in $\mathcal{G}$. As $G_Y$ fixes the vertex $v$ of $S_0$, so does $g_i$. As $S_0$ is a $\mathcal{Z}^{max}$-splitting, the element $h_i$, and hence $G_{x_i}$, also fixes $v$ in $S_0$. 

Therefore, by replacing $G_Y$ by $G_{\widehat{Y}}:=\langle G_Y,G_{x_{k+1}},\dots,G_{x_l}\rangle$, we build a new splitting of $T$ as a graph of actions $\widehat{\mathcal{G}}$, which has the following description. The graph of actions $\widehat{\mathcal{G}}$ consists of a new vertex tree $\widehat{Y}$ with dense orbits, whose stabilizer $G_{\widehat{Y}}$ is elliptic in $S_0$, attached to $x_1,\dots,x_k$, and all its edges have either peripheral or noncyclic stabilizer. The proof then goes as in the case of $\mathcal{Z}$-splittings, by working with the graph of actions $\widehat{\mathcal{G}}$ instead of $\mathcal{G}$.
\end{proof}

Using Proposition \ref{stab-nc2}, we deduce the following $\mathcal{Z}^{max}$-analogue of Proposition \ref{unique-projection-2}. 

\begin{prop}\label{unique-projection-2max}
Let $T_1$ and $T_2$ be $\mathcal{Z}^{max}$-tame $(G,\mathcal{F})$-trees. If $T_1$ and $T_2$ are compatible, and if $T_1$ is mixing and $\mathcal{Z}^{max}$-incompatible, then there is an alignment-preserving map from $T_2$ to $T_1$.
\end{prop}

\subsection{Folding paths ending at mixing and $\mathcal{Z}$-incompatible trees}

We now prove the following property for folding paths ending at mixing $\mathcal{Z}$-incompatible trees.

\begin{prop} \label{folding-uncollapsible}
Let $T\in\overline{\mathcal{O}(G,\mathcal{F})}$ be mixing and $\mathcal{Z}$-incompatible, and let $\gamma:[0,L]\to\overline{\mathcal{O}(G,\mathcal{F})}$ be an optimal liberal folding path ending at $T$. Then for all $t<L$, the tree $\gamma(t)$ is simplicial and has trivial edge stabilizers unless $\gamma(t)=T$. 
\end{prop}

\begin{proof}
As $T$ has dense orbits, all arc stabilizers in $T$ are trivial, hence all arc stabilizers in trees lying on optimal liberal folding paths ending at $T$ are trivial. Assume towards a contradiction that there exists $t_0<L$, such that $\gamma(t_0)\neq T$ is nonsimplicial. Notice that $\gamma(t_0)$ contains a nontrivial simplicial part, otherwise it would be equal to $T$, as any morphism between two $(G,\mathcal{F})$-trees with dense orbits is an isometry (Corollary \ref{alignment-preserved}). By Proposition \ref{Levitt}, the tree $\gamma(t_0)$ contains a subtree $T_0$ which has dense orbits for the action of its stabilizer $H$. Moreover, the group $H$ is a proper $(G,\mathcal{F})$-free factor.  

Let $Y:=f_{t_0,L}(T_0)$. We claim that for all $g\in G\smallsetminus H$, the intersection $gY\cap Y$ contains at most one point. Otherwise, there exist nondegenerate segments $I\subset T_0$ and $J\subset gT_0$ such that $f_{t_0,L}(I)=f_{t_0,L}(J)$. So there exist $h\in H$ hyperbolic in $T_0$ (whose axis intersects $I$ nondegenerately), and $h'\in H^g$ hyperbolic in $gT_0$ (whose axis intersects $J$ nondegenerately), such that the axes of $h$ and $h'$ have nondegenerate intersection in $T$. We thus have $||hh'||_{T}\le ||h||_{T}+||h'||_{T}$ (see \cite[1.8]{CM87}). Let $t_1$ be the smallest real number for which this inequality holds, so that for all $t<t_1$, the axes of $h$ and $h'$ are disjoint in $\gamma(t)$. By continuity of $\gamma$, we deduce that both $||hh'||_{\gamma(t_1)}$ and $||hh'^{-1}||_{\gamma(t_1)}$ are greater than or equal to $||h||_{\gamma(t_1)}+||h'||_{\gamma(t_1)}$, so the intersection of the axes of $h$ and $h'$ in $f_{t_0,t_1}(T_0)$ is reduced to a point. The image $f_{t_0,t_1}(I\cup J)$ is contained in a subtree with dense orbits of the Levitt decomposition of $\gamma(t_1)$ as a graph of actions given by Proposition \ref{Levitt}. The morphism $f_{t_1,L}$ is injective in restriction to this subtree (Corollary \ref{alignment-preserved}). This implies that $gY\cap Y$ is reduced to a point.  
 
Hence the collection $\{gY\}_{g\in G}$ is a transverse family in $T$, and so is the collection $\{g\overline{Y}\}_{g\in G}$. As $T$ is mixing, the collection $\{g\overline{Y}\}_{g\in G}$ is a transverse covering of $T$. In addition, the stabilizer of $\overline{Y}$ in $T$ is equal to $H$, and hence is elliptic in a $\mathcal{Z}$-splitting (it is even a $(G,\mathcal{F})$-free factor). This contradicts Proposition \ref{stab-not-cyclic}. 
\end{proof}

\subsection{The case of $\mathcal{Z}^{max}$-splittings}\label{sec-max}

By only considering $\mathcal{Z}^{max}$-splittings, we similarly define the space $\mathcal{X}^{max}(G,\mathcal{F})$ of \emph{$\mathcal{Z}^{max}$-averse trees} in the following way. For all $\mathcal{Z}^{max}$-tame trees $T$, we denote by $\mathcal{R}^{1,max}(T)$ the set of $\mathcal{Z}^{max}$-splittings that are compatible with $T$, and by $\mathcal{R}^{2,max}(T)$ the set of $\mathcal{Z}^{max}$-splittings that are compatible with a $\mathcal{Z}^{max}$-tame tree $T'$, which is compatible with $T$. A $\mathcal{Z}^{max}$-tame tree $T$ is \emph{$\mathcal{Z}^{max}$-averse} if $\mathcal{R}^{2,max}(T)=\emptyset$. Two $\mathcal{Z}^{max}$-averse trees $T,T'\in\mathcal{X}^{max}(G,\mathcal{F})$ are \emph{equivalent} if there exists a finite sequence $(T=T_0,T_1,\dots,T_k=T')$ of tame $(G,\mathcal{F})$-trees such that for all $i\in\{1,\dots,k\}$, the trees $T_i$ and $T_{i+1}$ are compatible. 

The analogues of Theorem \ref{averse}, Proposition \ref{mixing-representative}, and Theorem \ref{Luo} also hold true in this setting. The proofs are the same, the only difference is in the proof of Proposition \ref{unique-projection-2max}, as explained above.

\begin{theo}($\mathcal{Z}^{max}$-analogue of Theorem \ref{averse}) \label{maxi-averse}
For all $\mathcal{Z}^{max}$-tame $(G,\mathcal{F})$-trees $T$, the following assertions are equivalent.

\begin{enumerate}
\item There exists a finite sequence $(T=T_0,T_1,\dots,T_k=S)$ of $\mathcal{Z}^{max}$-tame $(G,\mathcal{F})$-trees, such that $S$ is simplicial, and for all $i\in\{0,\dots,k-1\}$, the trees $T_i$ and $T_{i+1}$ are compatible.
\item We have $\mathcal{R}^{2,max}(T)\neq\emptyset$.
\item The tree $T$ collapses to a $\mathcal{Z}^{max}$-tame $\mathcal{Z}^{max}$-compatible $(G,\mathcal{F})$-tree.
\end{enumerate}
\end{theo}

\begin{prop}($\mathcal{Z}^{max}$-analogue of Proposition \ref{mixing-representative})\label{mixing-rep-max}
For all $T\in\mathcal{X}^{max}(G,\mathcal{F})$, there exists a mixing tree in $\mathcal{X}^{max}(G,\mathcal{F})$ onto which all trees $T'\in\mathcal{X}^{max}(G,\mathcal{F})$ that are equivalent to $T$ collapse. In addition, any two such trees are $G$-equivariantly weakly homeomorphic. Any tree $T\in\overline{\mathcal{O}(G,\mathcal{F})}$ that is both mixing and $\mathcal{Z}^{max}$-incompatible is $\mathcal{Z}^{max}$-averse. 
\end{prop}

\begin{theo}\label{max-Luo}($\mathcal{Z}^{max}$-analogue of Theorem \ref{Luo})
Let $T\in\mathcal{X}^{max}(G,\mathcal{F})$, and let $(T_i)_{i\in\mathbb{N}}\in\mathcal{O}(G,\mathcal{F})^{\mathbb{N}}$ be a sequence that converges to $T$. Then $\psi^{max}(T_i)$ is unbounded in $FZ^{max}(G,\mathcal{F})$.
\end{theo}

\subsection{A few remarks and examples}

\subsubsection{$\mathcal{Z}$-averse trees versus $\mathcal{Z}^{max}$-averse trees}\label{sec-ex-max}

Building on our construction from the proof of Proposition \ref{ex-unbounded}, we give examples of $\mathcal{Z}^{max}$-averse trees that are not $\mathcal{Z}$-averse as soon as $\text{rk}_f(G,\mathcal{F})\ge 1$ and $\text{rk}_K(G,\mathcal{F})\ge 3$. Together with our main results (Theorem \ref{main} and \ref{main-2}), this implies that the inclusion map from $FZ^{max}(G,\mathcal{F})$ into $FZ(G,\mathcal{F})$ is not a quasi-isometry in these cases.

\begin{prop}
Let $G$ be a countable group, and let $\mathcal{F}$ be a free factor system of $G$. Assume that $\text{rk}_f(G,\mathcal{F})\ge 1$ and $\text{rk}_K(G,\mathcal{F})\ge 3$. Then $\mathcal{X}^{max}(G,\mathcal{F})\neq\mathcal{X}(G,\mathcal{F})$, so the inclusion map from $FZ^{max}(G,\mathcal{F})$ into $FZ(G,\mathcal{F})$ is not a quasi-isometry.
\end{prop}

\begin{proof}
For all $i\in\{1,\dots,k\}$, we choose an element $g_i\in G_i\smallsetminus\{e\}$, whose order we denote by $p_i\in\mathbb{N}\cup\{+\infty\}$. We denote by $l$ the number of indices such that $p_i=+\infty$. Up to reindexing the $g_i$'s, we can assume that $p_1,\dots,p_l=+\infty$, and $p_{l+1},\dots,p_k<+\infty$. 

Let $\mathcal{O}$ be the orbifold obtained from a sphere with $N+l+1$ boundary components, where $N:=\text{rk}_f(G,\mathcal{F})\ge 1$, by adding a conical point of order $p_i$ for each $i\in\{l+1,\dots,k\}$. As $\text{rk}_K(G,\mathcal{F})\ge 3$, we can equip $\mathcal{O}$ with an arational measured foliation. For all $i\in\{1,\dots,l\}$, we denote by $b_i$ a generator of the $i^{th}$ boundary curve in $\pi_1(\mathcal{O})$, and for all $i\in\{l+1,\dots,k\}$, we denote by $b_i$ a generator of the subgroup associated to the corresponding conical point. We denote by $b_0$ a generator of one of the other boundary curves. The group $G$ is isomorphic to the group obtained by amalgamating $\pi_1(\mathcal{O})$ with the groups $G_i$ and $\mathbb{Z}=\langle a_0\rangle$, identifying $b_i$ with $g_i$ for all $i\in\{1,\dots,k\}$, and identifying $b_0$ with $a_0^2$. 

We then form a graph of actions $\mathcal{G}$ over this splitting of $G$: vertex trees are the $\pi_1(\mathcal{O})$-tree $Y$ dual to the foliation on $\mathcal{O}$, a trivial $G_i$-tree for all $i\in\{1,\dots,k\}$, and a trivial $\langle a_0\rangle$-tree, and edges have length $0$.
 
This construction yields a $G$-tree $T$ which is not $\mathcal{Z}$-averse, because it splits as a graph of actions, one of whose edge groups belongs to $\mathcal{Z}$. We claim that $T$ is $\mathcal{Z}^{max}$-averse. Indeed, as $T$ is mixing, it is enough to show that $T$ is $\mathcal{Z}^{max}$-incompatible (Proposition \ref{mixing-rep-max}). Assume towards a contradiction that $T$ is compatible with a $\mathcal{Z}^{max}$-splitting $S_0$. The $\pi_1(\mathcal{O})$-minimal subtree of $T$ is indecomposable, so $\pi_1(\mathcal{O})$ has to be elliptic in $S_0$. We denote by $S$ the skeleton of $\mathcal{G}$, and by $x_0$ the vertex of $S$ with vertex group $\langle a_0\rangle$. Arguing as in the proof of Proposition \ref{ex-unbounded}, we then get that for any two adjacent vertices $u,u'\in S$ with $u,u'\notin G.x_0$, the vertex groups $G_u$ and $G_{u'}$ fix a common vertex $v$ of $S_0$. This is still true if $u'\in G.x_0$ because $S_0$ is a $\mathcal{Z}^{max}$-splitting. 
\end{proof}

On the other hand, we show that if $\text{rk}_f(G,\mathcal{F})=0$, then the graphs $FZ(G,\mathcal{F})$ and $FZ^{max}(G,\mathcal{F})$ are quasi-isometric to each other.

\begin{prop}
Let $G$ be a countable group, and let $\mathcal{F}$ be a free factor system of $G$. Assume that $\text{rk}_f(G,\mathcal{F})=0$, and $\text{rk}_K(G,\mathcal{F})\ge 3$. Then the inclusion from $FZ^{max}(G,\mathcal{F})$ into $FZ(G,\mathcal{F})$ is a quasi-isometry.
\end{prop}

\begin{proof}
We will define an inverse map $\tau:FZ(G,\mathcal{F})\to FZ^{max}(G,\mathcal{F})$. Let $S$ be a one-edge $\mathcal{Z}$-splitting, of the form $A\ast_{\langle g^k\rangle}B$ (where $\langle g\rangle\in\mathcal{Z}^{max}$). As $g^k$ is elliptic in $S$, so is $g$. We assume without loss of generality that $g\in A$ (and $g\notin B$), and we let $S^{max}$ be the $\mathcal{Z}^{max}$-splitting $A\ast_{\langle g\rangle}\langle B,g\rangle$. In the case where $S$ is an HNN extension of the form $A\ast_{\langle g^k\rangle}$, we let $S^{max}:=\langle A,g,g^t\rangle\ast_{\langle g\rangle}$, where $t$ is a stable letter. We claim that the $G$-minimal subtree of $S^{max}$ is nontrivial. The map $\tau$ is then defined by letting $\tau(S):=S^{max}$. In addition, if $S_1$ and $S_2$ are two compatible one-edge $\mathcal{Z}$-splittings, one checks that $S_1^{max}$ and $S_2^{max}$ are also compatible. This shows that $\tau$ is Lipschitz, and proves that $FZ(G,\mathcal{F})$ and $FZ^{max}(G,\mathcal{F})$ are quasi-isometric to each other. 

Assume towards a contradiction that $S^{max}$ is trivial. Then $S$ is of the form $\langle g\rangle\ast_{\langle g^k\rangle} B$, so $A=\langle g\rangle$ is cyclic and nonperipheral. We claim that $A$ is a proper $(G,\mathcal{F})$-free factor. This is a contradiction because $G$ has no free factor in $\mathcal{Z}$, since $\text{rk}_f(G,\mathcal{F})=0$. 

By \cite[Lemma 5.11]{Hor14-5}, the splitting $S$ is compatible with a one-edge $(G,\mathcal{F})$-free splitting $S_0$. Since $A$ is cyclic and $g^k\in A$ is elliptic in $S$, the splitting $S+S_0$ can only be obtained by splitting the vertex group $B$ in $S$. Some proper $(G,\mathcal{F})$-free factor $B'$ of $B$ is elliptic in $S+S_0$ and contains $g^k$. Repeating the above argument, we can split $B'$ further. Arguing by induction on the Kurosh rank of $B$, we end up with a $\mathcal{Z}$-splitting $S'$ in which the edge with nontrivial stabilizer $\langle g^k\rangle$ is attached to a vertex whose stabilizer has Kurosh rank equal to $1$, and is therefore equal to $\langle g^k\rangle$. The splitting $S'$ collapses to a $(G,\mathcal{F})$-free splitting in which $A$ is elliptic.
\end{proof}

\subsubsection{Why working with $\mathcal{R}^2(T)$ instead of $\mathcal{R}^1(T)$ ?}

\begin{ex}\label{R-empty}
We give an example of a tree $T\in\overline{cv_N}$ that is $\mathcal{Z}$-incompatible but is not $\mathcal{Z}$-averse. In other words, we have $\mathcal{R}^1(T)=\emptyset$, while $\mathcal{R}^2(T)\neq\emptyset$. This justifies the introduction of the set $\mathcal{R}^2(T)$ in our arguments. 

Let $T_1$ be an indecomposable $F_N$-tree in which some free factor $F_2\subseteq F_N$ of rank $2$ fixes a point $x_1$. Examples of such trees were given in \cite[Part 11.6]{Rey12}. Form a graph of actions over the splitting $F_{2N-2}=F_N\ast_{F_2} F_N$, where the vertex trees are two copies of $T_1$, and the attaching points are the copies of $x_1$. In this way, we get a tree $T\in\overline{cv_{2N-2}}$. 

We claim that $T$ is $\mathcal{Z}$-incompatible. Indeed, assume towards a contradiction that $T$ is compatible with a $\mathcal{Z}$-splitting $S$ of $F_{2N-2}$. Lemma \ref{lemma-ex} implies that both copies of $F_N$ are elliptic in $S$. Therefore, the subgroup $F_2$ is also elliptic in $S$. As edge stabilizers in $S$ are cyclic, this implies that $F_{2N-2}$ is elliptic in $S$, a contradiction. 

However, the tree $T$ is not $\mathcal{Z}$-averse. Indeed, let $\overline{T}$ be the tree obtained by equivariantly collapsing to a point one of the copies of $T_1$ in $T$ (but not the other). Then $\overline{T}$ is $\mathcal{Z}$-compatible, because one can blow up the copy of $F_N$ that got collapsed by using a splitting in which the free factor $F_2$ is elliptic.
\end{ex}

\subsubsection{The importance of working with cyclic splittings rather than free splittings}

\begin{ex}\label{surfaces}
We now give an example of two mixing compatible $F_N$-trees $T_1,T_2\in\overline{cv_N}$, such that $T_2$ is compatible with a free splitting of $F_N$, while $T_1$ is not. This shows that it is crucial to work with cyclic splittings rather than free splittings in Theorem \ref{averse}. The following situation is illustrated on Figure \ref{fig-surfaces}. Let $S$ be a compact orientable surface of genus $2$, with one boundary component. Let $c$ be a simple closed curve that splits the surface $S$ into two subsurfaces $S_1$ and $S_2$, where $S_1$ has genus $1$ and two boundary components, and $S_2$ has genus $1$ and one boundary component.

For all $i\in\{1,2\}$, let $L_i$ be an arational measured lamination on the surface $S_i$. Let $T_1$ be the tree dual to the measured lamination on $S$ obtained by equipping $S_1$ with the lamination $L_1$, and equipping $S_2$ with the empty lamination. Let $T_2$ be the tree dual to the measured lamination on $S$ obtained by equipping $S_1$ with the empty lamination, and equipping $S_2$ with the lamination $L_2$. Both trees $T_1$ and $T_2$ are mixing.  The trees $T_1$ and $T_2$ are compatible, as they are both refined by the tree $T$ dual to the lamination obtained by equipping $S_1$ with $L_1$ and $S_2$ with $L_2$. The tree $T_2$ is compatible with any free splitting of $F_4$ determined by an essential arc of $S$ that lies on the subsurface $S_1$. However, the tree $T_1$ is not compatible with any free splitting of $S$. Indeed, otherwise, Lemma \ref{lemma-ex} would imply that the boundary curve of $S$ is elliptic in this splitting, which is impossible. Notice however that both trees $T_1$ and $T_2$ are compatible with the $\mathcal{Z}^{max}$-splitting determined by the simple closed curve $c$.     
\end{ex}

\begin{figure}
\begin{center}
\input{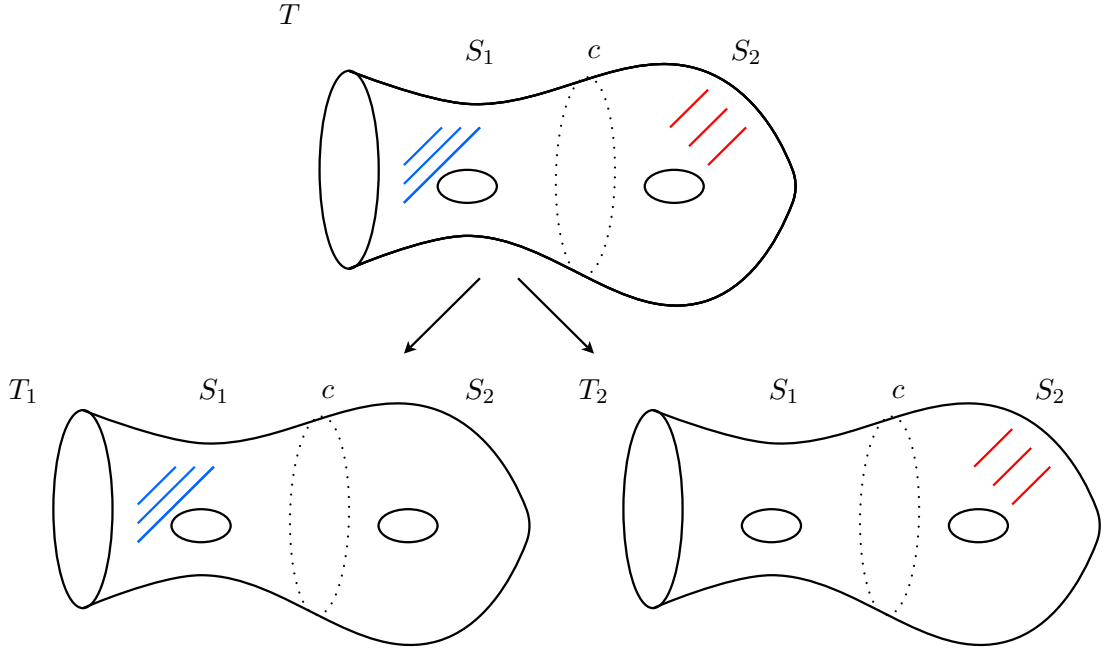}
\caption{The laminations dual to the trees in Example \ref{surfaces}.}
\label{fig-surfaces}
\end{center}
\end{figure}

\subsubsection{Non-mixing $\mathcal{Z}$-averse trees}

\begin{ex}
We have seen (Corollary \ref{mixing-representative}) that any equivalence class in $\mathcal{X}(G,\mathcal{F})$ contains mixing representatives. We now give an example of a tree $T\in\overline{cv_N}$ that is $\mathcal{Z}$-averse but not mixing. We refer to \cite[Example 10.10]{Rey11} for details. Let $\Phi\in\text{Out}(F_N)$ be an automorphism with two strata, and assume that the Perron--Frobenius eigenvalue of the lower stratum is strictly greater than the Perron--Frobenius eigenvalue of the upper stratum. Then the attractive tree of $\Phi$ is not mixing, however it collapses onto a tree which is mixing and $\mathcal{Z}$-incompatible, and hence it is $\mathcal{Z}$-averse.
\end{ex}

\section{Collapses and pullbacks of folding paths and folding sequences}\label{sec-coll-pull}

We now describe two constructions that will turn out to be useful in the next section, for the proof of Theorem \ref{canonical-splittings}. These constructions are inspired from the analogous constructions in \cite[Section 4.2]{HM12} or \cite[Section A.2]{BF13} in the case of folding paths between simplicial $F_N$-trees with trivial edge stabilizers.

\subsection{Collapses}\label{sec-collapse}

In this section, we will present a construction for proving the following proposition.

\begin{prop}\label{prop-coll}
Let $S,T$ and $\overline{T}$ be tame $(G,\mathcal{F})$-trees. Let $\gamma$ be a tame optimal liberal folding path from $S$ to $T$. Assume that there exists a $1$-Lipschitz alignment-preserving map $\pi:T\to\overline{T}$. Then there exists a tame optimal liberal folding path $\overline{\gamma}$ ending at $\overline{T}$ such that for all $t\in\mathbb{R}_+$, there exists a $1$-Lipschitz alignment-preserving map from $\gamma(t)$ to $\overline{\gamma}(t)$.
\end{prop}

Let $L\in\mathbb{R}_+$ be such that $\gamma(L)=T$, and let $t\le L$. Recall from Proposition \ref{Levitt} that $\gamma(t)$ splits as a graph of actions $\mathcal{G}$, all of whose vertex trees have dense orbits for the action of their stabilizer (some vertex trees may be trivial). If $Y\subseteq \gamma(t)$ is one of the vertex trees of this splitting, then the morphism $f_{t,L}:\gamma(t)\to T$ provided by the definition of a liberal folding path is an isometry in restriction to $Y$. Optimality of $f$ implies that the morphism $f_{t,L}$ is an isometry in restriction to each edge in the simplicial part of $\gamma(t)$. We define $\overline{\gamma}(t)$ from $\gamma(t)$ by replacing in $\mathcal{G}$ each vertex subtree $Y$ by its image $\pi\circ f_{t,L}(Y)$ in $\overline{T}$, replacing each attaching point $x$ in $\mathcal{G}$ by its image $\pi\circ f_{t,L}(x)$, and modifying the metric on each edge $e$ in the simplicial part of $\gamma(t)$, so that $e$ becomes isometric to the segment $\pi(f_{t,L}(e))\subseteq\overline{T}$ (this may collapse some subsegments of $e$). We denote by $\pi_t:\gamma(t)\to\overline{\gamma}(t)$ the natural alignment-preserving map. The $\pi_t$-preimage of any point $x\in\overline{\gamma}(t)$ is a subtree of $\gamma(t)$ whose $f$-image in $T$ is collapsed to a point by $\pi$. Therefore, for all $t'>t$, the $f_{t,t'}$-image of the subtree $\pi_t^{-1}(x)$ collapses to a point $x'$ in $\overline{\gamma}(t')$. Therefore, the optimal morphism $f_{t,t'}$ induces a map $\overline{f}_{t,t'}$ (sending $x$ to $x'$, with the above notations) for all $t<t'$, and this map is again a morphism by construction. 

\begin{lemma}
For all $t<t'$, the morphism $\overline{f}_{t,t'}$ is optimal.
\end{lemma}

\begin{proof}
Let $x\in\overline{\gamma}(t)$. We want to prove that $\overline{f}_{t,t'}$ has at least two gates at $x$. If $x\in\overline{\gamma}(t)$ belongs to the interior of one of the nondegenerate subtrees with dense orbits $T_v$ of the Levitt decomposition of $\overline{\gamma}(t)$ as a graph of actions, then any line passing through $x$ and contained in $T_v$ lifts to a legal line in $\gamma(t)$. 

We now assume that $x\in\overline{\gamma}(t)$ is contained in a simplicial edge of $\overline{\gamma}(t)$. We claim that for any $\pi_t$-preimage $\widetilde{x}$ of $x$, one can find two legal axes $l_1$ and $l_2$, such that there exists a legal line in $\overline{\gamma}(t)$ which is the concatenation of a half-line in $l_1$, a legal segment containing $\widetilde{x}$, and a half-line in $l_2$. Indeed, given any direction $d$ at $\widetilde{x}$, optimality of $f_{t,t'}$ ensures the existence of a legal half-line $l$ starting from $\widetilde{x}$ and contained in $d$. If $l$ enters a subtree $H$ in the dense orbits part of $\gamma(t)$, then one can ensure that there is a half-line in $l$ contained in an axis for an element of the stabilizer of $H$, and this axis is legal. If $l$ entirely stays in the simplicial part of $\gamma(t)$, then the construction of $l$ may be done so that $l$ crosses the same orbit of turns in $\gamma(t)$ twice. Hence we can assume that $l$ is eventually periodic, i.e. some half-line of $l$ is contained in a legal axis.

If either $\widetilde{x}$ is contained in the interior of $\widetilde{e}$, or if it is contained in an axis $l_i$ (for some $i\in\{1,2\}$) that does not get collapsed to a point by $\pi_t$, then there is a legal turn at $\widetilde{x}$ which projects to a legal turn at $x$. Otherwise, the point $x$ has nontrivial stabilizer (and $\widetilde{x}$ is contained in $l_i$ for some $i\in\{1,2\}$). This also implies that the morphism $\overline{f}_{t,t'}$ has at least two gates at $x$, since otherwise the stabilizer of $l_i$ would fix a nondegenerate arc in $\overline{T}$. Hence it would also fix a nondegenerate arc in $T$, contradicting the fact that $l_i$ is a legal axis. 
\end{proof}

The morphisms $\overline{f}_{t,t'}$ again satisfy $\overline{f}_{t,t''}=\overline{f}_{t',t''}\circ \overline{f}_{t,t'}$ for all $t<t'<t''$ by construction. We call the folding path constructed in this way the \emph{collapse} of $\gamma$ induced by $\pi$.

\begin{prop}\label{collapse}
Let $S,T$ and $\overline{T}$ be tame $(G,\mathcal{F})$-trees, and $\pi:T\to\overline{T}$ be a $1$-Lipschitz alignment-preserving map. Let $\gamma$ be a tame optimal liberal folding path from $S$ to $T$, and let $\overline{\gamma}$ be the collapse of $\gamma$ induced by $\pi$. Then $t\mapsto\overline{\gamma}(t)$ is continuous.
\end{prop}

\begin{proof}
Since $\pi$ is $1$-Lipschitz, all alignment-preserving maps $\pi_t:\gamma(t)\to\overline{\gamma}(t)$ are $1$-Lipschitz. Let $g\in G$, and let $\epsilon>0$. For $t$ close enough to $t_0$, the total length in a fundamental domain of $g$ that gets folded under the morphism $f_{t,t_0}$ (or $f_{t_0,t}$) between time $t$ and time $t_0$ is at most $\epsilon$. As $\pi_t$ is $1$-Lipschitz, this implies that the total length in a fundamental domain of $g$ that gets collapsed under $\pi_t$ is close to the total length in a fundamental domain of $g$ that gets collapsed under $\pi_{t_0}$. More precisely, for $t$ close enough to $t_0$, we have $|(||g||_{\gamma(t)}-||g||_{\overline{\gamma}(t)})-(||g||_{\gamma(t_0)}-||g||_{\overline{\gamma}(t_0)})|\le 2\epsilon$. This implies that $||g||_{\overline{\gamma}(t)}$ converges to $||g||_{\overline{\gamma}(t_0)}$ as $t$ tends to $t_0$. As this is true for all $g\in G$, the collapse $t\mapsto\overline{\gamma}(t)$ is continuous at $t_0$.
\end{proof}

This finishes the proof of Proposition \ref{prop-coll}. One can also give a discrete version of the above construction. We recall the definition of a tame optimal folding sequence from Section \ref{sec-good}. The following proposition follows from the above analysis.

\begin{prop}\label{def-collapse}
Let $T$ and $\overline{T}$ be tame $(G,\mathcal{F})$-trees, and $\pi:T\to\overline{T}$ be a $1$-Lipschitz alignment-preserving map. Let $(\gamma(n))_{n\in\mathbb{N}}$ be a tame optimal folding sequence ending at $T$. Then there exists a tame optimal folding sequence $(\overline{\gamma}(n))_{n\in\mathbb{N}}$ ending at $\overline{T}$, such that for all $n\in\mathbb{N}$, there exists an alignment-preserving map from $\gamma(n)$ to $\overline{\gamma}(n)$. 
\qed
\end{prop}

A sequence $(\overline{\gamma}(n))_{n\in\mathbb{N}}$ satisfying the conclusions of Proposition \ref{def-collapse} is called a \emph{collapse} of $(\gamma(n))_{n\in\mathbb{N}}$ induced by $\pi$.

\subsection{Pullbacks}\label{sec-pullbacks}

\subsubsection{Pullbacks of tame $(G,\mathcal{F})$-trees}

The following construction is inspired from the construction of pullbacks of simplicial metric trees, as it appears in \cite[Proposition 4.4]{HM12} or \cite[Lemma A.3]{BF13}. Let $S,T$ and $\widehat{T}$ be tame $(G,\mathcal{F})$-trees, such that there exists a $1$-Lipschitz alignment-preserving map $p:\widehat{T}\to T$. Let $f:S\to T$ be an optimal morphism. Let
\begin{displaymath}
\mathcal{C}'(S,\widehat{T}):=\{(x,y)\in S\times \widehat{T}|f(x)=p(y)\}
\end{displaymath}
\noindent be the fiber product of $S$ and $\widehat{T}$. The space $\mathcal{C}'(S,\widehat{T})$ is naturally equipped with a $G$-action induced by the diagonal action on $S\times\widehat{T}$. 

Let $(x,y),(x',y')\in\mathcal{C}'(S,\widehat{T})$. A finite sequence $((x,y)=(x_0,y_0),\dots,(x_k,y_k)=(x',y'))$ of elements of $\mathcal{C}'(S,\widehat{T})$ is \emph{admissible} if for all $i\in\{0,\dots,k-1\}$, the morphism $f$ is injective in restriction to $[x_i,x_{i+1}]$. The existence of admissible sequences between any two points of $\mathcal{C}'(S,\widehat{T})$ comes from the fact that $f$ is a morphism and $p$ is surjective. Given an admissible sequence $\sigma:=((x,y)=(x_0,y_0),\dots,(x_k,y_k)=(x',y'))$, we let $$l(\sigma):=\sum_{i=0}^{k-1} d_S(x_i,x_{i+1})+d_{\widehat{T}}(y_i,y_{i+1}).$$ For all $((x,y),(x',y'))\in\mathcal{C}'(S,\widehat{T})^2$, we then let $$d((x,y),(x',y')):=\inf_{\sigma}l(\sigma),$$ where the infimum is taken over all admissible sequences between $(x,y)$ and $(x',y')$. The map $d$ defines a metric on $\mathcal{C}'(S,\widehat{T})$: the triangle inequality follows from the fact that the concatenation of two admissible sequences is again admissible, and the separation axiom follows from the observation that $d((x,y),(x',y'))\ge d_S(x,x')+d_{\widehat{T}}(y,y')$. We first make the following observation.

\begin{lemma}\label{observation}
Let $(x,y),(x',y')\in\mathcal{C}'(S,\widehat{T})$, and let $\sigma:=((x,y)=(x_0,y_0),\dots,(x_k,y_k)=(x',y'))$ be an admissible sequence between $(x,y)$ and $(x',y')$. Let $i\in\{0,\dots,k-1\}$, let $x'_i\in [x_i,x_{i+1}]$, and let $y'_i$ be the projection of $y_i$ to $p^{-1}(x'_i)$. Let $\sigma'$ be the sequence obtained by inserting $(x'_i,y'_i)$ between $(x_i,y_i)$ and $(x_{i+1},y_{i+1})$ in $\sigma$. Then $\sigma'$ is admissible, and $l(\sigma')=l(\sigma)$.
\end{lemma}

\begin{proof}
As $\sigma$ is admissible, the morphism $f$ is injective in restriction to $[x_i,x_{i+1}]$. Since $x'_i\in [x_i,x_{i+1}]$, it is injective in restriction to both $[x_i,x'_i]$ and $[x'_i,x_{i+1}]$, so $\sigma'$ is admissible. We also have $d_S(x_i,x_{i+1})=d_S(x_i,x'_i)+d_S(x'_i,x_{i+1})$, and $f(x'_i)\in [f(x_i),f(x_{i+1})]$. Since $p$ preserves alignment, the bridge between the (disjoint) closed subtrees $p^{-1}(f(x_i))$ and $p^{-1}(f(x_{i+1}))$ meets $p^{-1}(f(x'_i))$, and therefore it contains the projection $y'_i$ of $y_i$ to $p^{-1}(x'_i)$. As $y_i\in p^{-1}(f(x_i))$ and $y_{i+1}\in p^{-1}(f(x_{i+1}))$, this implies that $d_{\widehat{T}}(y_i,y_{i+1})=d_{\widehat{T}}(y_i,y'_i)+d_{\widehat{T}}(y'_i,y_{i+1})$, from which Lemma \ref{observation} follows.
\end{proof}

\begin{lemma}\label{pullback-tree}
The metric space $(\mathcal{C}'(S,\widehat{T}),d)$ is a $(G,\mathcal{F})$-tree.
\end{lemma}

\begin{proof}
\noindent We start by proving that the topological space $\mathcal{C}'(S,\widehat{T})$ is path-connected. Let $(x,y),(x',y')\in\mathcal{C}'(S,\widehat{T})$. Since $f$ is a morphism, the collection $\{x_1,\dots,x_{k-1}\}$ of points in $[x,x']$ at which $f_{|[x,x']}$ is not locally injective is finite. We let $x_0:=x$ and $x_k:=x'$, and let $y_0:=y$. As $p$ preserves alignment, for all $i\in\{0,\dots,k-1\}$, the preimage $p^{-1}(f(x_i))$ is a closed subtree of $\widehat{T}$ (Lemma \ref{projection-closed}). We inductively define $y_{i+1}$ as the projection of $y_i$ to $p^{-1}(f(x_{i+1}))$, for $i\in\{0,\dots,k-2\}$, and we let $y_k:=y'$. We then let $\gamma_i:[0,1]\to\widehat{T}$ be the straight path joining $y_i$ to $y_{i+1}$ in $\widehat{T}$. Since $p$ is alignment-preserving, the path $p\circ\gamma_i$ is a continuous path joining $f(x_i)$ to $f(x_{i+1})$ in $T$, whose image is contained in the segment $[f(x_i),f(x_{i+1})]$. Therefore, composing with the inverse of $f_{|[x_i,x_{i+1}]}$, we get a path $\gamma_i^S:[0,1]\to S$ that joins $x_i$ to $x_{i+1}$. By construction, for all $t\in [0,1]$, we have $f(\gamma^S_i(t))=p(\gamma_i(t))$. The concatenation of all paths $(\gamma^S_i,\gamma_i)$ is a continuous path joining $(x,y)$ to $(x',y')$ in $\mathcal{C}'(S,\widehat{T})$. This proves that $\mathcal{C}'(S,\widehat{T})$ is path-connected.

We will now prove that the path we have constructed has length $d((x,y),(x',y'))$, which shows that the metric space $(\mathcal{C}'(S,\widehat{T}),d)$ is geodesic. Notice that the sequence $\sigma:=((x_0,y_0),\dots,(x_k,y_k))$ constructed above is admissible. We first claim that it realizes the infimum in the definition of $d((x,y),(x',y'))$. Let $\sigma':=((x'_0,y'_0),\dots,(x'_{k'},y'_{k'}))$ be another admissible sequence. For all $i\in\{1,\dots,k'-1\}$, let $x''_i$ be the projection of $x'_i$ to the segment $[x,x']$. Let $y''_0:=y'_0$, and inductively define $y''_{i+1}$ as the projection of $y''_i$ to the closed subtree $\pi^{-1}(x''_i)$. By Lemma \ref{observation}, the sequence $\sigma''$ we get by inserting $(x''_i,y''_i)$ between $(x'_i,y'_i)$ and $(x'_{i+1},y'_{i+1})$ in $\sigma'$ for all $i\in\{1,\dots,k'-1\}$ such that $x'_i$ and $x'_{i+1}$ do not project to the same point of $[x,x']$ is admissible, and $l(\sigma'')=l(\sigma')$.  By construction, the sequence $\sigma^{3}$ is a refinement of $\sigma$, so $l(\sigma)\le l(\sigma'')$. The claim follows.

In addition, if we identify $[x_i,x_{i+1}]$ (respectively $[y_i,y_{i+1}]$) with $[0,d_S(x_i,x_{i+1})]$ (resp. $[0,d_{\widehat{T}}(y_i,y_{i+1})]$), the path $(\gamma_i,\gamma_i^S)$ is the graph of a continuous non-decreasing map, whose length is thus equal to $d_{S}(x_i,x_{i+1})+d_{\widehat{T}}(y_i,y_{i+1})$. This follows from the fact that $f$ is isometric in restriction to $[x_i,x_{i+1}]$. Together with the claim from the above paragraph, this shows that the arc we have built from $(x,y)$ to $(x',y')$ has length $d((x,y),(x',y'))$. 

We finally show that $\mathcal{C}'(S,\widehat{T})$ is uniquely path-connected. Assume that there exists a topological embedding $\gamma=(\gamma_S,\gamma_{\widehat{T}}):S^1\to\mathcal{C}'(S,\widehat{T})$ from the circle into $\mathcal{C}'(S,\widehat{T})$. The map $\gamma_S$ cannot be constant, because the fiber of every point in $S$ (under the projection map from $\mathcal{C}'(S,\widehat{T})$ to $S$) is a tree. As $S$ is an $\mathbb{R}$-tree, there exists $u\in S^1$ whose $\gamma_S$-image is extremal in $\gamma_S(S^{1})$. Then $\gamma_S(S^{1})$ contains a segment $I\subseteq S$ whose extremity is equal to $\gamma_S(u)$. There is a subsegment $I'\subseteq I$, one of whose endpoints is equal to $\gamma_S(u)$, such that all points in the interior of $I'$ have at least two $\gamma_S$-preimages in $S^1$ (one on each side of $u$). Hence there exists an uncountable set $J$ of elements $s\in S^1$, with $\gamma_S(s)\neq\gamma_S(t)$ for all $s\neq t\in J$, and such that for all $s\in J$, there exists $s'\in S^1\smallsetminus\{s\}$ satisfying $\gamma_S(s)=\gamma_S(s')$. Injectivity of $\gamma$ implies that for all $s\in S$, the segment $[\gamma_{\widehat{T}}(s),\gamma_{\widehat{T}}(s')]$ is nondegenerate. In addition, for all $s\neq t\in J$, the segments $[\gamma_{\widehat{T}}(s),\gamma_{\widehat{T}}(s')]$ and $[\gamma_{\widehat{T}}(t),\gamma_{\widehat{T}}(t')]$ are disjoint, because they project to distinct points in $T$. We have thus found an uncountable collection of pairwise disjoint nondegenerate segments in $\widehat{T}$, which contradicts separability of $\widehat{T}$ (which follows from minimality). Therefore, there is no topological embedding from the circle into $\mathcal{C}'(S,\widehat{T})$. We have thus proved that any two points $(x,y),(x',y')\in\mathcal{C}'(S,\widehat{T})$ are joined by a unique embedded topological arc, and this arc has length $d((x,y),(x',y'))$. This shows that $\mathcal{C}'(S,\widehat{T})$ is an $\mathbb{R}$-tree. 

For all $g\in G$, the image of an admissible sequence under the action of $g$ is again admissible (by equivariance of $f$ and $p$). Therefore, as $G$ acts by isometries on each of the trees $S$ and $\widehat{T}$, it also acts by isometries on $\mathcal{C}'(S,\widehat{T})$. As all peripheral subgroups of $G$ act elliptically in both $S$ and $\widehat{T}$, they also act elliptically in $\mathcal{C}'(S,\widehat{T})$. Hence $\mathcal{C}'(S,\widehat{T})$ is a $(G,\mathcal{F})$-tree. 
\end{proof}

\begin{de}
Let $S,T$ and $\widehat{T}$ be tame $(G,\mathcal{F})$-trees, such that there exists a $1$-Lipschitz alignment-preserving map $p:\widehat{T}\to T$ and a morphism $f:S\to T$. The \emph{pullack} $\mathcal{C}(S,\widehat{T})$ induced by $f$ and $p$ is defined to be the $G$-minimal subtree of $\mathcal{C}'(S,\widehat{T})$.  
\end{de}

\begin{lemma}\label{pullback-small}
Let $S,T$ and $\widehat{T}$ be $(G,\mathcal{F})$-trees, let $p:\widehat{T}\to T$ be an alignment-preserving map, and let $f:S\to T$ be a morphism. Assume that there exists $k\in\mathbb{N}$ such that $S$ and $\widehat{T}$ are $k$-tame. Then $\mathcal{C}(S,\widehat{T})$ is $k$-tame. 
\end{lemma}

\begin{proof}
Let $g\in G$, and let $I:=[(x,y),(x',y')]$ be an arc in $\mathcal{C}(S,\widehat{T})$. Assume that there exists $l\in\mathbb{N}$ such that $g^lI=I$. Then both $[x,x']\subseteq S$ and $[y,y']\subseteq\widehat{T}$ are fixed by $g^l$. Since $S$ and $\widehat{T}$ are $k$-tame, we have $g^k[x,x']=[x,x']$ and $g^k[y,y']=[y,y']$, so $g^kI=I$. This implies that $\mathcal{C}(S,\widehat{T})$ is $k$-tame.
\end{proof}

\subsubsection{Pullbacks of tame optimal folding paths}

We will now present a construction that will prove the following proposition.

\begin{prop}\label{prop-pull}
Let $T$ and $\widehat{T}$ be tame $(G,\mathcal{F})$-trees, and let $p:\widehat{T}\to T$ be a $1$-Lipschitz alignment-preserving map. Let $\gamma$ be a tame optimal folding path ending at $T$. Then there exists a reparameterization $\gamma'$ of $\gamma$, and a reparameterized tame optimal folding path $\widetilde{\gamma}$ ending at $\widehat{T}$ such that for all $t\in\mathbb{R}_+$, there is an alignment-preserving map from $\widetilde{\gamma}(t)$ to $\gamma'(t)$.
\end{prop}

Let $S,T$ and $\widehat{T}$ be $k$-tame $(G,\mathcal{F})$-trees, such that there exists a $1$-Lipschitz alignment-preserving map $p:\widehat{T}\to T$. Let $\gamma$ be a $k$-tame optimal liberal folding path from $S$ to $T$, guided by an optimal morphism $f:S\to T$. For all $t\in\mathbb{R}_+$, let $\widehat{\gamma}(t)$ denote the pullback $\mathcal{C}(\gamma(t),\widehat{T})$. Lemma \ref{pullback-small} implies that for all $t\in\mathbb{R}_+$, the $(G,\mathcal{F})$-tree $\widehat{\gamma}(t)$ is $k$-tame. For all $t\in\mathbb{R}_+$, there is an alignment-preserving map $p_t:\widehat{\gamma}(t)\to\gamma(t)$.  

The path $\widehat{\gamma}$ will be called the \emph{pullback} of $\gamma$ induced by $p$. We will prove below (Lemma \ref{morphisms}) that for all $t<t'$, there is an optimal morphism $\widehat{f}_{t,t'}:\widehat{\gamma}(t)\to\widehat{\gamma}(t')$, and these satisfy $\widehat{f}_{t,t''}=\widehat{f}_{t',t''}\circ\widehat{f}_{t,t'}$ for all $t<t'<t''$. However, the path $\widehat{\gamma}$ may fail to be an optimal liberal folding path, because it may be discontinuous. There is a way of turning $\widehat{\gamma}$ into a (continuous) optimal liberal folding path. As length functions can only decrease along the path $\widehat{\gamma}$, there are (at most) countably many times $t$ at which $\widehat{\gamma}(t^-)\neq\widehat{\gamma}(t^+)$, where $\widehat{\gamma}(t^-)$ (resp. $\widehat{\gamma}(t^+)$) is the limit of the trees $\widehat{\gamma}(s)$ as $s$ converges to $t$ from below (resp. from above). If $t$ is one of these discontinuity times, we will show that there exists an optimal morphism from $\widehat{\gamma}(t^-)$ to $\widehat{\gamma}(t^+)$. By inserting the corresponding optimal liberal folding paths at all discontinuity times (and reparameterizing if needed, in particular in case the identification times for the inserted paths are unbounded), we will get a continuous path $\widehat{\gamma}^{cont}$, called a \emph{continuous pullback} of $\gamma$ induced by $p$. We will show that all trees in the inserted path collapse to $\gamma(t)$: this follows from Lemmas \ref{pullback-morphisms} and \ref{folding-collapse} below. 

\begin{lemma}\label{morphisms}
There exist optimal morphisms $\widehat{f}_{t,t'}:\widehat{\gamma}(t)\to\widehat{\gamma}(t')$ for all $t<t'$, such that $\widehat{f}_{t,t''}=\widehat{f}_{t',t''}\circ\widehat{f}_{t,t'}$ for all $t<t'<t''$.
\end{lemma}

\begin{proof}
Let $t<t'\in\mathbb{R}_+$. We can define a $G$-equivariant map $\widetilde{f}_{t,t'}:\mathcal{C}'(\gamma(t),\widehat{T})\to\mathcal{C}'(\gamma(t'),\widehat{T})$ by setting $\widetilde{f}_{t,t'}(x,y):=(f_{t,t'}(x),y)$. Indeed, if $(x,y)\in\mathcal{C}'(\gamma(t),\widehat{T})$, then $f_{t,L}(x)=p(y)$, hence $f_{t',L}(f_{t,t'}(x))=p(y)$. We claim that the map $\widetilde{f}_{t,t'}$ is a morphism. Indeed, let $((x,y),(x',y'))\in\mathcal{C'}(\gamma(t),\widehat{T})^2$. As $f_{t,t'}$ is a morphism, the segment $[x,x']\subseteq\gamma(t)$ can be subdivided into finitely many subsegments $[x_i,x_{i+1}]$, in restriction to which $f_{t,t'}$ is an isometry. Using the arguments from the proof of Lemma \ref{pullback-tree}, we see that the segment $[(x,y),(x',y')]$ can be subdivided into finitely many subsegments that are either of the form $[(x_i,y_i),(x_{i+1},y_{i+1})]$ or $[(x_i,y_i),(x_i,y'_i)]$, in restriction to which $\widetilde{f}_{t,t'}$ is an isometry. 

We now prove that $\widetilde{f}_{t,t'}$ induces an optimal morphism $\widehat{f}_{t,t'}:\widehat{\gamma}(t)\to\widehat{\gamma}(t')$. For all $t\in\mathbb{R}$, the map $p_t$ preserves alignment and is surjective by minimality of $\gamma(t)$, so every arc in $\gamma(t)$ lifts to an arc in $\widehat{\gamma}(t)$. We also notice that for all $t\in\mathbb{R}$, the fibers of the map $p_t$ isometrically embed into $\widehat{T}$, and hence into $\widehat{\gamma}(t')$ for all $t'>t$. Let $\widehat{x}\in\widehat{\gamma}(t)$. Let $x:=p_t(\widehat{x})$. If $p_t^{-1}(x)$ is not reduced to a point, then we can find a direction $d$ at $\widehat{x}$ in $\widehat{\gamma}(t)$ that is contained in $p_t^{-1}(x)$. If there exists another direction $d'$ at $\widehat{x}$ contained in $p_t^{-1}(x)$, then the turn $(d,d')$ is legal. Otherwise, minimality of $\widehat{\gamma}(t)$ shows the existence of a direction $d'$ at $x$ that is not contained in $p_t^{-1}(x)$. We claim that the directions $d$ and $d'$ cannot be identified by $\widehat{f}_{t,t'}$. Indeed, otherwise, any small nondegenerate arc $I$ contained in the direction $d'$ would be mapped to a point by $p_{t'}\circ \widehat{f}_{t,t'}$, and hence by $f_{t,t'}\circ \pi_t$, contradicting the fact that $f_{t,t'}$ is a morphism and $p_t(I)$ is a nondegenerate arc. If $p_t^{-1}(x)$ is reduced to a point, then every legal turn at $x$ lifts to a legal turn at $\widehat{x}$, and optimality of $f_{t,t'}$ ensures the existence of such turns.  
\end{proof}

For all $t\in\mathbb{R}_+$, we denote by $\widehat{\gamma}(t^-)$ (resp. $\widehat{\gamma}(t^+)$) the limit of the trees $\widehat{\gamma}(s)$ as $s$ converges to $t$ from below (resp. from above). This exists by monotonicity of length functions along the path $\widehat{\gamma}$, which comes from the existence of the morphisms $\widehat{f}_{t,t'}$ for all $t,t'\in\mathbb{R}$. As the space of $k$-tame $(G,\mathcal{F})$-trees is closed, the trees $\widehat{\gamma}(t^-)$ and $\widehat{\gamma}(t^+)$ are $k$-tame. It follows from Proposition \ref{Lipschitz-limit} that there are alignment-preserving maps $p_{t^-}:\widehat{\gamma}(t^-)\to\gamma(t)$ and $p_{t^+}:\widehat{\gamma}(t^+)\to\gamma(t)$. Proposition \ref{Lipschitz-limit} also implies that for all $s<t$, there are $1$-Lipschitz maps $\widehat{f}_{s,t^-}$ and $\widehat{f}_{s,t^+}$ from $\widehat{\gamma}(s)$ to the metric completions of both $\widehat{\gamma}(t^-)$ and $\widehat{\gamma}(t^+)$. For all $s>t$, there are $1$-Lipschitz maps $\widehat{f}_{t^-,s}$ and $\widehat{f}_{t^+,s}$ from both $\widehat{\gamma}(t^-)$ and $\widehat{\gamma}(t^+)$ to the metric completion of $\widehat{\gamma}(s)$. There is also a $1$-Lipschitz map $\widehat{f}_{t^-,t^+}$ from $\widehat{\gamma}(t^-)$ to the metric completion of $\widehat{\gamma}(t^+)$. We will show that all these maps are morphisms. We will make use of the following easy lemma, that was noticed by Guirardel and Levitt in \cite[Lemma 3.3]{GL07}.

\begin{lemma}\label{composition-morphisms}
Let $T_1,T_2,T_3$ be $\mathbb{R}$-trees. Let $f:T_1\to T_3$ be a morphism, and let $\phi:T_1\to T_2$ and $\psi:T_2\to T_3$ be $1$-Lipschitz surjective maps, such that $f=\psi\circ\phi$. Then $\phi$ and $\psi$ are morphisms.
\qed
\end{lemma}

\begin{lemma}\label{pullback-morphisms}
For all $t\in\mathbb{R}$, the map $\widehat{f}_{t^-,t^+}$ is an optimal morphism, and $p_{t^+}\circ\widehat{f}_{t^-,t^+}=p_{t^-}$. For all $s<t$, the maps $\widehat{f}_{s,t^-}$ and $\widehat{f}_{s,t^+}$ are optimal morphisms. For all $s>t$, the maps $\widehat{f}_{t^-,s}$ and $\widehat{f}_{t^+,s}$ are optimal morphisms. 
\end{lemma}

\begin{proof}
We refer to \cite[Section 4]{Hor14-1} for notations and definitions. Let $\widehat{\gamma}^{\omega}(t^-)$ (resp. $\widehat{\gamma}^{\omega}(t^+)$) be an ultralimit of a sequence of trees $\widehat{\gamma}(s)$, with $s$ converging to $t^-$ (resp. to $t^+$). Then $\widehat{\gamma}^{\omega}(t^-)$ (resp. $\widehat{\gamma}^{\omega}(t^+)$) contains $\widehat{\gamma}(t^-)$ (resp. $\widehat{\gamma}(t^+)$) as its $G$-minimal subtree. In \cite[Theorem 4.3]{Hor14-1}, the map $\widehat{f}_{t^-,t^+}$ is constructed from the ultralimit $\widehat{f}_{t^-,t^+}^{\omega}:\widehat{\gamma}^{\omega}(t^-)\to\widehat{\gamma}^{\omega}(t^+)$ of the maps $\widehat{f}_{s,s'}$, by restricting to the minimal subtree $\widehat{\gamma}(t^-)$, and projecting to the closure of the minimal subtree $\widehat{\gamma}(t^+)$. For all $s<t<s'$, we have $\widehat{f}_{s,s'}=\widehat{f}_{t^+,s'}^{\omega}\circ \widehat{f}_{t^-,t^+}^{\omega}\circ \widehat{f}_{s,t^-}^{\omega}$ (where the maps $\widehat{f}_{t^+,s'}^{\omega}$ and $\widehat{f}_{s,t^-}^{\omega}$ are defined similarly as ultralimits). The difficulty might come from projection to minimal subtrees in the definition of $\widehat{f}_{t^-,t^+}$. However, optimality of $\widehat{f}_{s,s'}$ implies that minimal subtrees are mapped to minimal subtrees, so we get $\widehat{f}_{s,s'}=\widehat{f}_{t^+,s'}\circ \widehat{f}_{t^-,t^+}\circ \widehat{f}_{s,t^-}$. Using Lemma \ref{composition-morphisms}, this implies that all maps $\widehat{f}_{t^+,s'}$, $\widehat{f}_{t^-,t^+}$ and $\widehat{f}_{s,t^-}$ are morphisms. We can similarly prove that $\widehat{f}_{s,t^+}$ and $\widehat{f}_{t^-,s'}$ are morphisms. Any ultralimit of alignment-preserving maps is again alignment-preserving, and hence maps minimal subtrees to minimal subtrees. We similarly deduce that $p_{t^+}\circ\widehat{f}_{t^-,t^+}=p_{t^-}$. 
\end{proof}

\begin{lemma}\label{folding-collapse}
Let $T,T'$ and $\overline{T}$ be minimal $(G,\mathcal{F})$-trees, let $f:T\to T'$ be an optimal morphism, and let $(T_t)_{t\in [0,L]}$ be a folding path guided by $f$. Let $\pi:T\to \overline{T}$ and $\pi':T'\to\overline{T}$ be alignment-preserving maps, such that $\pi'\circ f=\pi$. Then for all $t\in [0,L]$, there is an alignment-preserving map from $T_t$ to $\overline{T}$.
\end{lemma}

\begin{proof}
Let $y\in \overline{T}$, and let $x\in T_t$ be a preimage of $y$ under the map $\pi'\circ f_{t,L}$. As $\pi'\circ f=\pi$, all $f_{0,t}$-preimages of $x$ in $T$ map to $y$ under $\pi$, so $(\pi'\circ f_{t,L})^{-1}(y)=f_{0,t}(\pi^{-1}(y))$. As $\pi$ preserves alignment, the preimage $\pi^{-1}(y)$ is connected, and therefore $(\pi'\circ f_{t,L})^{-1}(y)$ is connected. This implies that $\pi'\circ f_{t,L}$ preserves alignment.
\end{proof}

This finishes the proof of Proposition \ref{prop-pull}. Again, there is a discrete version of the above construction. The following proposition follows from the above analysis.

\begin{prop}\label{def-pullback}
Let $T$ and $\widehat{T}$ be tame $(G,\mathcal{F})$-trees, and let $p:\widehat{T}\to T$ be a $1$-Lipschitz alignment-preserving map. Let $(\gamma(n))_{n\in\mathbb{N}}$ be a tame optimal folding sequence ending at $T$. Then there exists a tame $(G,\mathcal{F})$-tree $\widehat{\gamma}(\infty)$, and a tame optimal folding sequence $(\widehat{\gamma}(n))_{n\in\mathbb{N}}$ ending at $\widehat{\gamma}(\infty)$, such that 
\begin{itemize}
\item for all $n\in\mathbb{N}$, the tree $\widehat{\gamma}(n)$ collapses to $\gamma(n)$, and
\item there is an alignment-preserving map $p_{\infty}:\widehat{\gamma}(\infty)\to T$, and a morphism $\widehat{f}_{\infty}:\widehat{\gamma}(\infty)\to\widehat{T}$, such that $p_{\infty}=p\circ \widehat{f}_{\infty}$.
\end{itemize}
\qed 
\end{prop}

A sequence $(\widehat{\gamma}(n))_{n\in\mathbb{N}}$ satisfying the conclusions of Proposition \ref{def-pullback} will be called a \emph{pullback} of $(\gamma(n))_{n\in\mathbb{N}}$ induced by $p$. We note that in general, the morphism $\widehat{f}_{\infty}:\widehat{\gamma}(\infty)\to \widehat{T}$ given by Proposition \ref{def-pullback} need not be an isometry. In other words, the sequence $(\widehat{\gamma}(n))_{n\in\mathbb{N}}$ need not end at $\widehat{T}$. Here is an example: suppose that $\widehat{T}$ is a simplicial $(G,\mathcal{F})$-tree containing an edge with nontrivial stabilizer $\langle g\rangle$, and that this edge gets collapsed to a point in $T$. We may find an optimal folding sequence $(\gamma(n))_{n\in\mathbb{N}}$ ending at $T$ such that $g$ is hyperbolic in $\gamma(n)$ for all $n\in\mathbb{N}$. In this situation, the element $g$ will also be hyperbolic in $\widehat{\gamma}(n)$ for all $n\in\mathbb{N}$, which implies that $g$ cannot fix a nondegenerate edge in $\widehat{\gamma}(\infty)$: indeed, the quotient volume of $\widehat{\gamma}(\infty)$ has to be equal to the limit of the quotient volumes of the trees $\widehat{\gamma}(n)$, which would not be the case if $g$ fixed a nondegenerate arc in $\widehat{\gamma}(\infty)$ (see \cite[Propositions 3.12 and 3.15]{AK12}, for instance). This implies that $\widehat{\gamma}(\infty)$ is not isometric to $\widehat{T}$. The following proposition shows however that if $T$ is $\mathcal{Z}$-irreducible, then the map from $\widehat{\gamma}(\infty)$ to $\widehat{T}$ is an isometry.

\begin{prop}\label{pullback}
Let $T$ and $\widehat{T}$ be tame $(G,\mathcal{F})$-trees, and let $p:\widehat{T}\to T$ be an alignment-preserving map. Let $(\gamma(n))_{n\in\mathbb{N}}$ be a tame optimal folding sequence ending at $T$. If $T$ is $\mathcal{Z}$-incompatible, then any pullback of $(\gamma(n))_{n\in\mathbb{N}}$ induced by $p$ ends at $\widehat{T}$.
\end{prop}

\begin{proof}
There exists a $1$-Lipschitz alignment-preserving map from $\widehat{\gamma}(\infty)$ to $T$. As $T$ is not compatible with any $\mathcal{Z}$-splitting, this implies that $\widehat{\gamma}(\infty)$ has dense orbits. We also know that there exists a morphism from $\widehat{\gamma}(\infty)$ to $\widehat{T}$. Corollary \ref{alignment-preserved} implies that this morphism is an isometry.
\end{proof}

\section{Boundedness of the set of reducing splittings of a tame $(G,\mathcal{F})$-tree}\label{sec-not-X}

We will now prove the following result that bounds the diameter in $FZ(G,\mathcal{F})$ of the set of reducing splittings of a tame $(G,\mathcal{F})$-tree $T\notin\mathcal{X}(G,\mathcal{F})$. We recall there is a map $\psi:\mathcal{O}(G,\mathcal{F})\to FZ(G,\mathcal{F})$, which naturally extends to the set of tame $(G,\mathcal{F})$-trees having a nontrivial simplicial part. By definition, any tame optimal sequence ending at $T$ comes with morphisms $f_n:T_n\to T$ and is nonstationary. By Corollary \ref{alignment-preserved}, this implies that for all $n\in\mathbb{N}$, the tree $T_n$ does not have dense orbits. The existence of the Levitt decomposition of $T_n$ as a graph of actions (Proposition \ref{Levitt}) implies that $\psi(T_n)$ is well-defined for all $n\in\mathbb{N}$. The goal of the present section is to prove the following theorem. 

\begin{theo} \label{canonical-splittings}
There exists $C_1\in\mathbb{R}$ so that for all tame $(G,\mathcal{F})$-trees $T\notin\mathcal{X}(G,\mathcal{F})$, the diameter of $\mathcal{R}^{2}(T)$ in $FZ(G,\mathcal{F})$ is at most $C_1$. Furthermore, there exists $C_2\in\mathbb{R}$ such that for all tame $(G,\mathcal{F})$-trees $T\notin\mathcal{X}(G,\mathcal{F})$, all $\psi$-images of tame optimal folding sequences ending at $T$ eventually stay at distance at most $C_2$ from $\mathcal{R}^{2}(T)$ in $FZ(G,\mathcal{F})$. 
\end{theo}

Again, the proof of Theorem \ref{canonical-splittings} adapts without change to the case of $\mathcal{Z}^{max}$-splittings to give the following statement.

\begin{theo}($\mathcal{Z}^{max}$-analogue of Theorem \ref{canonical-splittings})
There exists $C_1\in\mathbb{R}$ so that for all $\mathcal{Z}^{max}$-tame trees $T\notin\mathcal{X}^{max}(G,\mathcal{F})$, the diameter of $\mathcal{R}^{2,max}(T)$ in $FZ^{max}(G,\mathcal{F})$ is at most $C_1$. Furthermore, there exists $C_2\in\mathbb{R}$ such that for all $\mathcal{Z}^{max}$-tame trees $T\notin\mathcal{X}^{max}(G,\mathcal{F})$, all $\psi^{max}$-images of tame optimal folding sequences ending at $T$ eventually stay at distance at most $C_2$ from $\mathcal{R}^{2,max}(T)$ in $FZ^{max}(G,\mathcal{F})$. 
\end{theo}

\begin{rk}
Here again, it is crucial to work with cyclic splittings rather than free splittings. Indeed, it is possible to find a tree $T\in\overline{cv_N}$ that is compatible with infinitely many free splittings of $F_N$ that do not lie in a region of finite diameter of the free splitting graph $FS_N$. Here is an example. The tree $T_2$ from Example \ref{surfaces} is compatible with all the free splittings of $F_N$ that are determined by arcs on $S$ that lie in the subsurface $S_1$. These arcs form an unbounded subset of the arc graph of $S$ \cite{MS13}. In addition, it is known that the arc graph of $S$ embeds quasi-isometrically into the free splitting graph of $F_N$ \cite[Lemma 4.17 and Proposition 4.18]{HH14}.
\end{rk}

\subsection{The case where $T$ has a nontrivial simplicial part}

\begin{prop} \label{folding-nondense}
Let $T$ be a tame $(G,\mathcal{F})$-tree having a nontrivial simplicial part. Let $(T_n)_{n\in\mathbb{N}}$ be a tame optimal folding sequence ending at $T$. Then for all $n\in\mathbb{N}$, the tree $T_n$ has a nontrivial simplicial part, and there exists $n_0\in\mathbb{N}$ such that for all $n\ge n_0$, we have $d_{FZ(G,\mathcal{F})}(\psi(T_n),\psi(T))\le 2$.
\end{prop}

\begin{proof}
Let $e$ be a simplicial edge in $T$. If $T$ contains an edge with nontrivial stabilizer, then we choose $e$ to be an edge whose stabilizer $\langle g_0\rangle$ is nontrivial, such that no proper root of $g_0$ fixes an arc in $T$. The tree $T$ splits as graph of actions $\mathcal{G}$, dual to the edge $e$. We will prove that for $n\in\mathbb{N}$ sufficiently large, the tree $T_n$ splits as a graph of actions over the skeleton of $\mathcal{G}$. This will imply that both $T_n$ and $T$ are compatible with this skeleton, and therefore $d_{FZ(G,\mathcal{F})}(\psi(T_n),\psi(T))\le 2$. We denote by $f_n:T_n\to T$ the morphism given by the definition of a tame optimal folding sequence.

We first assume that the edge $e$ projects to a separating edge in $\mathcal{G}$. We denote by $T^A$ and $T^B$ the adjacent vertex trees, with nontrivial stabilizers $A$ and $B$. Let $K$ denote the length of $e$ in $T$, and let $\epsilon>0$, chosen to be very small compared to $K$. For all $n\in\mathbb{N}$, let $T_n^A$ be the closure of the $A$-minimal subtree in $T_n$. We will first show that for $n\in\mathbb{N}$ large enough, the $f_n$-image of $T_n^A$ is contained in an $\epsilon$-neighborhood of $T^A$ in $T$. The analogous statement will also hold for the $B$-minimal subtrees.

Assume that $A$ is not elliptic in $T$. Let $s\in A$ be hyperbolic in $T$. The Kurosh decomposition of $A$ reads as $$A:=g_{1}G_{i_1}g_{1}^{-1}\ast\dots\ast g_{l}G_{i_l}g_{l}^{-1}\ast F.$$ We let $X_A$ be a finite set made of a free basis of $F$, and a nontrivial element in each of the peripheral subgroups $g_{j}G_{i_j}g_{j}^{-1}$. Let $X'_A$ be the finite set consisting of $s$, and of the elements of the form $s.s^a$ for all $a\in X_A$ such that the axes of $s$ and $s^a$ do not intersect in $T$. Notice that all elements in $X'_A$ are hyperbolic in $T$, and hence in $T_n$ for all $n\in\mathbb{N}$. For all $n\in\mathbb{N}$ the translates of the axes in $T_n$ of elements in $X'_A$ cover $T_n^A$, because the same holds true in $T$. For all $g\in X'_A$, the $f_n$-image of the axis of $g$ in $T_n$ is contained in the $\epsilon$-neighborhood of the axis of $g$ in $T$ for $n$ large enough. Our claim follows in this case.

Assume now that $A$ is elliptic in $T$, and let $s\in A$ be an element that fixes a unique point $p$ in $T$ (this exists by minimality of $T$). If $s$ is hyperbolic in $T_n$, then the axis of $s$ is mapped to the $\epsilon$-neighborhood of $p$ for $n$ large enough. If $s$ is elliptic in $T_n$, then $s$ fixes a unique point $p_n\in T_n$ (otherwise, by optimality, the $f_n$-image of a nondegenerate arc fixed by $s$ would be a nondegenerate arc fixed by $s$ in $T$), and $f_n(p_n)=p$. In both cases, we denote by $Y_n$ the union of the characteristic sets in $T_n$ of all elements $s^a$ with $a\in X_A$. As $||s.s^a||_{T_n}$ converges to $0$ for all $a\in X_A$, the convex hull of $Y_n$ in $T_n$ is contained in the $\epsilon$-neighborhood of $Y_n$, for $n$ large enough. The translates of the convex hull of $Y_n$ cover $T_n^A$. Hence the $f_n$-image of $T_n^A$ is contained in an $\epsilon$-neighborhood of $p$, and the claim follows.

Therefore, for $n$ large enough, all translates of $T_n^A$ and $T_n^B$ are disjoint, and the stabilizer of $T_n^A$ (resp. $T_n^B$) is equal to $A$ (resp. $B$). For $n$ large enough, the bridge between $T_n^A$ and $T_n^B$ has length $l_n\ge K-\epsilon$. Since $g_0\in A\cap B$, the element $g_0$ has to be elliptic in both $T_n^A$ and $T_n^B$, and therefore $g_0$ fixes the bridge between $T_n^A$ and $T_n^B$. Hence we can form a graph of actions $S_n$ over the skeleton of $\mathcal{G}$, with vertex trees $T_n^A$ and $T_n^B$, whose attaching points are given by the extremities of the bridge between $T_n^A$ and $T_n^B$. The unique orbit of edges $e_n$ of this splitting is assigned length $l_n$. We claim that $S_n$ is isometric to $T_n$, which will prove that $T_n$ splits as a graph of actions over $\mathcal{G}$. By construction, there is a morphism $f_n:S_n\to T_n$, which is an isometry in restriction to both $T_n^A$ and $T_n^B$. The morphism $f_n$ cannot identify a subarc of $e_n$ with a subarc in either $T_n^A$ or $T_n^B$ by definition of $l_n$. It cannot either identify a subarc of $e_n$ with one of its translates by our choice of the edge $e$ (because otherwise $T$ would contain an arc fixed by a proper subgroup of $\langle g_0\rangle$). This implies that $f_n$ is an isometry, and proves the claim.

In the case of an HNN extension, we denote by $C$ the vertex group and by $t$ a stable letter. The same argument as above yields $l_n\ge K-\epsilon$, where this time $l_n$ denotes the distance between the closure $T_n^C$ of the $C$-minimal subtree of $T_n$, and its $t$-translate (one has to be slightly careful if $G=F_2$, because in this case $C$ is cyclic, so the $C$-minimal subtree of $T_n$ is not well-defined; we leave the argument to the reader in this case). We can similarly define the graph of actions $S_n$ over the skeleton of $\mathcal{G}$, with vertex tree $T_n^C$, with attaching points the extremities of the bridge between $T_n^C$ and its $t$-translate. The unique orbit of edges $e_n$ is assigned length $l$. In this case, the morphism $f_n:S_n\to T_n$ may fail to be an isometry, however it can only fold $e_n$ with $t'\overline{e_n}$, where $t'$ is a stable letter of the HNN extension. As $t'$ is hyperbolic in $T$, and hence in $T_n$, the edge $e_n$ is not entirely folded. Again, we get that $T_n$ splits as a graph of actions over $\mathcal{G}$.
\end{proof}

\subsection{The general case}

\begin{prop} \label{freely-reducible-case}
Let $T_1,T_2$ and $\overline{T}$ be tame $(G,\mathcal{F})$-trees. Assume that there exist alignment-preserving maps $p_1:T_1\to\overline{T}$ and $p_2:T_2\to\overline{T}$, and a morphism $f:T_1\to T_2$, such that $p_1=p_2\circ f$. If $T_1$ and $T_2$ both have a nontrivial simplicial part, then $\mathcal{R}^1(T_1)\cup\mathcal{R}^1(T_2)$ has diameter at most $4$ in $FZ(G,\mathcal{F})$.
\end{prop}

The strategy of our proof of Proposition \ref{freely-reducible-case} is the following. We will first show that we can reduce the proof to the case where $\overline{T}$ has dense orbits, and no part of a subtree in the dense orbit part of the Levitt decomposition of either $T_1$ or $T_2$ is collapsed when mapping to $\overline{T}$. In this situation, we prove that the morphism $f:T_1\to T_2$ has a very simple form: all folding occurs in the simplicial part of $T_1$, and the maps $p_i$ collapse the simplicial part of $T_i$. From this observation, it is easy to find $\mathcal{Z}$-splittings of $(G,\mathcal{F})$ which are close to both $\mathcal{R}^1(T_1)$ and $\mathcal{R}^1(T_2)$ in $FZ(G,\mathcal{F})$. 

\begin{proof}
If $\overline{T}$ does not have dense orbits, then any splitting in $\mathcal{R}^1(T_1)\cup\mathcal{R}^1(T_2)$ is at distance $1$ from a splitting defined by a simplicial edge in $\overline{T}$. From now on, we assume that $\overline{T}$ has dense orbits. For all $i\in\{1,2\}$, let $\mathcal{Y}_i$ be the collection of all nondegenerate vertex subtrees with dense orbits of the Levitt decomposition of $T_i$ as a graph of actions (Proposition \ref{Levitt}). Let $\mathcal{Z}_i$ be the collection of all connected components of the union of the closures of the edges in the simplicial part of $T_i$. Then $\mathcal{Y}_i\cup\mathcal{Z}_i$ is a transverse covering of $T_i$. By definition, the vertex set of its skeleton is equal to $\mathcal{Y}_i\cup\mathcal{Z}_i\cup V_i$, where $V_i$ is the set of intersection points between distinct trees in $\mathcal{Y}_i\cup\mathcal{Z}_i$. Notice that trees in $\mathcal{Y}_i$ are pairwise disjoint, and similarly trees in $\mathcal{Z}_i$ are pairwise disjoint. Therefore, any vertex of $V_i$ is joined by an edge to exactly one vertex in $\mathcal{Y}_i$, and one vertex in $\mathcal{Z}_i$.

We first explain how to reduce to the case where for all $i\in\{1,2\}$, the map $p_i$ is isometric in restriction to the dense orbits part of $T_i$. For all $i\in\{1,2\}$, let $\overline{T_i}$ be the tree obtained from the decomposition of $T_i$ as a graph of actions by replacing each subtree $Y\in\mathcal{Y}_i$ by $p_i(Y)$, each $v\in V_i$ by $p_i(v)$, and leaving each $Z\in\mathcal{Z}_i$ unchanged. Let $\overline{\mathcal{Y}_i}$ be the collection of all subtrees of $\overline{T_i}$ defined by the trees $p_i(Y)$ with $Y\in\mathcal{Y}_i$. All $f$-images of trees $Y\in\mathcal{Y}_1$ are contained in the dense orbits part of $T_2$, and $p_2(f(Y))=p_1(Y)$. Hence $f$ induces a $1$-Lipschitz map $\overline{f}:\overline{T_1}\to\overline{T_2}$, which is an isometry in restriction to each of the subtrees in $\overline{\mathcal{Y}_1}$. By modifying the lengths of the edges in the simplicial part of $T_1$ if needed, we may turn $\overline{f}$ into a morphism. Some of these edges may have to be assigned length $0$, however some edge with positive length must survive because $T_2$ has a nontrivial simplicial part. Denoting by $\overline{p_1}:\overline{T_1}\to\overline{T}$ and by $\overline{p_2}:\overline{T_2}\to\overline{T}$ the induced alignment-preserving maps, we still have $\overline{p_1}=\overline{p_2}\circ\overline{f}$. In this way, as $\mathcal{R}^1(T_i)\subseteq\mathcal{R}^1(\overline{T_i})$ for all $i\in\{1,2\}$, we have reduced the proof of Proposition \ref{freely-reducible-case} to the case where the maps $f$ and $p_1$ (resp. $p_2$) are isometric in restriction to the trees in the dense orbits part of the Levitt decomposition of $T_1$ (resp. $T_2$). From now on, we assume that we are in this case.

\begin{figure}
\begin{center}
\def\JPicScale{.8}
\input{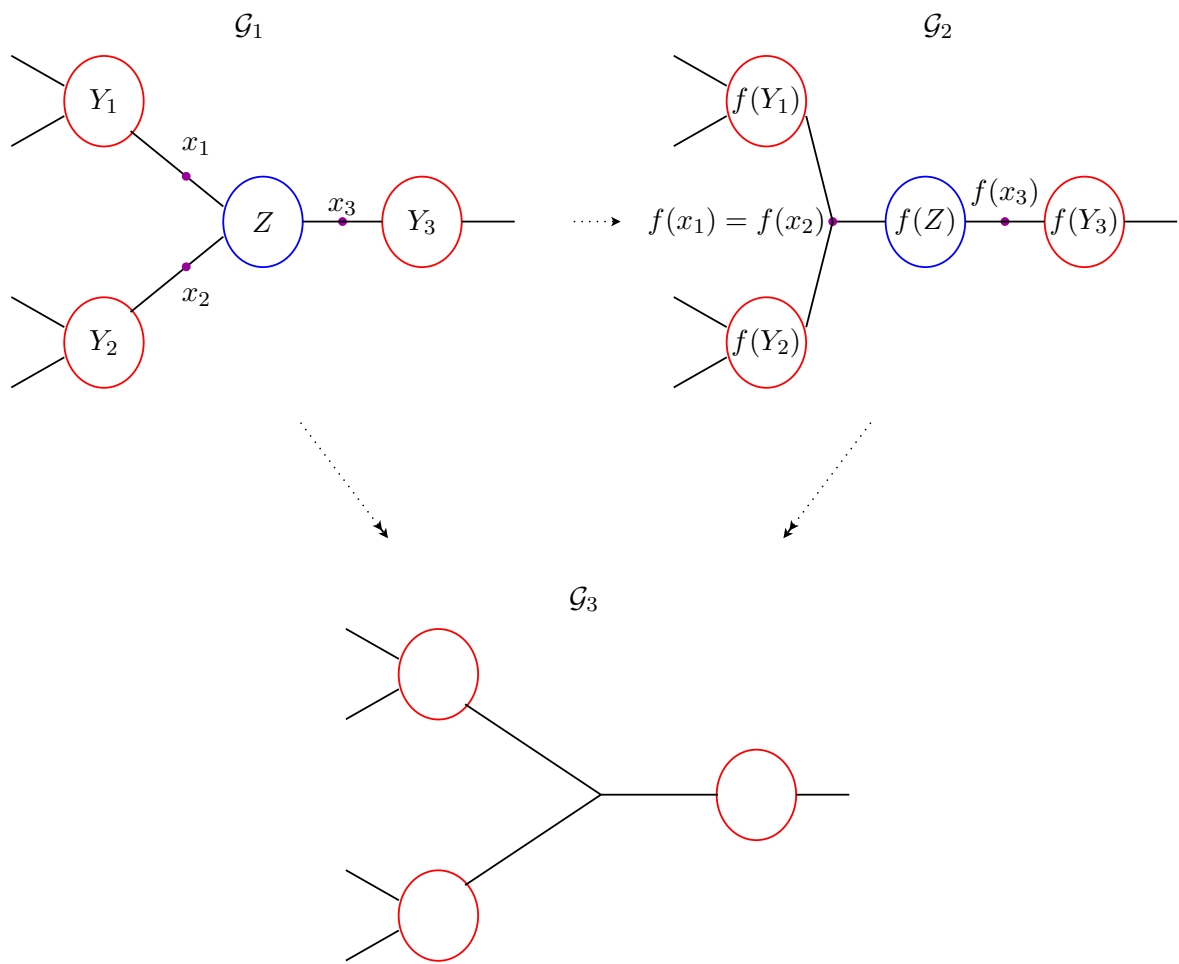}
\caption{The skeletons $\mathcal{G}_1$, $\mathcal{G}_2$ and $\mathcal{G}_3$ in the proof of Proposition \ref{freely-reducible-case}.}
\label{fig-goa}
\end{center}
\end{figure}

We now prove that $f(\mathcal{Y}_1)\cup f(\mathcal{Z}_1)$ is a transverse covering of $T_2$. This follows from the following observations.

\begin{itemize}
\item If $Y_1\neq Y_2\in\mathcal{Y}_1$, then $f(Y_1)\cap f(Y_2)$ contains at most one point. Indeed, the map $p_2$ is isometric on both $f(Y_1)$ and $f(Y_2)$. If $f(Y_1)\cap f(Y_2)$ contained a nondegenerate arc $I$, then $p_2\circ f(Y_1)\cap p_2\circ f(Y_2)$ would contain the nondegenerate arc $p_2(I)$. In other words, the intersection $p_1(Y_1)\cap p_1(Y_2)$ would be nondegenerate. This is impossible because $p_1$ preserves alignment and $Y_1\cap Y_2=\emptyset$.
\item If $Z_1\neq Z_2\in\mathcal{Z}_1$, then $f(Z_1)\cap f(Z_2)=\emptyset$. Indeed, since $\overline{T}$ has dense orbits, the image $p_1(Z_i)$ is a point for all $i\in\{1,2\}$, so $p_2(f(Z_i))$ is a point. If $f(Z_1)\cap f(Z_2)\neq\emptyset$, then the subtrees $f(Z_1)$ and $f(Z_2)$ would be mapped to the same point $z\in\overline{T}$ under $p_2$. Let $z_1\in Z_1$ and $z_2\in Z_2$. Then $p_1(z_1)=p_1(z_2)$, so $p_1$ collapses the segment $[z_1,z_2]$ to a point. Since $Z_1\neq Z_2$, it follows from the description of the skeleton $\mathcal{G}_1$ that the segment $[z_1,z_2]$ intersects some tree $Y\in{\mathcal{Y}_1}$ along a nondegenerate segment $I$, and $p_1(I)$ is nondegenerate, a contradiction.
\item Similarly, the $f$-images of two trees $Y_1\in{\mathcal{Y}_1}$ and $Z_1\in\mathcal{Z}_1$ intersect nontrivially if and only if $Y_1\cap Z_1\neq\emptyset$ (otherwise the bridge between $Y_1$ and $Z_1$ would be collapsed by $p_1$), and in this case their intersection is the $f$-image of the intersection point between $Y_1$ and $Z_1$.
\end{itemize}

This implies that the union $f({\mathcal{Y}_1})\cup {f}(\mathcal{Z}_1)$ is a transverse covering of ${T_2}$. We denote by $\mathcal{G}_2$ its skeleton, which is depicted on Figure \ref{fig-goa}. The above observations imply that the map $f$ induces a map from $\mathcal{G}_1$ to $\mathcal{G}_2$, that sends edges to edges, and is injective in a neighborhood of any vertex in $\mathcal{Y}_1$. The map $f$ can fold several edges attached to a vertex in $\mathcal{Z}_1$. In particular, the map $f$ induces an isomorphism between the graphs $\mathcal{G}'_1$ and $\mathcal{G}'_2$, obtained by equivariantly collapsing the $1$-neighborhood of all vertices in $\mathcal{Z}_1$ (resp. $f(\mathcal{Z}_1)$). We let $\mathcal{G}_3:=\mathcal{G}'_1=\mathcal{G}'_2$, see Figure \ref{fig-goa}.

Let $Y$ be one of the nontrivial subtrees in ${\mathcal{Y}_1}$ (this exists because $\overline{T}$ has dense orbits). We denote by $G_Y$ the stabilizer of $Y$. The subgroup $G_Y$ is a vertex stabilizer in a $\mathcal{Z}$-splitting, obtained from $T_1$ by collapsing all vertex trees of the Levitt decomposition to points. Therefore, we have $\text{rk}_K(G_Y)<+\infty$ (see \cite[Corollary 4.5]{Hor14-5}). Let $S'$ be a $\mathcal{Z}$-splitting of $(G_Y,\mathcal{F}_{|G_Y})$, such that the stabilizers of all attaching points in $Y$ are elliptic in $S'$ (this exists by Proposition \ref{stabilizers}). Let $e$ be an edge of $S'$. 

For all $i\in\{1,2,3\}$, let $T'_i$ be the $(G,\mathcal{F})$-tree obtained by replacing $Y$ by $S'$ in $\mathcal{G}_i$. This is well-defined because attaching points and edge stabilizers are the same in neighborhoods of $Y$ and $f(Y)$. Then $T'_1$ and $T'_2$ both collapse to $T'_3$, and for all $i\in\{1,2\}$, the tree $T'_i$ collapses to $\psi(T_i)$. This implies that for all $i\in\{1,2\}$, we have $d_{FZ(G,\mathcal{F})}(\psi(T_i),T'_3)\le 1$, and hence $d_{FZ(G,\mathcal{F})}(\psi(T_1),\psi(T_2))\le 2$. Since $T_i$ has a nontrivial simplicial part, the set $\mathcal{R}^1(T_i)$ is contained in the $1$-neighborhood of $\psi(T_i)$ in $FZ(G,\mathcal{F})$. Hence $\mathcal{R}^1(T_1)\cup\mathcal{R}^1(T_2)$ has diameter at most $4$ in $FZ(G,\mathcal{F})$.
\end{proof}

\noindent We now finish the proof of Theorem \ref{canonical-splittings}. In the case where $\mathcal{R}^1(T)\neq\emptyset$, we start by proving the following lemma.

\begin{lemma}\label{case-nonempty}
There exists a constant $C>0$ such that the following holds. Let $T$ be a tame $(G,\mathcal{F})$-tree, such that $\mathcal{R}^1(T)\neq\emptyset$. Let $(\gamma(n))_{n\in\mathbb{N}}$ be a tame optimal folding sequence ending at $T$. Then there exists $n_0\in\mathbb{N}$ such that the diameter in $FZ(G,\mathcal{F})$ of the set $\psi(\gamma([n_0,+\infty)))\cup\mathcal{R}^1(T)$ is at most $C$. 
\end{lemma}

\begin{proof}
The case of trees without dense orbits has been dealt with in Proposition \ref{folding-nondense}, hence we can assume $T$ to have dense orbits. Let $S\in\mathcal{R}^1(T)$, let $\widehat{T}:=T+S$, and let $p:T+S\to T$ be the corresponding $1$-Lipschitz alignment-preserving map. Let $(\widehat{\gamma}(n))_{n\in\mathbb{N}}$ be a pullback of $(\gamma(n))_{n\in\mathbb{N}}$ induced by $p$, provided by Proposition \ref{def-pullback}: the sequence $(\widehat{\gamma}(n))_{n\in\mathbb{N}}$ is a tame optimal folding sequence that ends at a tame tree $\widehat{\gamma}(\infty)$, and there is a $1$-Lipschitz alignment-preserving map $p_{\infty}:\widehat{\gamma}(\infty)\to T$, and a morphism $\widehat{f}_{\infty}:\widehat{\gamma}(\infty)\to T+S$, such that $p_{\infty}=p\circ \widehat{f}_{\infty}$. As $T+S$ has a nontrivial simplicial part, so does $\widehat{\gamma}(\infty)$. Proposition \ref{folding-nondense} implies that there exists $n_0\in\mathbb{N}$ such that for all $n\ge n_0$, we have $d_{FZ(G,\mathcal{F})}(\psi(\widehat{\gamma}(n)),\psi(\widehat{\gamma}(\infty)))\le 4$. By applying Proposition \ref{freely-reducible-case} to $T_1=\widehat{\gamma}(\infty)$ and $T_2=\widehat{T}$, we get that $d_{FZ(G,\mathcal{F})}(\psi(\widehat{\gamma}(\infty)),\psi(\widehat{T}))\le 2$. So for all $n\ge n_0$, we have $d_{FZ(G,\mathcal{F})}(\psi(\widehat{\gamma}(n)),\psi(\widehat{T}))\le 6$. As $\widehat{\gamma}(n)$ collapses to $\gamma(n)$, and $\widehat{T}$ collapses to $S$, this implies that for all $n\ge n_0$, we have $d_{FZ(G,\mathcal{F})}(\psi(\gamma(n)),S)\le 8$. The simplicial tree $S$ has been chosen independently from  the sequence $(\gamma(n))_{n\in\mathbb{N}}$, so the above inequality holds for all $S\in\mathcal{R}^1(T)$. This implies that the distance in $FZ(G,\mathcal{F})$ between any two splittings $S,S'\in\mathcal{R}^1(T)$ is at most $16$, and proves the lemma. 
\end{proof}

\begin{lemma}\label{push}
There exists a constant $C>0$ such that the following holds. Let $T$ and $T'$ be tame $(G,\mathcal{F})$-trees. Assume that $T$ admits a $1$-Lipschitz alignment-preserving map onto $T'$, and that $\mathcal{R}^1(T')\neq\emptyset$. Let $\gamma$ be a tame optimal liberal folding sequence ending at $T$. Then there exists $n_0\in\mathbb{N}$ such that the diameter in ${FZ(G,\mathcal{F})}$ of the set $\psi(\gamma([n_0,+\infty)))\cup\mathcal{R}^1(T')$ is at most $C$.
\end{lemma}

\begin{proof}
Let $p:T\to T'$ be a $1$-Lipschitz alignment-preserving map, and let $(\overline{\gamma}(n))_{n\in\mathbb{N}}$ be a collapse of $(\gamma(n))_{n\in\mathbb{N}}$ induced by $p$, provided by Proposition \ref{def-collapse}: the sequence $(\overline{\gamma}(n))_{n\in\mathbb{N}}$ is a tame optimal folding sequence ending at $T'$. Lemma \ref{case-nonempty} applied to $T'$ ensures that the diameter in $FZ(G,\mathcal{F})$ of $\psi(\overline{\gamma}([n_0,+\infty))\cup\mathcal{R}^1(T')$ is bounded. As $\gamma(n)$ collapses to $\overline{\gamma}(n)$ for all $n\in\mathbb{N}$, the diameter of $\psi(\gamma([n_0,+\infty)))\cup\mathcal{R}^1(T')$ is also bounded.
\end{proof}

\begin{proof} [Proof of Theorem \ref{canonical-splittings}]
As $T\notin\mathcal{X}(G,\mathcal{F})$, there exists a tame $(G,\mathcal{F})$-tree $T'$ that is compatible with $T$, such that $\mathcal{R}^1(T')\neq\emptyset$. Let $(\gamma(n))_{n\in\mathbb{N}}$ be a tame optimal folding sequence ending at $T$, whose existence follows from Proposition \ref{existence-folding}. We will show that there exists $n_0\in\mathbb{N}$ such that the diameter in $FZ(G,\mathcal{F})$ of $\psi(\gamma([n_0,+\infty)))\cup\mathcal{R}^1(T')$ is bounded. As this is true for any tree $T'$ which is compatible with $T$ and satisfies $\mathcal{R}^1(T')\neq\emptyset$, and as $\gamma$ is chosen independently from $T'$, all trees in $\mathcal{R}^2(T)$ will be close to each other (and close to the end of $\gamma$), and Theorem \ref{canonical-splittings} will follow. 

Let $\widehat{T}:=T+T'$. If $\widehat{T}$ has a nontrivial simplicial part, then any splitting determined by this simplicial part belongs to both $\mathcal{R}^1(T)$ and $\mathcal{R}^1(T')$. Lemma \ref{case-nonempty} implies that both $\mathcal{R}^1(T)$ and $\mathcal{R}^1(T')$ have bounded diameter, and since $\mathcal{R}^1(T)\cap\mathcal{R}^1(T')\neq\emptyset$, the union $\mathcal{R}^1(T)\cup\mathcal{R}^1(T')$ also has bounded diameter. Lemma \ref{case-nonempty} applied to $T$ and the sequence $(\gamma(n))_{n\in\mathbb{N}}$ shows the existence of $n_0\in\mathbb{N}$ such that the diameter of $\psi(\gamma([n_0,+\infty)))\cup\mathcal{R}^1(T)$ is bounded. The claim follows in this case.

From now on, we assume that $\widehat{T}$ has dense orbits. Consider $1$-Lipschitz alignment-preserving maps $p:\widehat{T}\to T$ and $p':\widehat{T}\to T'$. First suppose that $\mathcal{R}^1(T)\neq\emptyset$, and let $(\widetilde{\gamma}(n))_{n\in\mathbb{N}}$ be any tame optimal folding sequence ending at $\widehat{T}$. Lemma \ref{push} applied to both $p$ and $p'$ implies the existence of $n_0\in\mathbb{N}$ such that $\psi(\widetilde{\gamma}([n_0,+\infty)))\cup\mathcal{R}^1(T)$ and $\psi(\widetilde{\gamma}([n_0,+\infty)))\cup\mathcal{R}^1(T')$ are bounded. In addition, Lemma \ref{case-nonempty} ensures that we can choose $n_0$ so that the diameter of $\psi(\gamma([n_0,+\infty)))\cup\mathcal{R}^1(T)$ is also bounded. The claim follows.

Suppose now that $\mathcal{R}^1(T)=\emptyset$. Let $(\widehat{\gamma}(n))_{n\in\mathbb{N}}$ be a pullback of $(\gamma(n))_{n\in\mathbb{N}}$ induced by $p$ (provided by Proposition \ref{def-pullback}). Then $(\widehat{\gamma}(n))_{n\in\mathbb{N}}$ ends at $\widehat{T}$ by Proposition \ref{pullback}. Lemma \ref{push} applied to $p'$ shows the existence of $n_0\in\mathbb{N}$ such that the diameter of $\psi(\widehat{\gamma}([n_0,+\infty)))\cup\mathcal{R}^1(T')$ is bounded. As $\widehat{\gamma}(n)$ collapses to $\gamma(n)$ for all $n\in\mathbb{N}$, the diameter of $\psi({\gamma}([n_0,+\infty)))\cup\mathcal{R}^1(T')$ is also bounded, and we are done.
\end{proof}

\begin{cor} \label{reducing-projection}
There exists $C\in\mathbb{R}$ such that for all tame $(G,\mathcal{F})$-trees $T_1,T_2\notin\mathcal{X}(G,\mathcal{F})$, if $T_1$ and $T_2$ are both refined by a common tame $(G,\mathcal{F})$-tree, then the diameter of $\mathcal{R}^2(T_1)\cup\mathcal{R}^2(T_2)$ in $FZ(G,\mathcal{F})$ is bounded by $C$.
\end{cor}

\begin{proof}
Denoting by $T$ the common refinement of $T_1$ and $T_2$, Corollary \ref{reducing-projection} follows from Theorem \ref{canonical-splittings} and the fact that for all $i\in\{1,2\}$, we have $\mathcal{R}^2(T)\subseteq\mathcal{R}^2(T_i)$, and $\mathcal{R}^2(T)\neq\emptyset$ by Proposition \ref{averse}.
\end{proof}

\section{The Gromov boundary of the graph of cyclic splittings} \label{sec-boundary}

We now turn to the proof of our main theorem, which gives a description of the Gromov boundary of $FZ(G,\mathcal{F})$. We will extend the map $\psi:\mathcal{O}(G,\mathcal{F})\to FZ(G,\mathcal{F})$ to a map $\partial\psi:\mathcal{X}(G,\mathcal{F})/{\sim}\to\partial_{\infty} FZ(G,\mathcal{F})$, and show that this extension is an $\text{Out}(G,\mathcal{F})$-equivariant homeomorphism. 

\begin{theo}\label{main}
Let $G$ be a countable group, and let $\mathcal{F}$ be a free factor system of $G$. There exists a unique $\text{Out}(G,\mathcal{F})$-equivariant homeomorphism $$\partial\psi:\mathcal{X}(G,\mathcal{F})/{\sim}\to \partial_{\infty} FZ(G,\mathcal{F}),$$ so that for all $T\in\mathcal{X}(G,\mathcal{F})$ and all sequences $(T_n)_{n\in\mathbb{N}}\in \mathcal{O}(G,\mathcal{F})^{\mathbb{N}}$ converging to $T$, the sequence $(\psi(T_n))_{n\in\mathbb{N}}$ converges to $\partial\psi (T)$. 
\end{theo}

Theorem \ref{main} also holds true (with the same proof) for the graph $FZ^{max}(G,\mathcal{F})$. 

\begin{theo}\label{main-2}
Let $G$ be a countable group, and let $\mathcal{F}$ be a free factor system of $G$. There exists a unique $\text{Out}(G,\mathcal{F})$-equivariant homeomorphism $$\partial\psi^{max}:\mathcal{X}^{max}(G,\mathcal{F})/{\sim}\to \partial_{\infty} FZ^{max}(G,\mathcal{F}),$$ so that for all $T\in\mathcal{X}^{max}(G,\mathcal{F})$ and all sequences $(T_n)_{n\in\mathbb{N}}\in \mathcal{O}(G,\mathcal{F})^{\mathbb{N}}$ converging to $T$, the sequence $(\psi^{max}(T_n))_{n\in\mathbb{N}}$ converges to $\partial\psi^{max} (T)$. 
\end{theo}

The following consequence regarding the topology of $\partial_{\infty}FZ^{(max)}(G,\mathcal{F})$ was suggested to us by the referee.

\begin{cor}
The Gromov boundaries $\partial_{\infty}FZ(G,\mathcal{F})$ and $\partial_{\infty}FZ^{max}(G,\mathcal{F})$ have cohomological dimension at most $3N+2k-5$.
\end{cor}

\begin{proof}
The maps $\partial\psi$ and $\partial\psi^{max}$ from Theorems \ref{main} and \ref{main-2} induce maps from the projectivization $P\mathcal{X}(G,\mathcal{F})$ to $\partial_{\infty}FZ(G,\mathcal{F})$ and $\partial_{\infty}FZ^{max}(G,\mathcal{F})$. Corollary \ref{contractible} implies that these induced maps are cell-like in the sense of \cite{Lac69}. Therefore \cite{Wal81}, the cohomological dimension of $\partial_{\infty}FZ(G,\mathcal{F})$ and $\partial_{\infty}FZ^{max}(G,\mathcal{F})$ is bounded above by the topological dimension of $\partial P\mathcal{O}(G,\mathcal{F})$, which was proved in \cite{Hor14-5} to be equal to $3N+2k-5$.
\end{proof}

\begin{rk}
We do not know however whether the covering dimension of either $\partial_{\infty}FZ(G,\mathcal{F})$ or $\partial_{\infty}FZ^{max}(G,\mathcal{F})$ is finite.
\end{rk}

\paragraph{Definition of $\partial\psi$.}

The following lemma may be viewed as a kind of Cauchy criterion for Gromov products.

\begin{lemma}\label{lemma-psi}
Let $S,T\in\mathcal{X}(G,\mathcal{F})$, such that there exists a $1$-Lipschitz alignment-preserving map from $S$ to $T$. Let $(S_i)_{i\in\mathbb{N}}\in\mathcal{O}(G,\mathcal{F})^{\mathbb{N}}$ (resp. $(T_i)_{i\in\mathbb{N}}\in\mathcal{O}(G,\mathcal{F})^{\mathbb{N}}$) be a sequence of trees that converges (non-projectively) to $S$ (resp. to $T$). Assume that for all $i\in\mathbb{N}$, there exists $J_i\in\mathbb{N}$ so that for all $j\ge J_i$, we have $\text{Lip}(S_i,T_j)\le 1+\frac{1}{i}$. Then 
\begin{displaymath}
\forall C\ge 0, \exists I_C\in\mathbb{N}, \forall i\ge I_C, \exists J_{i,C}\in\mathbb{N}, \forall j\ge J_{i,C}, (\psi(S_i)|\psi(T_j))\ge C.
\end{displaymath}
\end{lemma}

\begin{proof}
Otherwise, there would exist $C\ge 0$ and increasing sequences $(i_k)_{k\in\mathbb{N}}$ and $(j_k)_{k\in\mathbb{N}}$ of integers, so that for all $k\in\mathbb{N}$, we have $j_k\ge J_{i_k}$ and $(\psi(S_{i_k})|\psi(T_{j_k}))\le C$. For all $k\in\mathbb{N}$, let $\gamma_k$ be an optimal liberal folding path from a point in the cone of $S_{i_k}$ to $T_{j_k}$ given by Proposition \ref{limit-folding}. As $\psi$-images of optimal liberal folding paths are uniformly Hausdorff close to geodesics (Theorem \ref{FZ-hyperbolic}), for all $k\in\mathbb{N}$, we can find $Z_{k}$ in the image of $\gamma_k$, so that the sequence $(\psi(Z_{k}))_{k\in\mathbb{N}}$ is bounded in $FZ(G,\mathcal{F})$. Proposition \ref{limit-folding} implies that $(Z_{k})_{k\in\mathbb{N}}$ has an accumulation point $Z\in\overline{\mathcal{O}(G,\mathcal{F})}$ that comes with alignment-preserving maps from $S$ to $Z$ and from $Z$ to $T$. In particular, we have $Z\in\mathcal{X}(G,\mathcal{F})$. By Theorem \ref{Luo}, the sequence $(\psi(Z_k))_{k\in\mathbb{N}}$ should be unbounded, a contradiction.
\end{proof}

We will also make use of the following general statement about Gromov hyperbolic metric spaces. 

\begin{lemma}\label{hyp-prop}
Let $X$ be a Gromov hyperbolic metric space. Let $(X_i)_{i\in\mathbb{N}}\in X^{\mathbb{N}}$ and $(Y_i)_{i\in\mathbb{N}}\in X^{\mathbb{N}}$ be two sequences in $X$. Assume that 
\begin{displaymath}
\forall C\ge 0, \exists I_C\in\mathbb{N}, \forall i\ge I_C, \exists J_{i,C}\in\mathbb{N},\forall j\ge J_{i,C},(X_i|Y_j)\ge C.
\end{displaymath} 

\noindent Then $(X_i)_{i\in\mathbb{N}}$ and $(Y_i)_{i\in\mathbb{N}}$ both converge to the same point of the Gromov boundary $\partial_{\infty} X$.
\end{lemma}

\begin{proof}
Let $\delta$ be the hyperbolicity constant of $X$. The assumption implies that for all $C\ge 0$, and all $j,j'\ge J_{I_C,C}$, we have $(X_{I_C}|Y_j)\ge C$ and $(X_{I_C}|Y_{j'})\ge C$, whence $(Y_j|Y_{j'})\ge C-\delta$. Therefore, the sequence $(Y_j)_{j\in\mathbb{N}}$ converges to some point $\xi\in\partial_{\infty} X$. Then for all $C\ge 0$, there exists $I_C\in\mathbb{N}$ such that for all $i\ge I_C$, we have $(X_i|\xi)\ge C-\delta$. This implies that $(X_i)_{i\in\mathbb{N}}$ also converges to $\xi$.
\end{proof}

\begin{prop}\label{X-psi}
There exists a unique map $\partial\psi:\mathcal{X}(G,\mathcal{F})\to\partial_{\infty} FZ(G,\mathcal{F})$ such that for all $T\in\mathcal{X}(G,\mathcal{F})$ and all sequences $(T_j)_{j\in\mathbb{N}}\in \mathcal{O}(G,\mathcal{F})^{\mathbb{N}}$ converging to $T$, the sequence $(\psi(T_j))_{j\in\mathbb{N}}$ converges to $\partial\psi(T)$. In addition, if $S,T\in\mathcal{X}(G,\mathcal{F})$ satisfy $S\sim T$, then $\partial\psi(S)=\partial\psi(T)$.
\end{prop}

\begin{proof}
Let $T\in\mathcal{X}(G,\mathcal{F})$. We will prove that for all sequences $(T_j)_{j\in\mathbb{N}}\in\mathcal{O}(G,\mathcal{F})^{\mathbb{N}}$ that converge (non-projectively) to $T$, the sequence $(\psi(T_j))_{j\in\mathbb{N}}$ converges to a point in $\partial_{\infty} FZ(G,\mathcal{F})$. This implies that all such sequences have the same limit, which we call $\partial\psi(T)$. This will define $\partial\psi$ (and show its uniqueness). 

Let $(T_j)_{j\in\mathbb{N}}\in\mathcal{O}(G,\mathcal{F})^{\mathbb{N}}$ be a sequence that converges (non-projectively) to $T$. Let $T'\in\mathcal{X}(G,\mathcal{F})$ be any tree that is compatible with $T$, and let $\widehat{T}:=T+T'$. The tree $\widehat{T}$ has dense orbits, otherwise $T$ would be $\mathcal{Z}$-compatible, so Proposition \ref{Lipschitz-approximation} shows the existence of a Lipschitz approximation $(\widehat{T}_i)_{i\in\mathbb{N}}\in \mathcal{O}(G,\mathcal{F})^{\mathbb{N}}$ of $\widehat{T}$. Proposition \ref{csq-Lipschitz} ensures that for all $i\in\mathbb{N}$, there exists $J_i\in\mathbb{N}$ such that for all $j\ge J_i$, we have $\text{Lip}(\widehat{T}_i,T_j)\le 1+\frac{1}{i}$. Lemma \ref{lemma-psi} then shows that $$\forall C\in\mathbb{N},\exists I_C\in\mathbb{N},\forall i\ge I_C,\exists J_{i,C}\in\mathbb{N},\forall j\ge J_{i,C},(\psi(\widehat{T}_i)|\psi(T_j))\ge C.$$ Together with Lemma \ref{hyp-prop}, this implies that $(\psi(T_j))_{j\in\mathbb{N}}$ converges to some point $\xi\in\partial_{\infty} FZ(G,\mathcal{F})$. This defines $\partial\psi$. 

Furthermore, the sequence $(\psi(\widehat{T}_i))_{i\in\mathbb{N}}$ also converges to $\xi$ (Lemma \ref{hyp-prop}). Therefore, we have proved that for all $T,T'\in\mathcal{X}(G,\mathcal{F})$, if $T$ is compatible with $T'$, then $\partial\psi(T)=\partial\psi(T')$. Therefore, if $T,T'\in\mathcal{X}(G,\mathcal{F})$ satisfy $T\sim T'$, then  $\partial\psi(T)=\partial\psi(T')$. 
\end{proof}

\noindent Proposition \ref{X-psi} shows that $\partial\psi$ descends to a map $\partial\psi:\mathcal{X}(G,\mathcal{F})/{\sim}\to\partial_{\infty} FZ(G,\mathcal{F})$.

\paragraph{Continuity of $\partial\psi$.}

\begin{prop}\label{psi-continuous}
The map $\partial\psi:\mathcal{X}(G,\mathcal{F})/{\sim}\to\partial_{\infty} FZ(G,\mathcal{F})$ is continuous.
\end{prop}

\begin{proof}
By definition of the quotient topology, it is enough to check that $\partial\psi:\mathcal{X}(G,\mathcal{F})\to\partial_{\infty} FZ(G,\mathcal{F})$ is continuous. Let $T\in\mathcal{X}(G,\mathcal{F})$, and let $(T_i)_{i\in\mathbb{N}}\in\mathcal{X}(G,\mathcal{F})^{\mathbb{N}}$ be a sequence that converges non-projectively to $T$. We want to show that $(\partial\psi(T_i))_{i\in\mathbb{N}}$ converges to $\partial\psi(T)$. Let $(S_k)_{k\in\mathbb{N}}\in\mathcal{O}(G,\mathcal{F})^{\mathbb{N}}$ be a sequence that converges to $T$. Proposition \ref{X-psi} implies that the sequence $(\psi(S_k))_{k\in\mathbb{N}}$ converges to $\partial\psi(T)$. Therefore, up to replacing $(S_k)_{k\in\mathbb{N}}$ by a subsequence, we can assume that for all $k\in\mathbb{N}$, we have 
\begin{equation}\label{eq-1}
(\psi(S_k)|\partial\psi(T_k))\le(\partial\psi(T)|\partial\psi(T_k))+\delta.
\end{equation}

\noindent Recall that for all $k\in\mathbb{N}$, we have $T_k\in\mathcal{X}(G,\mathcal{F})$. Therefore, using Proposition \ref{X-psi}, we can find a sequence $(S'_k)_{k\in\mathbb{N}}\in\mathcal{O}(G,\mathcal{F})^{\mathbb{N}}$ (where we choose $S'_k$ sufficiently close to $T_k$) such that 

\begin{itemize}
\item the sequence $(S'_k)_{k\in\mathbb{N}}$ converges to $T$ in $\overline{\mathcal{O}(G,\mathcal{F})}$, and 
\item for all $k\in\mathbb{N}$, we have 
\begin{equation}\label{eq-2}
(\psi(S_k)|\psi(S'_k))\le (\psi(S_k)|\partial\psi(T_k))+\delta. 
\end{equation}
\end{itemize}

\noindent Combining Equations \eqref{eq-1} and \eqref{eq-2}, we then get that for all $k\in\mathbb{N}$, we have
\begin{equation}\label{eq1}
(\psi(S_k)|\psi(S'_k))\le(\partial\psi(T)|\partial\psi(T_k))+2\delta.
\end{equation}

\noindent As both $(\psi(S_k))_{k\in\mathbb{N}}$ and $(\psi(S'_k))_{k\in\mathbb{N}}$ converge to $\partial\psi(T)$ (Proposition \ref{X-psi}), the Gromov product $(\psi(S_k)|\psi(S'_k))$ tends to $+\infty$, and hence $(\partial\psi(T)|\partial\psi(T_k))$ tends to $+\infty$. This implies that $(\partial\psi(T_k))_{k\in\mathbb{N}}$ converges to $\partial\psi(T)$.
\end{proof}

\paragraph{Injectivity of $\partial\psi$.}

\begin{prop}\label{psi-injective}
The map $\partial\psi:\mathcal{X}(G,\mathcal{F})/{\sim}\to\partial_{\infty} FZ(G,\mathcal{F})$ is injective.
\end{prop}

\begin{proof}
Let $T,T'\in\mathcal{X}(G,\mathcal{F})$ be such that $\partial\psi(T)=\partial\psi(T')$. We choose $T$ and $T'$ to be mixing and $\mathcal{Z}$-incompatible representatives in their equivalence classes (Proposition \ref{mixing-representative}). Let $\ast\in \mathcal{O}(G,\mathcal{F})$, and consider an optimal liberal folding path $\gamma$ (resp. $\gamma'$) from the cone of $\ast$ to $T$ (resp. to $T'$). By Proposition \ref{folding-uncollapsible}, all trees along the paths $\gamma$ and $\gamma'$ are simplicial and have trivial arc stabilizers. As $\partial\psi(T)=\partial\psi(T')$, it follows from Theorem \ref{FZ-hyperbolic} that the images $\psi(\gamma)$ and $\psi(\gamma')$ are at finite Hausdorff distance $M$ from each other, where $M$ is bounded independently from the paths $\gamma$ and $\gamma'$. Let $(\psi(T_i))_{i\in\mathbb{N}}$ be a sequence of points lying on $\psi(\gamma)$ and converging to $\partial\psi(T)$, and let $(\psi(T'_i))_{i\in\mathbb{N}}$ be a sequence of points lying on $\psi(\gamma')$ and converging to $\partial\psi(T)=\partial\psi(T')$, so that for all $i\in\mathbb{N}$, we have $d_{FZ(G,\mathcal{F})}(\psi(T_i),\psi(T'_i))\le M$. Up to passing to a subsequence, we may assume that the distance between $\psi(T_i)$ and $\psi(T'_i)$ is constant, we denote it by $C$. For all $i\in\mathbb{N}$, let $\psi(T_i)=\psi(T_i^0),\dots,\psi(T_i^C)=\psi(T'_i)$ be a geodesic segment in $FZ(G,\mathcal{F})$ joining $\psi(T_i)$ to $\psi(T'_i)$. Up to rescaling and passing to a subsequence, we may assume that for all $k\in\{0,\dots,C\}$, the sequence of one-edge simplicial $(G,\mathcal{F})$-trees $(T_i^k)_{i\in\mathbb{N}}$ converges non-projectively to a tame $(G,\mathcal{F})$-tree $T_{\infty}^k$ (Proposition \ref{limit-one-edge}). For all $i\in\mathbb{N}$, the splitting $\psi(T_i)$ (resp. $\psi(T'_i)$) is a collapse of $T_i$ (resp. $T'_i$), so $T$ (resp. $T'$) collapses to $T_{\infty}^0$ (resp. $T_{\infty}^C$). For all $k\in\{0,\dots,C-1\}$ and all $i\in\mathbb{N}$, the trees $T_i^k$ and $T_i^{k+1}$ are compatible, so Lemma \ref{compatible-limit} implies that $T_{\infty}^k$ and $T_{\infty}^{k+1}$ are compatible. Therefore $T\sim T'$.  
\end{proof}

\paragraph{Surjectivity of $\partial\psi$.}

\begin{prop} \label{psi-bounded}
For all $M\in\mathbb{R}$, there exists $C_M\in\mathbb{R}$ such that the following holds. Let $T\in\overline{\mathcal{O}(G,\mathcal{F})}\smallsetminus\mathcal{X}(G,\mathcal{F})$, and let $(T_n)_{n\in\mathbb{N}}\in \mathcal{O}(G,\mathcal{F})^{\mathbb{N}}$ be a sequence that converges to $T$, such that the sequence $(\psi(T_n))_{n\in\mathbb{N}}$ lies in a region of $FZ(G,\mathcal{F})$ of diameter bounded by $M$. Then $d_{FZ(G,\mathcal{F})}(\psi(T_n),\mathcal{R}^2(T))\le C_M$ for all $n\in\mathbb{N}$.
\end{prop}

\begin{proof}
It is enough to prove the desired bound for a subsequence of $(\psi(T_n))_{n\in\mathbb{N}}$, since the bound for the whole sequence follows by replacing $C_M$ by $C_M+M$. Up to passing to a subsequence, we can assume that there exists $\ast\in FZ(G,\mathcal{F})$ and $M'\le M$ such that for all $n\in\mathbb{N}$, we have $d_{FZ(G,\mathcal{F})}(\ast,\psi(T_n))=M'$. For all $n\in\mathbb{N}$, let $(\psi(T_n^k))_{k=0,\dots,M'}$ be a geodesic segment joining $\ast$ to $\psi(T_n)$ in $FZ(G,\mathcal{F})$. We may rescale the one-edge simplicial trees $T_n^k$ so that up to passing to a subsequence, for all $k\in\{0,\dots,M'\}$, the sequence $(T_n^k)_{n\in\mathbb{N}}$ converges non-projectively to a tame $(G,\mathcal{F})$-tree $T^k$ (Proposition \ref{limit-one-edge}). For all $k\in\{0,\dots,M'-1\}$ and all $n\in\mathbb{N}$, the trees $T_n^k$ and $T_n^{k+1}$ are compatible, so Lemma \ref{compatible-limit} implies that $T^k$ and $T^{k+1}$ are compatible. None of the trees $T^k$ is $\mathcal{Z}$-averse, and Corollary \ref{reducing-projection} shows the existence of $C'\in\mathbb{R}$ such that for all $k\in\{0,\dots,M'-1\}$, the diameter of $\mathcal{R}^2(T^k)\cup\mathcal{R}^2(T^{k+1})$ is bounded by $C'$. Since $T^0=\ast$ and $T^{M'}$ is compatible with $T$, the distance between $\ast$ and $\mathcal{R}^2(T)$ is at most $(M+1)C'$. In addition, we have $d_{FZ(G,\mathcal{F})}(\psi(T_n),\ast)\le M$ for all $n\in\mathbb{N}$. It then follows from the triangular inequality that for all $n\in\mathbb{N}$, the distance in $FZ(G,\mathcal{F})$ between $\psi(T_n)$ and $\mathcal{R}^2(T)$ is at most $(M+1)C'+M$.
\end{proof}

The following statement follows from classical arguments about Gromov hyperbolic spaces. We leave its proof to the reader.

\begin{prop} \label{hyperbolic}
Let $\delta>0$, let $X$ be a $\delta$-hyperbolic geodesic metric space. There exists $M\in\mathbb{R}$ only depending on $\delta$ such that the following holds. Let $\xi\in\partial_{\infty} X$, let $(x_n)_{n\in\mathbb{N}}\in X^{\mathbb{N}}$ be a sequence that converges to $\xi$, and for all $n\in\mathbb{N}$, let $\gamma_n$ be a geodesic segment from $x_0$ to $x_n$. Let $R>0$, and for all $n\in\mathbb{N}$, let $y_n\in\gamma_n$ be a point at distance exactly $R$ from $x_0$. Then there exists $n_0\in\mathbb{N}$, such that $\{y_n\}_{n\ge n_0}$ is contained in a region of $X$ of diameter at most $M$. 
\qed 
\end{prop}

\begin{prop}\label{psi-not-infty}
Let $T\in\overline{\mathcal{O}(G,\mathcal{F})}\smallsetminus\mathcal{X}(G,\mathcal{F})$, and let $(T_n)_{n\in\mathbb{N}}\in \mathcal{O}(G,\mathcal{F})^{\mathbb{N}}$ be a sequence that converges to $T$. Then $(\psi(T_n))_{n\in\mathbb{N}}$ does not converge to any point in $\partial_{\infty} FZ(G,\mathcal{F})$.
\end{prop}

\begin{proof}
We will assume that $T_0$ is a standard Grushko $(G,\mathcal{F})$-tree. We first notice that there exist a tree $T'_0\in\overline{\mathcal{O}(G,\mathcal{F})}$ in the closure of the cone of $T_0$, and trees $T_0^m\in\mathcal{O}(G,\mathcal{F})$ in the cone of $T_0$ for all $m\in\mathbb{N}$, together with optimal morphisms $f:T'_0\to T$ and $f_{m}:T_0^m\to T_m$, so that up to passing to a subsequence, the sequence $(f_{m})_{m\in\mathbb{N}}$ converges to $f$. The proof of this fact only requires to find a continuous way of sending a representative in each orbit of vertices in $T_0$ to the trees $T_m$ and $T$. The images of the vertices with peripheral stabilizers are prescribed, and if $\mathcal{F}\neq\emptyset$, we can choose to map the other vertex to the vertex stabilized by $G_1$ (we refer to \cite[Section 3]{GL07-2} for an argument in the case where $G$ is a free group). Let $\gamma$ (resp. $\gamma_{m}$) be the canonical folding path directed by $f$ (resp. $f_{m}$) constructed in \cite[Section 3]{GL07} (see Section \ref{sec-folding} of the present paper for a brief review of Guirardel and Levitt's construction). By \cite[Proposition 3.4]{GL07}, for all $t\in\mathbb{R}_+$, the trees $\gamma_{m}(t)$ converge to $\gamma(t)$ as $m$ tends to $+\infty$. In addition, by Proposition \ref{canonical-good}, all trees in the image of $\gamma$ are tame $(G,\mathcal{F})$-trees.

Assume towards a contradiction that the sequence $(\psi(T_n))_{n\in\mathbb{N}}$ converges to some $\xi\in\partial_{\infty} FZ(G,\mathcal{F})$. Let $M\in\mathbb{R}$ be the constant provided by Proposition \ref{hyperbolic}, and let $C_M\in\mathbb{R}$ be the constant provided by Proposition \ref{psi-bounded}, which we can choose to be greater than the constants $C_1$ and $C_2$ from Theorem \ref{canonical-splittings}. By Theorem \ref{canonical-splittings}, the $\psi$-image of $\gamma$ is a bounded region of $FZ(G,\mathcal{F})$. We apply Proposition \ref{hyperbolic} to the collection of paths $\psi(\gamma_n)$, which are all uniformly Hausdorff close to geodesic segments in $FZ(G,\mathcal{F})$ by Theorem \ref{FZ-hyperbolic}. We choose $R$ to be large enough compared to the diameter of the $\psi$-image of $\gamma$. This provides an integer $m_0\in\mathbb{N}$, and a sequence $(t_m)_{m\ge m_0}\in\mathbb{R}^{\mathbb{N}}$, so that $(\psi(\gamma_m(t_m)))_{m\ge m_0}$ lies in a region of diameter bounded by $M$ in $FZ(G,\mathcal{F})$, and for all $m\in\mathbb{N}$, the distance between $\psi(\gamma_m(t_m))$ and the $\psi$-image of $\gamma$ is at least $4C_M$. Up to passing to a subsequence, we can assume that $(t_m)_{m\in\mathbb{N}}$ converges to some $t_{\infty}\in\mathbb{R}\cup\{+\infty\}$, and hence $(\gamma_m(t_m))_{m\in\mathbb{N}}$ converges to $\gamma(t_{\infty})$. We have $\mathcal{R}^2(\gamma(t_{\infty}))\neq\emptyset$, and

\begin{itemize}
\item for all $m\in\mathbb{N}$, we have $d_{FZ(G,\mathcal{F})}(\psi(\gamma_m(t_m)),\mathcal{R}^2(\gamma(t_{\infty})))\le C_M$ (Proposition \ref{psi-bounded}), and
\item the diameter of $\mathcal{R}^2(\gamma(t_{\infty}))$ is at most $C_M$ (Theorem \ref{canonical-splittings}), and 
\item there exists $t_0\in\mathbb{R}$ so that $d_{FZ(G,\mathcal{F})}(\mathcal{R}^2(\gamma(t_{\infty})),\psi(\gamma(t_0)))\le C_M$ (Theorem \ref{canonical-splittings}).
\end{itemize}

\noindent Therefore, for all $m\in\mathbb{N}$, we have $d_{FZ(G,\mathcal{F})}(\psi(\gamma_m(t_m)),\psi(\gamma(t_0)))\le 3C_M$, which is a contradiction. 
\end{proof}

\begin{prop}\label{psi-surjective}
The map $\partial\psi:\mathcal{X}(G,\mathcal{F})/{\sim}\to\partial_{\infty} FZ(G,\mathcal{F})$ is surjective.
\end{prop}

\begin{proof}
Let $\xi\in\partial_{\infty}FZ(G,\mathcal{F})$, and let $(Z_n)_{n\in\mathbb{N}}\in FZ(G,\mathcal{F})^{\mathbb{N}}$ be a sequence that converges to $\xi$. For all $n\in\mathbb{N}$, let $T_n\in\mathcal{O}(G,\mathcal{F})$ be a simplicial tree such that $\psi(T_n)$ is at bounded distance from $Z_n$. Up to passing to a subsequence and rescaling, we can assume that $(T_n)_{n\in\mathbb{N}}$ converges non-projectively to some tree $T\in\overline{\mathcal{O}(G,\mathcal{F})}$. Proposition \ref{psi-not-infty} ensures that $T\in\mathcal{X}(G,\mathcal{F})$, and Proposition \ref{X-psi} ensures that $\xi=\partial\psi(T)$.  
\end{proof}

\paragraph{Closedness of $\partial\psi$.}

\begin{prop}\label{psi-closed}
The map $\partial\psi:{\mathcal{X}(G,\mathcal{F})}\to\partial_{\infty} FZ(G,\mathcal{F})$ is closed.
\end{prop}

\begin{proof}
Let $\xi\in\partial_{\infty} FZ(G,\mathcal{F})$, and let $(T_n)_{n\in\mathbb{N}}\in {\mathcal{X}(G,\mathcal{F})}^{\mathbb{N}}$ be such that $(\partial\psi(T_n))_{n\in\mathbb{N}}$ converges to $\xi$. We will show that for all limit points $T\in\overline{\mathcal{O}(G,\mathcal{F})}$ of the sequence $(T_n)_{n\in\mathbb{N}}$, we have $T\in \mathcal{X}(G,\mathcal{F})$ and $\partial\psi(T)=\xi$. For all $n\in\mathbb{N}$, we have $T_n\in \mathcal{X}(G,\mathcal{F})$, so it follows from Proposition \ref{X-psi} that there exists a sequence $(X_n)_{n\in\mathbb{N}}\in \mathcal{O}(G,\mathcal{F})^{\mathbb{N}}$ that converges to $T$, so that $(\psi(X_n))_{n\in\mathbb{N}}$ converges to $\xi$. Proposition \ref{psi-not-infty} ensures that $T\in{\mathcal{X}(G,\mathcal{F})}$, and Proposition \ref{X-psi} then ensures that $\partial\psi(T)=\xi$.
\end{proof}

\paragraph{End of the proof of the main theorem.}

\begin{proof}[Proof of Theorem \ref{main}]
The map $\partial\psi:\mathcal{X}(G,\mathcal{F})/{\sim}\to\partial_{\infty} FZ(G,\mathcal{F})$ is a continuous, bijective, closed map (Propositions \ref{psi-continuous}, \ref{psi-injective}, \ref{psi-surjective} and \ref{psi-closed}), and hence a homeomorphism. That $\partial\psi$ is $\text{Out}(G,\mathcal{F})$-equivariant follows from its construction.
\end{proof}

\appendix 
\section{Hyperbolicity results}

In this appendix, we sketch proofs of Theorems \ref{FS-hyperbolic} and \ref{FZ-hyperbolic}, by following very closely Bestvina--Feighn's \cite[Appendix]{BF13} and Mann's \cite{Man12} approaches.

\subsection{Hyperbolicity of $FS(G,\mathcal{F})$: proof of Theorem \ref{FS-hyperbolic}}

We follow very closely the exposition from \cite[Appendix]{BF13}. The first step of the proof consists in establishing several distance estimates in $FS(G,\mathcal{F})$ (Corollaries \ref{cor-descendant} and \ref{cor-hanging} below). 

A \emph{natural vertex} of a $(G,\mathcal{F})$-free splitting is a vertex which either has valence at least $3$, or is the center of an inversion. A \emph{natural edge} is a complementary component of the set of natural vertices. Assume that both $T$ and $T'$ have been equipped with a simplicial metric, and let $f:T\to T'$ be an optimal morphism. Let $R$ and $B$ be two disjoint $G$-invariant sets of points in $T'$, both disjoint from the set of natural vertices, which project to finite sets in the quotient graph $T/G$. A \emph{mixed region} in $T'$ is a component of the complement of $R\cup B$ in $T'$ whose frontier intersects both $R$ and $B$. Assuming in addition that $f^{-1}(R)$ and $f^{-1}(B)$ are disjoint from the set of natural vertices, we also define mixed regions in $T$.

\begin{prop}(Bestvina--Feighn \cite[Lemma A.4]{BF13})\label{descendant}
There exists a constant $C_1>0$ such that the following holds. 
\\
Let $T,T'\in FS(G,\mathcal{F})$. Assume that $T$ and $T'$ are equipped with simplicial metrics, and let $f:T\to T'$ be an optimal morphism. Let $R$ and $B$ be nonempty disjoint $G$-invariant sets in $T'$, both disjoint from the set of natural vertices in $T'$, whose projections to the quotient graph $T'/G$ are finite, and such that $f^{-1}(R)$ and $f^{-1}(B)$ are disjoint from the set of natural vertices in $T$. Let $N_0$ denote the number of $G$-orbits of mixed regions in $T$. Then $d_{FS(G,\mathcal{F})}(T,T')\le 2N_0+2\text{rk}_K(G,\mathcal{F})+4$.
\end{prop}

\begin{proof}
The proof goes exactly as in \cite[Lemma A.4]{BF13}, by using the Kurosh rank instead of the rank when defining the complexity. 
\end{proof}

By choosing $R$ and $B$ to consist of the $G$-orbits of two nearby points in the interior of an edge of $T'$, Proposition \ref{descendant} yields the following distance estimate in $FS(G,\mathcal{F})$.

\begin{cor}(Bestvina--Feighn \cite[Lemma A.3]{BF13})\label{cor-descendant}
Let $T,T'\in FS(G,\mathcal{F})$. Assume that $T$ and $T'$ are equipped with simplicial metrics, and let $f:T\to T'$ be an optimal morphism. Let $y\in T'$ be a point that belongs to the interior of an edge, such that $f^{-1}(y)$ does not intersect the set of natural vertices in $T$. Then $d_{FS(G,\mathcal{F})}(T,T')$ is bounded by a linear function of the cardinality of $f^{-1}(y)$.
\end{cor}

Let $T$ be a $(G,\mathcal{F})$-free splitting, equipped with a train track structure. A \emph{hanging tree} in $T$ is a triple $(H,l,v)$, where 
\begin{itemize}
\item there exists an element $g_0\in G$ that is hyperbolic in $T$, whose axis in $T$ is equal to $l$ and is legal, and 
\item the subtree $H\subseteq T$ is a finite (not necessarily closed) subtree of $T$, such that 
\begin{itemize}
\item we have $gH\cap H=\emptyset$ for all $g\in G\smallsetminus\{e\}$, and
\item the intersection $H\cap l$ is a finite segment of the form $[v,v']$ (possibly reduced to a point), and 
\item the vertex $v$ has exactly two gates in $T$, and $v$ has exactly one gate in $H$, and 
\item every other vertex in $H$ has two gates in $T$, with the direction towards $v$ being its own gate.
\end{itemize}
\end{itemize}

\noindent We call $v$ the \emph{top vertex} of the hanging tree. The following proposition is a restatement in our context of \cite[Lemma A.7]{BF13}. Bestvina--Feighn's lemma is stated in terms of graphs in Culler--Vogtmann's outer space. As trees in the relative outer space $\mathcal{O}(G,\mathcal{F})$ are not locally compact, it is more adapted to state the result in terms of trees (and not in terms of their quotient graphs) in our setting. We note for the statement that any simplicial edge $e$ in a $(G,\mathcal{F})$-tree defines a $(G,\mathcal{F})$-splitting by collapsing all edges that do not belong to the $G$-orbit of $e$ to points. 

\begin{prop}(Bestvina--Feighn \cite[Lemma A.7]{BF13})\label{hanging}
There exists $C_2>0$ such that for all $T,T'\in FS(G,\mathcal{F})$, equipped with simplicial metrics, and all optimal morphisms $f:T\to T'$, if there exist edges $e\subset T$ and $e'\subset T'$ that define the same $(G,\mathcal{F})$-free splitting, then either
\begin{itemize}
\item there exists a point in the interior of $e'$ whose $f$-preimage has cardinality at most $C_2$, or
\item there is a hanging tree $(H,l,v)$ in $T$ (for the train track structure determined by $f$), such that the $f$-preimage of any interior point of $e'$ is contained in the union of $l$ and of the $\text{Stab}(l)$-translates of $H$, and contains at most one point in $l$. 
\end{itemize}
\end{prop}

\begin{proof}
We borrow the proof from Bestvina--Feighn's paper; the only novelty here is that arguments are phrased at the level of trees, while they are given in the quotient graphs in Bestvina--Feighn's paper.

We assume that the edges of $T$ and $T'$ have been $G$-equivariantly subdivided, so that $f$ maps edgelets to edgelets. Let $(T_t)_{t\in [0,L]}$ be a liberal folding path guided by $f$, chosen so that on an initial segment $[0,s]$, we perform all folds which do not involve the image of $e$, and the edge $e_s:=f_{0,s}(e)$ is involved in all illegal turns in $T_s$. We denote by $\widehat{e_s}$ the natural edge in $T_s$ that contains $e_s$. If $T_s=T'$, then $\widehat{e_s}=e'$, and the first conclusion of the proposition holds, so we assume otherwise. Note that since all folds in $T_s$ involve $e_s$, and $f$ is optimal, at least one endpoint of $e_s$ is a natural vertex (but the other might not be natural). 

Let $G'$ be one of the elliptic groups of the $(G,\mathcal{F})$-free splitting determined by $\widehat{e_s}$. Let $\widetilde{T_s}$ (resp. $\widetilde{T'}$) denote the $G'$-tree which is one of the components of $T_s\smallsetminus G.\widehat{e_s}$ (resp. $T'\smallsetminus G.e'$). By construction, the splitting determined by $\widehat{e_s}$ in $T_s$ is the same as the splitting determined by $e$ in $T$, and it is also the same as the splitting determined by $e'$ in $T'$ by assumption.
\\
\\
\textit{Case 1}: The tree $\widetilde{T_s}$ is a minimal $G'$-tree.
\\ In this case, we will show that the first conclusion of Proposition \ref{hanging} holds. The morphism $f_{s,L}$ restricts to an isometric embedding from $\widetilde{T_s}$ to the minimal $G'$-subtree of $T'$, which is contained in $\widetilde{T'}$ because $e$ and $e'$ define the same $(G,\mathcal{F})$-splitting. In other words, the image $f_{s,L}(\widetilde{T_s})$ does not meet the interior of $e'$. 

Let $\widehat{T_s}$ (resp. $\widehat{T'}$) be the tree obtained by collapsing all edges in $T_s\smallsetminus G.\widehat{e_s}$ (resp. $T'\smallsetminus G.e'$) to points in $T_s$ (resp. $T'$). The optimal morphism $f_{s,L}$ induces a $G$-equivariant simplicial map $\widehat{f}_{s,L}:\widehat{T_s}\to\widehat{T'}$ which preserves alignment in restriction to each natural edge of $\widehat{T_s}$. The trees $\widehat{T_s}$ and $\widehat{T'}$ are $G$-equivariantly homeomorphic by construction, so $\widehat{f}_{s,L}$ can be viewed as a map from $\widehat{T_s}$ to itself. As vertices of $\widehat{T_s}$ have nontrivial stablizer, the map $\widehat{f}_{s,L}$ maps vertices to themselves. Optimality of $f_{s,L}$ then implies that $\widehat{f}_{s,L}$ also maps edges to themselves. This implies that $\widehat{f}_{s,L}(\widehat{e}_s)=ge'$ for some $g\in G$, so $f_{s,L}(\widehat{e_s})$ crosses the $G$-orbit of $e'$ exactly once. Therefore, the $f_{s,L}$-preimage of any point $y$ in the interior of $e'$ is a single point $y_s$ in some $G$-translate of $\widehat{e_s}$. As the $f_{0,s}$-image of any edge in $T\smallsetminus G.e$ is disjoint from $e_s$, we can choose $y$ in the interior of $e'$ so that the $f_{0,s}$-preimage of $y_s$ has cardinality $1$, and the first conclusion of the proposition follows.
\\
\\
\textit{Case 2}: The tree $\widetilde{T_s}$ is not minimal.\\
This implies that the splitting of $G$ defined by $e$ (and $\widehat{e_s}$) is an HNN extension, and the edge $\widehat{e_s}$ projects in the quotient graph $T_s/G$ to a loop-edge, whose extremal vertex lifts to a valence $3$ vertex in $T_s$ (with trivial stabilizer). In this case, the vertex we are considering has finite valence, and the proof then runs exactly as in \cite[Lemma A.7]{BF13}.
\end{proof}

As in \cite[Proposition A.9]{BF13}, we get as a corollary the following distance estimate in $FS(G,\mathcal{F})$.

\begin{cor}(Bestvina--Feighn \cite[Proposition A.9]{BF13})\label{cor-hanging}
There exists $C_3>0$ so that for all $T,T'\in FS(G,\mathcal{F})$, and all optimal morphisms $f:T\to T'$, if $T$ and $T'$ contain edges which determine the same $(G,\mathcal{F})$-free splitting, then the image in $FS(G,\mathcal{F})$ of any liberal folding path guided by $f$ has diameter at most $C_3$. 
\end{cor}

Once the above distance estimates have been established, the proof of Theorem \ref{FS-hyperbolic} goes by checking Masur and Minsky's axioms \cite[Theorem 2.3]{MM99} for the set of $\phi$-images in $FS(G,\mathcal{F})$ of optimal liberal folding paths between simplicial trees with trivial edge stabilizers in $\overline{\mathcal{O}(G,\mathcal{F})}$, which is coarsely transitive (the existence of optimal morphisms between splittings in $\mathcal{O}(G,\mathcal{F})$ follows from \cite[Corollary 6.8]{FM13}, and there is a canonical way to build a folding path from a morphism, as recalled in Section \ref{sec-folding}).

\subsection{Hyperbolicity of $FZ^{(max)}(G,\mathcal{F})$: proof of Theorem \ref{FZ-hyperbolic}}

We denote by $FZ'(G,\mathcal{F})$ (resp. $(FZ^{max})'(G,\mathcal{F})$) the graph whose vertices are one-edge $(G,\mathcal{F})$-free splittings, two splittings being joined by an edge if they are both compatible with a common $\mathcal{Z}$-splitting (resp. $\mathcal{Z}^{max}$-splitting). Since every one-edge $\mathcal{Z}$-splitting is compatible with a one-edge $(G,\mathcal{F})$-free splitting \cite[Lemma 5.11]{Hor14-5}, the graphs $FZ'(G,\mathcal{F})$ and $FZ(G,\mathcal{F})$ are quasi-isometric to each other, and similarly $(FZ^{max})'(G,\mathcal{F})$ and $FZ^{max}(G,\mathcal{F})$ are quasi-isometric. Following Mann's proof \cite{Man12}, we will show hyperbolicity of $FZ'(G,\mathcal{F})$ and $(FZ^{max})'(G,\mathcal{F})$. This will follow from the hyperbolicity of $FS(G,\mathcal{F})$ by applying a criterion due to Kapovich and Rafi, which we now recall, to the natural inclusion maps from $FS(G,\mathcal{F})$ to these graphs. 

\begin{prop}(Kapovich--Rafi \cite[Proposition 2.5]{KR12})\label{Kapovich-Rafi}
For all $\delta_0,M>0$, there exist $\delta_1,H>0$ such that the following holds.
\\
Let $X$ and $Y$ be connected graphs, such that $X$ is $\delta_0$-hyperbolic. Let $f:X\to Y$ be a map sending $V(X)$ onto $V(Y)$, and sending edges to edges. Assume that for all $x,y\in V(X)$, if $d_Y(f(x),f(y))\le 1$, then the $f$-image of any geodesic segment joining $x$ to $y$ in $X$ has diameter bounded by $M$ in $Y$. Then $Y$ is $\delta_1$-hyperbolic, and for any $x,y\in V(X)$, the $f$-image of any geodesic segment joining $x$ to $y$ in $X$ is $H$-Hausdorff close to any geodesic segment joining $f(x)$ to $f(y)$ in $Y$.
\end{prop}

\begin{proof}[Proof of Theorem \ref{FZ-hyperbolic}]
The arguments from the following proof are due to Mann \cite{Man12}. Let $T_1$ and $T_2$ be two one-edge $(G,\mathcal{F})$-free splittings, both compatible with a one-edge $\mathcal{Z}^{(max)}$-splitting $T$. For simplicity of notations, we consider the case where the quotient graphs $T_1/G$, $T_2/G$ and $T/G$ are segments. The case of loop-edges is left to the reader, as the argument is similar, and similar to that in the proof of \cite[Theorem 5]{Man12}. Then $T$ is of the form $A\ast_{\langle w\rangle} B$. Without loss of generality, we can assume that there exist two free splittings of $B$ of the form $B=B_1\ast B'_1$ and $B=B_2\ast B'_2$, such that for all $i\in\{1,2\}$, the splitting $T+T_i$ is of the form $A\ast_{\langle w\rangle}B_i\ast B'_i$. Indeed, otherwise, the trees $T_1$ and $T_2$ are compatible, in which case they are already at distance $1$ in $FS(G,\mathcal{F})$. 

By blowing up the vertex groups of the splitting $T+T_1$, using their action on $T+T_2$ (which is possible because $T+T_1$ and $T+T_2$ have the same edge stabilizers), we get a tree $\widehat{T_1}$ that collapses to $T+T_1$, and comes with a morphism $f:\widehat{T}_1\to T+T_2$. We denote by $\widehat{p_1}:\widehat{T_1}\to T$ and $p_2:T+T_2\to T$ the natural alignement-preserving maps. The $B$-minimal subtree of $\widehat{T_1}$ is mapped by $f$ to the $B$-minimal subtree of $T+T_2$, so $\widehat{p_1}=p_2\circ f$. Using Lemma \ref{folding-collapse}, we see that all trees $T_t$ on an optimal liberal folding path guided by $f$ collapse to $T$. By equivariantly collapsing the edge with stabilizer $\langle w\rangle$ to a point in $T_t$, we get an optimal liberal folding path $\gamma$ from $T_1$ to $T_2$, whose $\psi$-image stays at bounded distance from $T$ in $FZ^{(max)}(G,\mathcal{F})$.  

We already know that $FS(G,\mathcal{F})$ is Gromov hyperbolic, and $\phi$-images of optimal liberal folding paths between simplicial trees with trivial edge groups are reparameterized quasi-geodesics with uniform constants. Therefore, any geodesic from $T_1$ to $T_2$ in $FS(G,\mathcal{F})$ is uniformly close to the folding path $\gamma$, with constants depending only on the hyperbolicity constant of $FS(G,\mathcal{F})$. Hence there is a constant $M$ such that the diameter of the $f$-image of any geodesic segment joining $T_1$ to $T_2$ in $FS(G,\mathcal{F})$ is bounded by $M$ in $FZ'(G,\mathcal{F})$. By choosing for $T$ a $\mathcal{Z}^{max}$-splitting, the same holds true for $(FZ^{max})'(G,\mathcal{F})$.
\end{proof}

\bibliographystyle{amsplain}
\bibliography{bdyZ}

\providecommand{\bysame}{\leavevmode\hbox to3em{\hrulefill}\thinspace}
\providecommand{\MR}{\relax\ifhmode\unskip\space\fi MR }
\providecommand{\MRhref}[2]{%
  \href{http://www.ams.org/mathscinet-getitem?mr=#1}{#2}
}
\providecommand{\href}[2]{#2}
\begin{thebibliography}{10}

\bibitem{AK12}
Y.~Algom-Kfir, \emph{The {M}etric {C}ompletion of {O}uter {S}pace},
  arXiv:1202.6392v4 (2013).

\bibitem{BF08}
G.C. Bell and K.~Fujiwara, \emph{The asymptotic dimension of a curve graph is
  finite}, J. Lond. Math. Soc. \textbf{77} (2008), no.~1, 33--50.

\bibitem{BBF14}
M.~Bestvina, K.~Bromberg, and K.~Fujiwara, \emph{Constructing group actions on
  quasi-trees and applications to mapping class groups}, arXiv:1006.1939v5
  (2014).

\bibitem{BF95}
M.~Bestvina and M.~Feighn, \emph{Stable actions of groups on real trees},
  Invent. math. \textbf{121} (1995), no.~1, 287--321.

\bibitem{BF12}
\bysame, \emph{Hyperbolicity of the complex of free factors}, Adv. Math.
  \textbf{256} (2014), 104--155.

\bibitem{BF13}
\bysame, \emph{Subfactor projections}, J. Topology \textbf{7} (2014), no.~3,
  771--804.

\bibitem{BFH97}
M.~Bestvina, M.~Feighn, and M.~Handel, \emph{Laminations, trees, and
  irreducible automorphisms of free groups}, Geom. Funct. Anal. \textbf{7}
  (1997), no.~2, 215--244.

\bibitem{BF02}
M.~Bestvina and K.~Fujiwara, \emph{Bounded cohomology of subgroups of mapping
  class groups}, Geom. Topol. \textbf{6} (2002), no.~1, 69--89.

\bibitem{BR13}
M.~Bestvina and P.~Reynolds, \emph{The boundary of the complex of free
  factors}, arXiv:1211.3608v2 (2013).

\bibitem{BH99}
M.R. Bridson and A.~Haefliger, \emph{Metric spaces of non-positive curvature},
  Springer-Verlag, Berlin, 1999.

\bibitem{BCM12}
J.F. Brock, R.D. Canary, and Y.N. Minsky, \emph{The classification of
  {K}leinian surface groups, {II}: {T}he {E}nding {L}amination {C}onjecture},
  Ann. Math. \textbf{176} (2012), no.~1, 1--149.

\bibitem{CDP90}
M.~Coornaert, T.~Delzant, and A.~Papadopoulos, \emph{Géométrie et théorie
  des groupes: {L}es groupes hyperboliques de {G}romov}, Lecture Notes in
  Mathematics, vol. 1441, Springer, 1990.

\bibitem{CM87}
M.~Culler and J.W. Morgan, \emph{Group actions on $\mathbb{R}$-trees}, Proc.
  London Math. Soc. \textbf{55} (1987), no.~3, 571--604.

\bibitem{FM02}
B.~Farb and L.~Mosher, \emph{Convex cocompact subgroups of mapping class
  groups}, Geom. Topol. \textbf{91-152} (2002), no.~1, 91--152.

\bibitem{FM13}
S.~Francaviglia and A.~Martino, \emph{Stretching factors, metrics and train
  tracks for free products}, arXiv:1312.4172v2 (2014).

\bibitem{Gab09}
D.~Gabai, \emph{Almost filling laminations and the connectivity of ending
  lamination space}, Geom. Topol. \textbf{13} (2009), no.~2, 1017--1041.

\bibitem{GdlH90}
E.~Ghys and P.~de~la Harpe, \emph{Sur les groupes hyperboliques d'après
  {M}ikhael {G}romov}, Birkhäuser, 1990.

\bibitem{Gro87}
M.~Gromov, \emph{Hyperbolic groups}, Essays in {G}roup {T}heory (S.M. Gersten,
  ed.), MSRI Publications, vol.~8, Springer-Verlag, 1987, pp.~75--265.

\bibitem{Gui98}
V.~Guirardel, \emph{Approximations of stable actions on $\mathbb{R}$-trees},
  Comment. Math. Helv. \textbf{73} (1998), no.~1, 89--121.

\bibitem{Gui04}
\bysame, \emph{Limit groups and groups acting freely on $\mathbb{R}^n$-trees},
  Geom. Topol. \textbf{8} (2004), 1427--1470.

\bibitem{Gui08}
\bysame, \emph{Actions of finitely generated groups on $\mathbb{R}$-trees},
  Ann. Inst. Fourier \textbf{58} (2008), no.~1, 159--211.

\bibitem{GL14-2}
V.~Guirardel and G.~Levitt, in preparation.

\bibitem{GL07-2}
\bysame, \emph{Deformation spaces of trees}, Groups Geom. Dyn. \textbf{1}
  (2007), no.~2, 135--181.

\bibitem{GL07}
\bysame, \emph{The outer space of a free product}, Proc. London Math. Soc.
  \textbf{94} (2007), no.~3, 695--714.

\bibitem{GL10-2}
\bysame, \emph{{JSJ} decompositions : definitions, existence, uniqueness. {II}
  : {C}ompatibility and acylindricity.}, ar{X}iv:1002.4564 (2010).

\bibitem{GL14}
\bysame, \emph{Splittings and automorphisms of relatively hyperbolic groups},
  to appear in Groups. Geom. Dyn. (2014).

\bibitem{Ham06}
U.~Hamenstädt, \emph{Train tracks and the {G}romov boundary of the complex of
  curves}, London Math. Soc. Lec. Notes, vol. 329, pp.~187--207, Cambridge
  University Press, 2006.

\bibitem{Ham09}
\bysame, \emph{Geometry of the mapping class group {I}: {B}oundary
  amenability}, Invent. Math. \textbf{175} (2009), no.~3, 545--609.

\bibitem{Ham12}
\bysame, \emph{The boundary of the free splitting graph and of the free factor
  graph}, arXiv:1211.1630v5 (2014).

\bibitem{HH14}
U.~Hamenstädt and S.~Hensel, \emph{Spheres and {P}rojections for
  $\text{{O}ut}({F}_n)$}, arXiv:1109.2687v3 (2014).

\bibitem{HM12}
M.~Handel and L.~Mosher, \emph{The free splitting complex of a free group {I}:
  {H}yperbolicity}, Geom. Topol. \textbf{17} (2013), no.~3, 1581--1670.

\bibitem{HH13}
A.~Hilion and C.~Horbez, \emph{The hyperbolicity of the sphere complex via
  surgery paths}, to appear in J. reine angew. Math.

\bibitem{Hor14-5}
C.~Horbez, \emph{The boundary of the outer space of a free product},
  arXiv:1408.0543v2 (2014).

\bibitem{Hor14-1}
\bysame, \emph{Spectral rigidity for primitive elements of ${F}_{N}$},
  arXiv:1405.4624v1 (2014).

\bibitem{Hor14-8}
\bysame, \emph{The {T}its alternative for the automorphism group of a free
  product}, arXiv:1408.0546v2 (2014).

\bibitem{KR12}
I.~Kapovich and K.~Rafi, \emph{On hyperbolicity of free splitting and free
  factor complexes}, Groups Geom. Dyn. \textbf{8} (2014), no.~2, 391--414.

\bibitem{Kid08}
Y.~Kida, \emph{The mapping class group from the viewpoint of measure
  equivalence theory}, Mem. Amer. Math. Soc. \textbf{196} (2008), no.~916.

\bibitem{Kla99}
E.~Klarreich, \emph{The {B}oundary at {I}nfinity of the {C}urve {C}omplex and
  the {R}elative {T}eichmüller {S}pace}, preprint (1999).

\bibitem{Kob88}
T.~Kobayashi, \emph{Heights of simple loops and pseudo-{A}nosov
  homeomorphisms}, Contemp. Math. \textbf{78} (1988), 327--338.

\bibitem{Kur34}
A.G. Kurosh, \emph{Die {U}ntergruppen der freien {P}rodukte von beliebigen
  {G}ruppen}, Math. Ann. \textbf{109} (1934), no.~1, 647--660.

\bibitem{Lac69}
R.C. Lacher, \emph{Cell-like mappings, {I}}, Pacific J. Math. \textbf{30}
  (1969), no.~3, 717--731.

\bibitem{LS09}
C.J. Leininger and S.~Schleimer, \emph{Connectivity of the space of ending
  laminations}, Duke Math. J. \textbf{150} (2009), no.~3, 533--575.

\bibitem{Lev94}
G.~Levitt, \emph{Graphs of actions on $\mathbb{R}$-trees}, Comment. Math. Helv.
  \textbf{69} (1994), no.~1, 28--38.

\bibitem{Man12}
B.~Mann, \emph{Hyperbolicity of the cyclic splitting complex}, Geom. Dedicata
  (2013), doi 10.1007/s10711-013-9941-3.

\bibitem{MM99}
H.A. Masur and Y.N. Minsky, \emph{Geometry of the complex of curves {I} :
  {H}yperbolicity}, Invent. math. \textbf{138} (1999), no.~1, 103--149.

\bibitem{MS13}
H.A. Masur and S.~Schleimer, \emph{The geometry of the disk complex}, J. Amer.
  Math. Soc. \textbf{26} (2013), no.~1, 1--62.

\bibitem{Min10}
Y.N. Minsky, \emph{The classification of {K}leinian surface groups {I}: models
  and bounds}, Ann. Math. \textbf{171} (2010), no.~1, 1--107.

\bibitem{Mor88}
J.W. Morgan, \emph{Ergodic theory and free actions of groups on
  $\mathbb{R}$-trees}, Invent. math. \textbf{94} (1988), no.~3, 605--622.

\bibitem{Pau88}
F.~Paulin, \emph{Topologie de {G}romov équivariante, structures hyperboliques
  et arbres réels}, Invent. math. \textbf{94} (1988), no.~1, 53--80.

\bibitem{RS11}
K.~Rafi and S.~Schleimer, \emph{Curve complexes are rigid}, Duke Math. J.
  \textbf{158} (2011), no.~2, 225--246.

\bibitem{Rey11}
P.~Reynolds, \emph{Dynamics of irreducible endomorphisms of ${F}_n$},
  arXiv:1008.3659v3 (2011).

\bibitem{Rey12}
\bysame, \emph{Reducing systems for very small trees}, arXiv:1211.3378v1
  (2012).

\bibitem{SW79}
P.~Scott and T.~Wall, \emph{Topological methods in group theory}, Homological
  {G}roup {T}heory (C.T.C. Wall, ed.), LMS Lecture Notes Series, vol.~36, 1979,
  pp.~137--203.

\bibitem{Wal81}
J.J. Walsh, \emph{Dimension, cohomological dimension, and cell-like mappings},
  Shape Theory and Geometric Topology, Lecture Notes in Math., vol. 870,
  Springer-Verlag, 1981, pp.~105--118.

\end{thebibliography}

\end{document}